%% file: main.tex
\documentclass[twoside]{article}

\usepackage[margin=1in]{geometry}

\title{A Bayesian Approach to Estimating Background Flows from a Passive Scalar}
\author{Jeff Borggaard, Nathan E. Glatt-Holtz, Justin A. Krometis\\
  \scriptsize{emails: jborggaard@vt.edu, negh@tulane.edu, jkrometis@vt.edu}}
\date{}

\usepackage{amsfonts, amssymb, amsmath, amsthm, multicol,bbm}
\usepackage[colorlinks=true, pdfstartview=FitV, linkcolor=blue,
            citecolor=blue, urlcolor=blue]{hyperref}
\usepackage[usenames]{color}
\definecolor{Red}{rgb}{0.7,0,0.1}
\definecolor{Green}{rgb}{0,0.7,0}
\usepackage{accents}
\usepackage{comment,url}

\numberwithin{equation}{section}

\usepackage{enumitem}
\usepackage[colorinlistoftodos]{todonotes}
\usepackage[capitalize,nameinlink,noabbrev]{cleveref} 
\usepackage[section]{algorithm}      
\usepackage{algpseudocode}    
\usepackage{wrapfig}                            
\usepackage{tabularx}
\usepackage{bigints}

\usepackage{tikz}
\usetikzlibrary{shapes,arrows,positioning}

\input{hdr.tex}

\newcommand{\x}{\mathbf{x}}
\newcommand{\e}{\mathbf{e}}
\newcommand{\kbf}{\mathbf{k}}
\newcommand{\ek}{\e_{\kbf}}
\newcommand{\spatdom}{\mathbb{T}^2}

\newcommand{\conductivity}{\kappa}
\newcommand{\pdesol}{\theta}
\newcommand{\pdesolvec}{\vec{\pdesol}}
\newcommand{\vfieldnobf}{v}
\newcommand{\vfield}{\mathbf{\vfieldnobf}}
\newcommand{\vk}{\vfieldnobf_{\kbf}}
\newcommand{\vtrue}{\vfield^{\star}}
\newcommand{\vfspace}{H}

\newcommand{\G}{\mathcal{G}}
\newcommand{\Obs}{\mathcal{O}}
\newcommand{\Sol}{\mathcal{S}}
\newcommand{\Q}{\mathbb{Q}}
\newcommand{\data}{\mathcal{Y}}

\newcommand{\dataspace}{Y}
\newcommand{\noise}{\eta}

\newcommand{\noisemeasure}{\gamma_0}
\newcommand{\covar}{\mathcal{C}}

\newcommand{\mcmcsamp}{\vfield}         
\newcommand{\mcmccand}{\tilde{\vfield}} 

\newcommand{\indFn}[1]{1 \! \! 1_{#1}}
\newcommand{\E}{\mathbb{E}}
\newcommand{\Prb}{\mathbb{P}}
\newcommand{\Fd}{\mathcal{F}}
\newcommand{\RR}{\mathbb{R}}

\newcommand{\QQ}{\mathbb{Q}}

\newcommand{\urv}{\mathbf{V}}
\newcommand{\usm}{\mathbf{v}}
\newcommand{\usp}{H}
\newcommand{\yrv}{\mathcal{Y}}
\newcommand{\ysm}{y}
\newcommand{\ysp}{\RR^N}
\newcommand{\nrv}{\eta}

\newcommand{\nsp}{\RR^M}
\newcommand{\rcd}{\gimel}

\newtheorem{Thm}{Theorem}[section]
\newtheorem{Lem}[Thm]{Lemma}
\newtheorem{Prop}[Thm]{Proposition}

\newtheorem{Ex}[Thm]{Example}
\newtheorem{Def}[Thm]{Definition}
\newtheorem{Rmk}[Thm]{Remark}

\newtheorem{Assum}[Thm]{Assumption}
 
\newtheorem{Not}[Thm]{Notation}

\begin{document}

\markboth{J. Borggaard, N. Glatt-Holtz, J. Krometis}
{A Bayesian Approach to Estimating Background Flows from a Passive Scalar}

\maketitle

\begin{abstract}
We consider the statistical inverse problem of estimating a background flow field (e.g., of air or water) from the partial and noisy observation of a passive scalar (e.g., the concentration of a solute), a common experimental approach to visualizing complex fluid flows.  Here the unknown is a vector field that is specified by a large or infinite number of degrees of freedom. Since the inverse problem is ill-posed, i.e., there may be many or no background flows that match a given set of observations, we adopt a Bayesian approach to regularize it. In doing so, we leverage frameworks developed in recent years for infinite-dimensional Bayesian inference. The contributions in this work are threefold. First, we lay out a functional analytic and Bayesian framework for approaching this problem. Second, we define an adjoint method for efficient computation of the gradient of the log likelihood, a key ingredient in many numerical methods. Finally, we identify interesting example problems that exhibit posterior measures with simple and complex structure. We use these examples to conduct a large-scale benchmark of Markov Chain Monte Carlo methods developed in recent years for infinite-dimensional settings. Our results indicate that these methods are capable of resolving complex multimodal posteriors in high dimensions.
\end{abstract}

{\noindent \small {\it \bf Keywords:} Bayesian Statistical Inversion,
  Markov Chain Monte Carlo (MCMC), Passive
  Scalars, Fluid Turbulence}

\setcounter{tocdepth}{1}
\tableofcontents

\newpage

\section{Introduction}
\label{sec:introduction}\label{intro}
A common approach to investigating complex fluid flows is through
measurement of a substance moving within the fluid. For example, dye,
smoke, or neutrally-buoyant particles are injected into fluids to
visualize vortices or other structures in turbulent flows
\cite{karch2012dye,kellay2002two,merzkirch1987flow,smits2012flow}. In
this work we consider the inverse problem of estimating a background
fluid flow from partial, noisy observations of a dye, pollutant, or
other solute advecting and diffusing within the fluid. The initial
condition is assumed to be known, so the problem can be interpreted as
a controlled experiment, where a substance is added at known locations
and then observed as the system evolves to investigate the structure
of the underlying flow.

The physical model considered is the two-dimensional
advection-diffusion equation on the periodic domain
$\spatdom=[0,1]^2$:
\begin{equation} \label{eq:adr}
  \frac{\partial}{\partial t}{\pdesol}(t,\x) = -\vfield(\x) \cdot \nabla \pdesol(t,\x) + \conductivity \Delta \pdesol(t,\x) 
  \quad \text{ , } \quad \pdesol(0,\x) = \pdesol_0(\x).
\end{equation}
Here 
\begin{itemize}
\item $\pdesol: \R^{+} \times \spatdom \to \R$ is a \emph{passive
    scalar}, typically the concentration of some solute of interest,
  which is spread by diffusion and the motion of a
  (time-stationary) fluid flow $\vfield$. This solute is ``passive''
  in that it does not affect the motion of the underlying fluid.
\item $\vfield: \spatdom \to \R^2$ is an \emph{incompressible
    background flow}, i.e., $\vfield$ is constant in time and
  satisfies $\nabla \cdot \vfield = 0$.
\item $\conductivity>0$ is the \emph{diffusion coefficient}, which
  models the rate at which local concentrations of the solute spread
  out within the solvent in the absence of advection.
\end{itemize}

We obtain finite observations $\data \in \dataspace$ (e.g., $\R^N$ or
$\C^N$) subject to additive noise $\noise$, i.e.
\begin{equation} \label{eq:Gmap}
	\data = \G(\vfield)+\noise
  \quad \text{ , } \quad \noise \sim \noisemeasure
\end{equation}
for some measure $\noisemeasure$ related to the precision of the
observations. Here, $\G: \vfspace \to \dataspace$ is the
\emph{parameter-to-observable}, or forward, map. This $\G$ associates the
background flow $\vfield$, sitting in a suitable function space
$\vfspace$, with a finite collection of measurements (observables) of
the resulting $\pdesol=\pdesol(\vfield)$. The observations may take a
number of forms, such as:
\begin{itemize}
\item Spatial-temporal point observations:
  $\G_{j}(\vfield) = \pdesol(t_j, \x_j,\vfield)$ for $t_j \in [0,T]$ and
  $\x_j \in [0,1]^2$.
\item Spectral components:
  $\G_{j}(\vfield) =
  \ipZ{\pdesol(t_j,\cdot,\vfield)}{\e_{\kbf_j}}{L^2(\spatdom)}$
  for some basis $\{\e_{\kbf}\}$ of the scalar field $\pdesol$.
\item Local averages:
  $\G_j(\vfield) = \frac{1}{|\mathcal{D}_j|} \int_{\mathcal{D}_j}
  \pdesol(t,\x,\vfield) d\x\,dt$,
  for sub-domains $\mathcal{D}_j \subset [0,T] \times [0,1]^2$
  where $|\mathcal{D}_j|$ denotes the volume of $\mathcal{D}_j$.
\item Other physical quantities of interest from $\pdesol$, such as
  variance, dissipation rate, or structure functions.
\end{itemize}

This work will focus on point observations as the most obvious
practical implementation. However, we note that the methodology
outlined in this manuscript is quite general. Moreover, while we have assumed a
divergence-free flow, point observations, and periodic boundary
conditions, the framework herein could be adapted to other assumptions
via a different definition of the forward map $\G$.
 
As we illustrate below, the proposed inverse problem is ill-posed,
i.e., there may be many or no background flows $\vfield$ that match a
given dataset $\data$. To address this issue, we adopt a Bayesian
approach, incorporating prior knowledge of background flows and descriptions
of the observation error to develop probabilistic estimates of
$\vfield$. Summaries of the Bayesian approach to inverse problems can
be found in \cite{gelman2014bayesian} and
\cite{kaipio2005statistical}. Moreover, since the target of the
inversion, the background flow $\vfield$, is infinite-dimensional,
this work will leverage the considerable amount of recent research in
infinite-dimensional Bayesian inference, grounding much of our
approach in the overview of the field provided in
\cite{dashti2017bayesian}. 
To compute observables, such as the mean, variance, or
(normalized) histogram of a given quantity on $\vfield$ or $\pdesol$,
we use recently-developed Metropolis-Hastings Markov Chain Monte Carlo
(MCMC) algorithms that are well-defined in infinite dimensions. We focus on preconditioned Crank-Nicolson (pCN) \cite{cotter2013mcmc} and Hamiltonian (or Hybrid) Monte Carlo (HMC) \cite{beskos2011hybrid} samplers; some results for the independence sampler and Metropolis-adjusted Langevin (MALA) methods (see descriptions in, e.g., \cite{dashti2017bayesian} and \cite{beskos2017geometric}, respectively) are provided in the appendices.

This work makes a number of important contributions. We lay out a
Bayesian framework for the estimation of divergence-free background
flows from observations of scalar behavior, a common experimental
approach to investigating complex fluid flows
\cite{karch2012dye,kellay2002two,merzkirch1987flow,smits2012flow}.
We define, prove, and numerically implement an adjoint
method for the efficient computation of the gradient of the log likelihood, a key ingredient in higher-order MCMC methods. Finally, we identify two interesting examples for which the
resulting posterior measures have very different structures - one
fairly simple and one highly multi-modal. 
We use these two examples to conduct a
systematic, large-scale numerical study to benchmark the convergence
of the MCMC methods mentioned above for ``easy'' and ``hard''
problems. This is a companion paper to \cite{borggaard2018consistency}, where we investigate the behavior of the posterior measure as the number of point observations grows large (see also \cite{krometis2018bayesian}), and \cite{borggaard2019particle}, where we identify a computationally-efficient approach to computing the forward map.

The structure of the paper is as follows. \Cref{sec:math_framework}
defines our parametrization of the space of divergence-free flows, describes why the inverse problem is
ill-posed in the traditional sense, and presents the Bayesian approach to the inverse problem. \Cref{sec:numericalmethods} describes the
numerical approach to computing the posterior measure:
MCMC methods for sampling from the posterior, numerical
methods for solving the advection-diffusion equation \eqref{eq:adr}, and an adjoint method for computing the
gradients required for some MCMC methods. \Cref{sec:results} provides 
results of the inference and convergence of MCMC methods as
applied to two example problems. 
For completeness, appendices provide additional numerical results (\cref{sec:results_is_mala} and \cref{sec:results_obs}) and a description of Bayesian inference in a very general setting (\cref{sec:bayes_general}).

\subsection{Literature Review} \label{sec:lit_review}
\paragraph{Bayesian Inference and MCMC}

Comprehensive overviews of modern Bayesian
techniques, from the basics of probability theory to computational
practicalities, can be found in \cite{gelman2014bayesian} and
\cite{kaipio2005statistical}. 
In the last ten years, much attention has been paid to development of the theory of Bayesian inference for infinite-dimensional problems (e.g., where the target of the inversion is a function). These advances are summarized in \cite{dashti2017bayesian}, building upon the work in \cite{stuart2010inverse}; we follow the former closely in \cref{sec:math_framework} and provide a somewhat more general derivation of Bayes' Theorem in \cref{sec:bayes_general}. 

Similarly, while Metropolis-Hastings Markov Chain Monte Carlo (MCMC) methods date to
the foundational works \cite{metropolis1953equation} and \cite{hastings1970monte}, substantial recent work has gone into extending these methods to
problems where the space to be sampled is high- or
infinite-dimensional
\cite{beskos2009optimal,beskos2008mcmc,beskos2009mcmc}. The goal of
these efforts has been to define sampling kernels that are both well-defined
and yield robust convergence even as the number of dimensions to be
sampled grows large. The extension of Metropolis-Hastings methods to
generalized state spaces was described in \cite{tierney1998note}. The
behavior of the traditional random walk approach as the dimension
grows large was investigated in \cite{mattingly2012diffusion} for a
broad class of target measures. The pCN and MALA
algorithms suitable for infinite-dimensional problems were laid out in
\cite{cotter2013mcmc}; the optimal choice of the step size parameter in
the MALA algorithm was shown in \cite{pillai2012optimal}. HMC was
similarly extended to infinite-dimensions in \cite{beskos2011hybrid},
work that was later generalized in \cite{beskos2017geometric} and
\cite{girolami2011riemann}. Dimension-independent convergence of some
of the above methods has been investigated by showing that the kernels
have spectral gaps
\cite{eberle2014error,hairer2014spectral,vollmer2015dimension},
leveraging a generalized version of Harris' Theorem
\cite{hairer2008spectral,hairer2011asymptotic,hairer2011yet,meyn1993markov}
for Markov kernels. The work in \cref{sec:results} represents one of the first attempts to benchmark these methods for an infinite-dimensional application.

In \cref{sec:ex2_multihump} we present a multimodal posterior measure, which MCMC methods have difficulty resolving. This has been a known problem with MCMC almost since its inception and a number of ideas have been proposed for improving sampling for these distributions. One example is tempering, in which a series of ``less steep'' distributions are used to try to increase the probability of jumps between modes; see, e.g., the description in \cite[Section 12.3]{gelman2014bayesian} and associated references. A related method is equi-energy sampling \cite{kou2006equi}, in which rings of parameter values associated with different energy levels are constructed and samples are allowed to jump within rings.

\paragraph{Advection-Diffusion}

The problem of observing scalar behavior to infer the underlying
velocity field is a common experimental approach for investigating the
structure of complex fluid flows. The textbooks
\cite{merzkirch1987flow} and \cite{smits2012flow} describe many such
methods, including examples where dye, smoke, temperature, hydrogen
bubbles, or photo-sensitive tracers are used. An overview of dye-based
visualization techniques is provided in \cite{karch2012dye}. An
application of dye to investigate two-dimensional turbulence is
described in \cite{williams1997mixing}; see the survey article
\cite{kellay2002two} for additional examples.

To our knowledge, this work is the first to apply Bayesian inference
to the problem of estimating a background fluid flow from measurements
of a passive scalar. However, a number of works, such as
\cite{akccelik2003variational}, have used inversion techniques to
determine a source (forcing) term in advection-diffusion problems. In
those previous works, the background flow was assumed to be known and
the initial condition assumed to be zero; their goal was to determine
the function (in particular the location) from which the pollutant was
dispersed. The source-identification work was extended to ensure
robustness to uncertainties in the velocity field in
\cite{zhuk2016source}.

More generally, the advection and diffusion of passive scalars has
been studied extensively. Numerical difficulties in modeling the
behavior of passive scalars for advection-dominated cases are
described in \cite{morton1996numerical} and
\cite{stynes2013numerical}. Passive scalars exhibit similar
behavior for turbulent and random flows, so the latter, simpler case
may be used to model the former. One such model was introduced by
Kraichnan~\cite{kraichnan1967inertial,kraichnan1968small,kraichnan1991stochastic}; the energy spectrum from this model motivates the construction of the prior measure in \cref{sec:results}.

\section{Mathematical Framework and Bayesian Inference}
\label{sec:math_framework}

In this section, we describe the mathematical framework of the inverse
problem (\ref{eq:Gmap}). We begin by defining the functional analytic
setting for the problem, including how we represent divergence-free
background flows. We then describe reasons why the inverse problem is
ill-posed, i.e., why a given set of measurements $\data$ cannot
identify a unique background flow $\vfield$ that generated them. We
close the section by defining the Bayesian approach to the inverse problem.

\subsection{Representation of Divergence-Free Background Flows}
\label{sec:modelDivFree}

The target of the inference is a divergence-free background flow
$\vfield$, so we start by describing the space $\vfspace$ of such
flows that we will consider.  For this purpose we begin by recalling
the Sobolev spaces of (scalar valued) periodic functions on the 
domain $\spatdom = [0,1]^2$
\begin{equation}
    H^s(\spatdom) 
    = \left\{ u : u = \sum_{\kbf \in \Z^2 \setminus \{\mathbf{0}\}} c_{\kbf} e^{2 \pi i \kbf \cdot \x }, 
      \, \overline{c_{\kbf}} = c_{-\kbf}, \, \| u\|_{H^s} < \infty \right\},
    \,  \| u\|_{H^s}^2 :=  \sum_{\kbf \in \Z^2} \norm{\kbf}^{2s} |c_{\kbf}|^2,
    \label{eq:Hs}
 \end{equation}
 defined for any $s \in \mathbb{R}$; see
 e.g.~\cite{robinson2001infinite, temam1995navier}.  We will abuse
 notation and use the same notation for periodic divergence-free
 background flows by replacing the coefficients $c_{\kbf}$ in
 \eqref{eq:Hs} with
\begin{equation}
  c_{\kbf} = \vk \frac{\kbf^\perp}{\norm{\kbf}_2}, 
  \quad \overline{\vk}=-\vfieldnobf_{-\kbf},
  \label{eq:reality}
\end{equation}
where $\kbf^\perp = [-k_y, k_x]$ to ensure
$\kbf \cdot \kbf^\perp = 0$.  Throughout what follows we fix our parameter space as
\begin{Not}[Parameter space, $\vfspace$]
  We consider background flows $\vfield \in \vfspace$, where
  $\vfspace = H^{m}(\spatdom)$ (see \eqref{eq:Hs}) for some $m > 1$, 
  with coefficients $c_{\kbf}$ given by \eqref{eq:reality}.
  \label{def:vfspace}
\end{Not}
Here the exponent $m$ is chosen so that vector fields in $\vfspace$,
  as well as their corresponding solutions $\pdesol(\vfield)$, exhibit
  continuity properties convenient for our analysis below (see
  \cref{def:adr_weak}).  We take $L^p(\spatdom)$ with
  $p \in [1,\infty]$ for the usual Lebesgue spaces and denote the
  space of continuous and $p$-th integrable, $X$-valued functions by
  $C([0,T];X)$ and $L^p([0,T]; X)$, respectively, for a given Banach space $X$.  All
  of these spaces are endowed with their standard topologies unless
  otherwise specified.

  In what follows we frequently consider Borel probability measures on
  $H$, denoted sometimes as $\mbox{Pr}(H)$, in
  reference to the prior and posterior measures produced by Bayes'
  theorem below.  A natural approach to construct certain classes
  of such infinite dimensional probability measures is to decompose them
  into one-dimensional probability measures acting independently on
  individual components of a sequence of elements sitting in an
  underlying function space; see e.g.~\cite[Section
  2]{dashti2017bayesian}. Concretely in our setting, probability
  measures on the space of divergence-free vector fields can be
  defined by letting $\vk$ be random fields, as long as $\vk$ exhibits
  suitable decay to zero as $\norm{\kbf} \to \infty$ commensurate with
  $\vfield \in \vfspace$ (see \cref{def:vfspace}). In particular we make use of this construction
  on the space of divergence-free vector fields to define
  prior distributions in the numerical examples in \cref{sec:results}.

\subsection{Mathematical Setting for 
  the Advection-Diffusion Equation}\label{sec:math_setting}
In this section, we provide a precise definition of solutions
$\pdesol$ for the advection-diffusion problem \eqref{eq:adr}.
Crucially, the setting we choose yields a map from $\vfield$ to
$\pdesol$ and then to observations of $\pdesol$ that is continuous.
\begin{Prop}[Well-Posedness and Continuity of the solution map for
  \eqref{eq:adr}]\label{def:adr_weak}
  \mbox{}
\begin{itemize}
  \item[(i)]
  Fix any $s \geq 0$ and $m \geq s$ with $m > 0$ and suppose that
  $\vfield \in H^{m}(\spatdom)$ and $\pdesol_0 \in H^{s}(\spatdom)$.
  Then there exists a unique $\pdesol = \pdesol(\vfield, \pdesol_0)$
  such that
  \begin{align*}
    \pdesol \in L^2_{loc}([0,\infty); H^{s+1}(\spatdom)) 
        \cap L^\infty([0,\infty); H^{s}(\spatdom))
    \quad \text{ with } \quad
    \frac{\partial \pdesol}{\partial t} 
    \in L^2_{loc}([0,\infty); H^{s-1}(\spatdom))
  \end{align*}
  so that in particular $\pdesol \in C([0,\infty); H^{s}(\spatdom))$ solves~(\ref{eq:adr}) at least weakly, namely
    \begin{equation}
      \ip{\frac{\partial\pdesol}{\partial t}{}}{\phi}_{H^{-1}\left( \spatdom \right) \times H^1(\spatdom)}  
      + \ip{\vfield \cdot \nabla \pdesol}{\phi}_{L^2\left( \spatdom \right)} 
      + \conductivity \ip{\nabla \pdesol}{\nabla \phi}_{L^2\left( \spatdom \right)} = 0
    \label{eq:adr_weak}
  \end{equation}
  for all $\phi \in H^1(\spatdom)$ and almost all time
  $t \in [0,\infty)$.  
  \item[(ii)]
  For any $T > 0$ the map that associates $\vfield \in H^m(\spatdom)$ and
  $\pdesol_0 \in H^s(\spatdom)$ to the corresponding
  $\pdesol(\vfield, \pdesol_0)$ is continuous relative the standard
  topologies on $H^m(\spatdom) \times H^s(\spatdom)$ and
  $C\left( [0,T] \times H^s(\spatdom) \right)$.
  \item[(iii)]
  For any $T > 0, m \ge s > 1$ the map which associates $\vfield \in H^m(\spatdom)$ and
  $\pdesol_0 \in H^s(\spatdom)$ to the corresponding
  $\pdesol(\vfield, \pdesol_0)$ is continuous relative the standard
  topologies on $H^m(\spatdom) \times H^s(\spatdom)$ and
  $C\left( [0,T] \times \spatdom \right)$.
\end{itemize}
\end{Prop}
A sketch of the proof is provided in \cite{borggaard2018consistency}.

\begin{Rmk}
  \label{rmk:Illposs}
  Since the background flow $\vfield$ enters \eqref{eq:adr} through
  the $\vfield \cdot \nabla \pdesol$ term, the inverse problem of
  recovering $\vfield$ from $\pdesol(\vfield)$ can be ill-posed.  One
  important class of examples illustrating this difficulty arises
  when $\vfield \cdot \nabla \pdesol$ is zero everywhere, in which
  case the fluid flow does not have any influence on $\pdesol$. Two such
  examples are as follows:
\begin{itemize}
\item[(i)] \underline{Ill-posedness: Laminar Flow:}
 Let $\pdesol_0(\x)$ be independent of $y$ and
    $\vtrue = \left[ 0,f(x) \right]$. Then
    $\pdesol(\vtrue)=\pdesol(\vfield)$ for any
    $\vfield=\left[ 0,g(x) \right]$.
  \item[(ii)] \underline{Ill-posedness: Radial Symmetry:} Set
    $\pdesol_0(\x)\propto \sin (\pi x) + \sin(\pi y)$ and
    $\vtrue = [\cos(\pi x),$ $-\cos(\pi y)]$. Then
    $\pdesol(\vtrue)=\pdesol(\vfield)$ for any
    $\vfield = c \vtrue, \,c \in \R$.
\end{itemize}
\noindent In these cases, the even noiseless and complete spatial/temporal
observations of $\pdesol$ cannot discriminate between a range
of background flows, making it impossible to uniquely identify a true
background flow $\vtrue$.
\end{Rmk}

With this general result in hand we now fix some notation used
for the remainder of the paper.
\begin{Def}[Solution Operator $\Sol$, Observation Operator $\Obs$]
  Fix $\pdesol_0 \in H^s(\spatdom)$ and a final time $T > 0$ and consider the phase
  space $\vfspace$ defined as in \cref{def:vfspace}.  The forward map $\G$ as
  in \eqref{eq:Gmap} is interpreted as the composition
  $\G(\vfield) = \Obs \circ \Sol(\vfield)$, where:
  \begin{enumerate}
  \item The \emph{solution operator}
    $\Sol:\vfspace \to C([0,T]; H^s(\spatdom))$ maps a
    given $\vfield$ to the corresponding solution $\pdesol(\vfield, \pdesol_0)$ 
    of \eqref{eq:adr} (in the sense of \cref{def:adr_weak}).
  \item The \emph{observation operator}
    $\Obs: C([0,T]; H^s(\spatdom) ) \to \dataspace$ measures some
    quantities (e.g. point measurements, spectral data, tracers) from
    $\pdesol$.  Here, in general, $\dataspace$ is a separable Hilbert
    space.  However, since we are primarily focused on the setting of
    finite observations, we typically have $\dataspace = \mathbb{R}^N$.
  \end{enumerate}
  \label{def:sol}
  \label{def:obs}
\end{Def}
To make a connection with the range of observations provided in the introduction, 
we detail the following possibilities for $\Obs$.
\begin{Ex}
  \label{ex:Obs:Op}
  We consider finite observations $\Obs(\pdesol) = (\Obs_1(\pdesol), \ldots, \Obs_N(\pdesol))$
  that could be
  \begin{enumerate}
  \item Spectral Observations:
    $\Obs_j(\pdesol) = \int_{[0,1]^2} \pdesol(t_j,\x) e_j d\x$, with
    $\{e_j\}_{j \geq 0}$ an orthonormal basis for $H^s$ and $t_j \in [0,T]$.
  \item Local averages:
    $\Obs_j(\pdesol) = \frac{1}{|\mathcal{D}_j|} \int_{\mathcal{D}_j}
    \pdesol \,d\x \,dt$,
    for any sub-domains $\mathcal{D}_j \subset [0,T] \times [0,1]^2$
    where $|\mathcal{D}_j|$ denotes the volume of $\mathcal{D}_j$.
  \item Spatial-temporal point observations:
      $\Obs_j(\pdesol) = \pdesol(t_j, \x_j)$ for any $t_j \in [0,T]$ and
      $\x_j \in [0,1]^2$. (Note that point observations are well-defined by \cref{def:adr_weak}, (iii).)
  \end{enumerate}
\end{Ex}
\noindent This paper focuses on final case of point observations
$\G_j(\vfield)=\pdesol(t_j,\x_j,\vfield),\, j=1,\dots,N$ as the most
obvious practical implementation for the advection-diffusion problem.

\subsection{Ill-posedness} \label{sec:illposed}
We note that the 
classical inverse problem of recovering $\vfield$ from data $\data$ (see \eqref{eq:Gmap}) is
highly ill-posed in a number of ways:
\begin{enumerate}
\item The data is incomplete, i.e., we do not observe $\pdesol$
  everywhere.  For this reason we are interested in forward maps
  $\G(\vfield)$ that are non-invertible and hence that do not
  uniquely specify $\vfield$.  One such example is provided in
  \cref{sec:ex2_multihump}.
\item Even if solutions $\pdesol$ of (\ref{eq:adr}) are observed
  everywhere in space and time, there are initial conditions
  $\pdesol_0$ such that any of a range of background flows
  $\vfield$ produce the same scalar field $\pdesol$. Two such examples
  are provided in \cref{rmk:Illposs} above.
\item Because of the observational noise $\eta$ in
  (\ref{eq:Gmap}), there may be no $\vfield$
  such that $\G(\vfield) = \data$ for given data $\data$.  For
  example, in the case of point observations
  $\data_j = \pdesol(t_j,\x_j,\vtrue)+\eta_j$, some realizations of
  $\eta_j$ may cause $\data_j$ to exceed the maximum value (or be less
  than the minimum value) of $\pdesol_0$. However, because
  $\nabla \cdot \vfield = 0$, \eqref{eq:adr} is a parabolic PDE that
  is subject to a maximum principle implying
  $\norm{\pdesol(t)}_{L^\infty(\spatdom)} \le
  \norm{\pdesol_0}_{L^\infty(\spatdom)}$
  for all $t > 0$. Thus there would be no $\vfield$ such that
  $\G(\vfield) = \data$.
\end{enumerate}
These considerations are typical of ill-posed inverse problems more
broadly. See, e.g.~\cite[Section 2]{kaipio2005statistical} or
\cite{stuart2010inverse} for further commentary.

\subsection{Bayesian Inference} \label{sec:bayes}
Following the Bayesian approach to inverse problems \cite{dashti2017bayesian,kaipio2005statistical}, instead of
seeking a single best match $\vtrue$, we take a statistical
interpretation of $\vfield$ and $\eta$ as random quantities that we
refer to as `the prior' and `the observation noise'.  The solution
of \eqref{eq:Gmap} is a probability measure, known as the
`posterior', associated with the conditional random variable
`$\vfield | \data$'.  The concentration
of the prior measure in the limit of a large number of observations,
i.e.~the question of consistency, is investigated in detail in \cite{borggaard2018consistency}. 
A quite general formulation of Bayes' Theorem is provided in \cref{sec:bayes_general}; in this section we follow closely the derivation in \cite{dashti2017bayesian}. We begin by imposing the following typical assumption:
\begin{Assum}
The joint distribution of the observation noise and the prior take the form $(\vfield,\eta) \sim \mu_0 \otimes\gamma_0$ for $\mu_0 \in \mbox{Pr}(H)$, $\gamma_0 \in \mbox{Pr}(\RR^N)$ so that $\vfield$ and $\eta$ are
statistically independent.
  \label{ass:prior:noise:ind}
\end{Assum}
\noindent Under \cref{ass:prior:noise:ind}, the `likelihood' $\Q_{\vfield}$, heuristically
$\data | \vfield$, is
\begin{Lem}[Likelihood $\Q_\vfield$]
  \label{def:Qv}
  For any deterministic background flow $\vfield \in \vfspace$ and
  observation noise $\noise \sim \noisemeasure$, the \emph{likelihood}
  $\Q_{\vfield}$ satisfies $\G(\vfield) + \noise \sim \Q_{\vfield}$ 
  so that for any $A \in \mathcal{B}(\RR^N)$,
  \begin{equation}
    \Q_{\vfield}(A) = \noisemeasure(\{\ysm - \G(\vfield) : \ysm \in A\}).
    \label{eq:Qv:2}
  \end{equation}
\end{Lem}

With the form of the likelihood measure $\Q_{\vfield}$ in hand we
introduce the following notational convention used several times
below
\begin{Not}[True background flow, $\vtrue$]
  We frequently fix a ``true'' background flow by $\vtrue \in H$.  For the given
  $\vtrue$, the observed data 
  $\data = \G(\vtrue) + \noise$
  can be viewed as draws from the  distribution $\Q_{\vtrue}$
  (though $\vtrue$ is not necessarily the only $\vfield$ that
  could produce such data).
  \label{def:vtrue}
\end{Not}

As in \cite{dashti2017bayesian}, we make the following assumption:
\begin{Assum}
  The likelihood $\Q_{\vfield}$ (see \cref{def:Qv}) is absolutely
  continuous with respect to the noise measure $\noisemeasure$ for all
  $\vfield \in \vfspace$.
  \label{ass:Q0_Qv_abscon}
\end{Assum}
\noindent We note that this assumption holds when $\noisemeasure$ is any
continuously distributed measure, such as a (non-degenerate) Gaussian,
that has the whole space $\RR^N$ as its support. (We also note in \cref{thm:bayes_general} that $\noisemeasure$ is not the only suitable choice of reference measure.) Then we define:

\begin{Def}[Potential, $\Phi$]
  When \cref{ass:Q0_Qv_abscon} holds, the \emph{potential} or
  \emph{negative log-likelihood}
  $\Phi: \vfspace \times \RR^N \to \R$ is defined as
  \begin{equation} \label{eq:pot_def}
  	\Phi(\vfield;\data)
        = -\log\left(\frac{d\Q_{\vfield}}{d\noisemeasure}(\data)\right)
  \end{equation}
  where $\frac{d\Q_{\vfield}}{d\noisemeasure}$ is the Radon-Nikodym derivative 
  of $\Q_{\vfield}$ with respect to $\noisemeasure$.
  \label{def:phi}
\end{Def}

\begin{Ex}[Gaussian Noise]\label{ex:gaussian_noise}
If the observation noise is a centered Gaussian, 
i.e.~$\noisemeasure = N(0,\mathcal{C}_\noise)$, then by \eqref{eq:Qv:2} we have 
(up to a factor independent of $\vfield$)
\begin{equation} \label{eq:pot_def_gaussian}
	\Phi(\vfield;\data) = \half \norm{ \mathcal{C}_\noise^{-1/2} \left( \data - \G(\vfield) \right) }^2.
\end{equation}
\end{Ex}

Finally, we have the following adaptation of Bayes' Theorem to the
advection-diffusion problem:
\begin{Thm}[Bayes' Theorem, \cite{dashti2017bayesian}]
	\label{thm:Bayes:AD}
        Let $\Q_{\vfield}$ and $\Phi$ be defined as in \cref{def:Qv}
        and \cref{def:phi}, respectively, and let $\Q_{\vfield}$
        satisfy \cref{ass:Q0_Qv_abscon}.
  Suppose that $\Phi$ is measurable in $\vfield$ and $\data$ and that
  \begin{equation} \label{eq:bayes_gt0_cond}
  	Z = \int \exp\left( - \Phi(\vfield;\data) \right) \mu_0(d\vfield) > 0.
  \end{equation}
  Then the measure $\mu_\data$ associated with the random variable
  $\vfield | \data$ is absolutely continuous with respect to $\mu_0$,
  with Radon-Nikodym derivative
\begin{equation}
	\frac{d\mu_{\data}}{d\mu_0}(\vfield) = \recip{Z} \exp\left( - \Phi(\vfield;\data) \right).
  \label{eq:bayes}
\end{equation}
\end{Thm}

\section{Computational Approach and Challenges}\label{sec:numericalmethods}
In this section, we describe the numerical methods used to approximate
the posterior measure $\mu$. We begin by introducing Markov Chain
Monte Carlo (MCMC) methods used to generate samples from $\mu$
(\cref{sec:mcmc_stub}). \Cref{sec:gEval} describes how we discretize and solve the advection-diffusion equation \eqref{eq:adr} to compute the potential
$\Phi$ (see \cref{def:phi}). Finally, in \cref{sec:adjoint} we define an
adjoint method for efficient computation of the \Frechet derivative
$D \Phi$, which is required for implementation of some of the more
advanced MCMC algorithms described in
\cref{sec:mcmc_stub}. 

\subsection{Sampling from $\mu$ via Markov Chain Monte Carlo (MCMC)}
\label{sec:mcmc_stub}

To sample from the posterior measure $\mu_{\data}$ (see \cref{thm:Bayes:AD}), we use two Markov Chain Monte Carlo (MCMC) methods recently developed for or extended to infinite-dimensional problems: (1) preconditioned Crank-Nicolson (pCN) \cite{cotter2013mcmc}, a generalization of the classical random walk algorithm that requires one forward evaluation (PDE solve) per iteration and represents the ``inexpensive'' end of the computational spectrum (see \cref{alg:mcmcpcn}); and (2) Hamiltonian Monte Carlo (HMC) \cite{bou2018geometric,beskos2017geometric,duane1987hybrid}, a ``computationally expensive'' method that requires multiple PDE solves and gradient computations per iteration (see \cref{alg:mcmchmc}). In \cref{sec:results_is_mala}, we additionally present some results for the independence sampler and Metropolis-adjusted Langevin Algorithm (MALA) (see, e.g., descriptions in \cite{dashti2017bayesian} and \cite{beskos2017geometric}, respectively). See also \cite[Chapter 5]{krometis2018bayesian} for a detailed description of these four methods and \cite{hoffman2014no} for an algorithm to recursively select parameters for HMC.

\begin{algorithm}
\caption{Preconditioned Crank-Nicolson (pCN) MCMC.}\label{alg:mcmcpcn}
\begin{algorithmic}[1]
\item Given free parameter $\beta \in (0,1]$ and initial sample $\mcmcsamp^{(k)}$
\item Propose $\mcmccand = \sqrt{1-\beta^2}\mcmcsamp^{(k)} + \beta \xi^{(k)}$, $\xi^{(k)} \sim N(0,\covar)$
\item Set $\mcmcsamp^{(k+1)} = \mcmccand$ with probability
  $\min\left\{1,\exp\left(\Phi\left(\mcmcsamp^{(k)}\right)
      - \Phi\left(\mcmccand\right) \right)\right\}$, otherwise $\mcmcsamp^{(k+1)} = \mcmcsamp^{(k)}$
\end{algorithmic}
\end{algorithm}

\newcommand{\hmcq}{\mathbf{q}}
\newcommand{\hmcp}{\mathbf{p}}
\newcommand{\hmcv}{\mathbf{w}}
\newcommand{\hmcH}{\mathcal{H}}
\newcommand{\hmcU}{\mathcal{U}}
\newcommand{\hmcK}{\mathcal{K}}
\begin{algorithm}
\caption{Hamiltonian MCMC (HMC).}\label{alg:mcmchmc}
\begin{algorithmic}[1]
\State Given free parameters $\tau \ge \epsilon > 0$ and initial sample $\mcmcsamp^{(k)}$. Set $L=\frac{\tau}{\epsilon}$.
\State Set $(\hmcq_0,\hmcv_0) = (\mcmcsamp^{(k)},\hmcv)$, where $\hmcv \sim \mu_0 = N(0,\covar)$
\For{$i=1$ to $L$} 
\State Integrate Hamiltonian dynamics (compute $\Phi$ and $D\Phi$): $(\hmcq_{i-1},\hmcv_{i-1}) \mapsto (\hmcq_i,\hmcv_i)$
\EndFor
\State Compute $\Delta \hmcH =  \hmcH(\hmcq_L,\hmcv_L) - \hmcH(\mcmcsamp^{(k)},\hmcv)$ 
\State Set $\mcmcsamp^{(k+1)} = \hmcq_L$ with probability $\min\left\{1,\exp \left[ -\Delta \hmcH \right]\right\}$, otherwise $\mcmcsamp^{(k+1)} = \mcmcsamp^{(k)}$ 
\end{algorithmic}
\end{algorithm}

\subsection{Evaluation of $\G$}\label{sec:gEval}
Computing the potential $\Phi(\vfield)$ (see \cref{def:phi}) as in \cref{alg:mcmcpcn} or \cref{alg:mcmchmc} requires evaluating $\G(\vfield)$,
i.e., computing (e.g., point) observations for $\pdesol(\vfield)$. This
requires numerically solving \eqref{eq:adr} using a PDE solver. We do so using a spectral method \cite{canuto2007spectrali,gottlieb1977numerical}, expanding $\vfield$ in a Fourier basis as in \eqref{eq:Hs}, \eqref{eq:reality} and $\pdesol$ similarly as $\pdesol(t,\x)=\sum_{\kbf} \pdesol_{\kbf}(t) e^{2 \pi
  i\kbf\cdot\x}$.
We apply a Galerkin projection, writing the coefficients $\pdesol_{\kbf}$ as a system of ODEs that reduces to
\begin{equation}
  \ddt \pdesolvec(t) = A \pdesolvec(t),\quad\text{where } (A)_{lm} = -i v_{\kbf'} \left( {\kbf'}^\perp \cdot \kbf_m \right) - \conductivity \norm{\kbf_l}^2\delta_{lm}, \quad \text{ where } \kbf' = \kbf_l - \kbf_m.
  \label{eq:spectralOde}
\end{equation}
This system is then integrated using the implicit midpoint (Crank-Nicolson) method to approximate $\pdesolvec(t)$. Point observations are calculated as $\pdesol(t_j,\x_j) = \sum_{\kbf} \pdesol_{\kbf}(t_j) e^{2 \pi i\kbf\cdot\x_j}$, with evaluation at time $t_j$ interpolated if $t_j$ does not fall on a timestep of the integration method. 
When $\conductivity$ is small, accurate representation of $\pdesol$ requires a large number of components due to the frequency cascade $\kbf'$ in \eqref{eq:spectralOde} (see \cite{stynes2013numerical}). This challenge motivates our concurrent work in
\cite{borggaard2019particle}, in which we introduce a particle method
for efficient evaluation of $\G(\vfield)$, allowing the computation of
large numbers of samples for low-$\conductivity$ problems.

\subsection{An Adjoint Method for Evaluating the Gradient of $\Phi$}
\label{sec:adjoint}
HMC (\cref{alg:mcmchmc}) requires evaluating the \Frechet derivative of the potential $\Phi$ (see \cref{def:phi}) with respect to changes in $\vfield$, a direct approach to which would require many PDE solves. Here we introduce an adjoint approach that requires a single PDE solve per gradient computation using (a forced version of) the same solver used to solve \eqref{eq:adr}. See also \cite{hinze2009optimization} for a detailed discussion of adjoint methods and \cite{akccelik2003variational}, in which a similar adjoint equation was derived for a different application.

\newcommand{\obssp}{H^m\left( [0,T]; H^s(\spatdom) \right)}
\newcommand{\obsdualsp}{H^{-m}\left( [0,T]; H^{-s}(\spatdom) \right)}
\newcommand{\adjsp}{H^{1-m}\left( [0,T]; H^{2-s}(\spatdom) \right)}
\newcommand{\obsspsh}{H^{m,s}}
\newcommand{\adjspsh}{H^{1-m,2-s}}
\newcommand{\ipsp}{\adjspsh\times H^{m-1,s-2}}
\newcommand{\testl}{\rho}
\newcommand{\adjsol}{\testl_0}
\newcommand{\adjsolbw}{\tilde{\testl}_0}
\newcommand{\adjsolbwl}{\tilde{\testl}_{0_l}}
\newcommand{\testt}{\phi}
\newcommand{\dth}{\psi}
\newcommand{\dvf}{\hat{\vfield}}
\begin{Thm}[Adjoint Method for Evaluating $D\Phi$]
  For a given background flow $\vfield \in H$, let $\pdesol(\vfield)\in \obssp$ be a weak solution 
  of the advection-diffusion equation \eqref{eq:adr} in the sense that $\pdesol(0,\x)=\pdesol_0$ a.e. and
  \begin{equation}
      \ipZ{\testl}{\ddt\pdesol+\vfield\cdot\nabla\pdesol-\conductivity\Delta \pdesol}{\ipsp}=0
      \label{eq:adjweak}
  \end{equation}
  for all $\testl \in \adjsp$. Suppose that $\Phi$ (see \cref{def:phi}) and $\Obs: \obssp \to \R^N$ (see \cref{def:obs}) are continuously differentiable in $\G$ and $\pdesol$, respectively. Suppose there exists a
  $\adjsol \in \adjsp$ with $\adjsol(0,\x)=0$ a.e. that solves the forced adjoint equation
  \begin{equation}
    \begin{aligned}
      \ipZ{\ddt\adjsol-\vfield\cdot\nabla\adjsol-\conductivity\Delta \adjsol}{\testt}{H^{-m,-s}\times\obsspsh}
      &= - \frac{\partial \Phi}{\partial \G}(\vfield) \cdot \Obs[\tilde{\testt}] \\
    \end{aligned}
    \label{eq:adjoint_tau}
  \end{equation}
  for all $\testt \in \obssp$ with $\testt(T,\x)=0$ a.e., where $\tilde{\testt}(t,\x) \coloneqq \testt(T-t,\x)$. Then the \Frechet 
  derivative of $\Phi$ at $\vfield$ in the direction $\dvf$ 
  is given by
  \begin{equation}
    D_{\dvf} \Phi(\vfield) 
    = \ipZ{\adjsolbw}{\dvf \cdot \nabla \pdesol}{\ipsp}
    \label{eq:adjoint_integral}
  \end{equation}
  where $\adjsolbw(t,\x) \coloneqq \adjsol(T-t,\x)$.
  \label{thm:adjoint}
\end{Thm}
\begin{proof}
  Application of the chain rule yields
  \begin{equation}
    D_{\dvf}\Phi(\vfield) 
    = \frac{\partial \Phi}{\partial \G}(\vfield) \cdot D_{\dvf} \G(\vfield) 
    = \frac{\partial \Phi}{\partial \G}(\vfield) \cdot \Obs\left[ D_{\dvf} \pdesol(\vfield) \right]. 
    \label{eq:dphi}
  \end{equation}
  Denote $D_{\dvf}\pdesol(\vfield)$ by
  $\psi(\vfield,\dvf)$. Then by applying \eqref{eq:adjweak} to
  $\pdesol(\vfield+\epsilon\dvf)$ and $\pdesol(\vfield)$,
  subtracting, taking the $\epsilon \to 0$ limit, and using the
  definition of the \Frechet derivative, we see that
  $\dth\in\obssp$ satisfies
  \begin{equation}
    \ipZ{\testl}{\ddt\dth+\vfield\cdot\nabla\dth-\conductivity\Delta \dth+\dvf\cdot\nabla\pdesol}{\ipsp}=0
    \label{eq:dth_pde}
  \end{equation}
  with $\dth(0,\x)=0$ a.e., for all $\testl \in \adjsp$. Also, changing variables from $t$ to $T-t$ in \eqref{eq:adjoint_tau} yields the following relationship for $\adjsolbw(t)=\adjsol(T-t)$:
  \begin{equation}
    \begin{aligned}
      \ipZ{\ddt\adjsolbw+\vfield\cdot\nabla\adjsolbw+\conductivity\Delta \adjsolbw}{\tilde{\testt}}{H^{-m,-s}\times\obsspsh}
      &= \frac{\partial \Phi}{\partial \G}(\vfield) \cdot \Obs[\testt] \\
    \end{aligned}
    \label{eq:adjoint}
  \end{equation}
  with $\adjsolbw(T,\x)=\tilde{\testt}(0,\x)=0$ a.e. Then applying \eqref{eq:dphi}, \eqref{eq:adjoint}, 
  and \eqref{eq:dth_pde} in succession yields
  \begin{align*}
      D_{\dvf} \Phi(\vfield) 
    = \frac{\partial \Phi}{\partial \G}(\vfield) \cdot \Obs\left[ \psi \right] 
    &= \ipZ{\ddt\adjsolbw+\vfield\cdot\nabla\adjsolbw+\conductivity\Delta \adjsolbw}{\dth}{H^{-m,-s}\times\obsspsh}\\
    &= \ipZ{\adjsolbw}{-\ddt\dth-\vfield\cdot\nabla\dth+\conductivity\Delta \dth}{\ipsp}\\
    &= \ipZ{\adjsolbw}{\dvf\cdot\nabla\pdesol}{\ipsp}, 
    \end{align*}
  which is the desired result.
\end{proof}
\begin{Rmk}
  Note that $\left[ \frac{\partial \Phi}{\partial \G}(\vfield) \cdot \Obs \right] \in \obsdualsp$, so solving \eqref{eq:adjoint_tau} amounts to finding the weak solution in $\obsdualsp$ of
  \begin{equation}
    \ddt\adjsol-\vfield\cdot\nabla\adjsol-\conductivity\Delta \adjsol = \frac{\partial \Phi}{\partial \G}(\vfield) \cdot \tilde{\Obs}, \quad \adjsol(0,\x)=0 \text{ a.e.},
      \label{eq:adj_pde}
  \end{equation}
  where $\tilde{\Obs}[\testt(t,\x)]\coloneqq\Obs[\testt(T-t,\x)]$.
\end{Rmk}

To compute the full gradient $D\Phi$ (the derivative with respect to
an array of bases $\{\ek\}$), we compute the integration
\eqref{eq:adjoint_integral} for $\dvf=\ek$ for each
$\kbf$. The resulting algorithm is summarized in
\cref{alg:adjoint}. Note that solving \eqref{eq:dth_pde} and
substituting into \eqref{eq:dphi} would also yield the derivative of
$\Phi$. However, this approach would require a PDE solve for each
direction $\dvf$ in which we want to take the
derivative. In particular, if we want the full gradient, we have to do
many PDE solves. By contrast, \cref{alg:adjoint} requires only one
additional PDE solve per gradient calculation. Moreover, note that \eqref{eq:adj_pde} is equivalent to
\eqref{eq:adr} with zero initial condition, a reversed vector field,
and a forcing term. Thus, the same PDE solver can be used for both the
forward and adjoint solves with minimal modification. 
For the numerical experiments in \cref{sec:results}, the adjoint equation \eqref{eq:adjoint_tau} was solved using the same spectral method described in \cref{sec:gEval}, where the forcing term was similiarly expanded in the Fourier basis $e^{2\pi i \kbf\cdot\x}$ via Galerkin projection. The integration \eqref{eq:adjoint_integral} was computed directly from the spectral representation of $\pdesol$ and $\adjsol$, yielding
\begin{equation}
  D_{\ek} \Phi(\vfield) = \sum_j 2 \pi i \left( \kbf^{\perp} \cdot \kbf_j \right) \int_{0}^{T} \adjsolbwl(t) \pdesol_j(t) \quad\text{where}\quad l \ni \kbf_l = -\kbf-\kbf_j,
\end{equation}
where the time integration was computed via trapezoidal rule.

\begin{algorithm}
\caption{Adjoint Method for Computing $D\Phi$.}\label{alg:adjoint}
\begin{algorithmic}[1]
  \State Given $\vfield$ and basis $\ek$
  \State Solve \eqref{eq:adr} for $\pdesol(t,\x,\vfield)$
  \State Solve \eqref{eq:adjoint_tau} for $\adjsol(t,\x,\vfield)$
  \For{each $\kbf$}
    \State Compute $D_{\kbf}\Phi(\vfield)$ via \eqref{eq:adjoint_integral} with $\dvf=\ek$
  \EndFor
\end{algorithmic}
\end{algorithm}

\begin{Ex}[Point Observations, Gaussian Noise]
  Let the observation operator be point observations $\Obs_j[\pdesol] = \pdesol(t_j,\x_j)$. These observations are well-defined for $\pdesol \in \obssp$ with $m>\frac{1}{2}, s>1$ (see \cref{def:adr_weak}, (iii)). Let $\noise_j \sim N(0,\sigma_\noise^2)$ for $j=1,\dots,N$ so that
  $\Phi$ is given by \eqref{eq:pot_def_gaussian}.
  Then solving \eqref{eq:adjoint_tau} amounts to finding the weak solution of
  \begin{align*}
    \frac{\partial}{\partial t}\adjsol - \vfield \cdot \nabla \adjsol 
    - \conductivity \Delta \adjsol 
    = \sum_j \frac{1}{\sigma_\noise^2} 
       \left( \data_j - \pdesol(t_j,\x_j,\vfield) \right) \delta(T-t_j-t,\x-\x_j),
  \end{align*}
  where $\adjsol(0,\x)=0$ a.e.~and $\delta(t-t_0,\x-\x_0)$ is a Dirac distribution centered at $(t_0,\x_0)$.
\end{Ex}

\begin{Ex}[Integral Observations]
  Let observations be given by $\Obs_{j}[\pdesol]=\ipZ{f_j}{\pdesol}{L^2([0,T]\times\spatdom)}$ 
  for some $f_j \in H^1([0,T];H^2(\spatdom))$. 
  Let $\pdesol \in L^2([0,T];L^2(\spatdom))$ (i.e., $m=s=0$ in \cref{thm:adjoint}). Let $\noise_j \sim N(0,\sigma_\noise^2)$ for $j=1,\dots,N$ so that
  $\Phi$ is given by \eqref{eq:pot_def_gaussian}. Then solving \eqref{eq:adjoint_tau} amounts to finding $\adjsol \in H^1([0,T],H^2(\spatdom))$ with $\adjsol(0,\x)=0$ such that
  \begin{align*}
    \frac{\partial}{\partial t}\adjsol - \vfield \cdot \nabla \adjsol 
    - \conductivity \Delta \adjsol 
    = \sum_j \frac{1}{\sigma_\noise^2} 
       \left( \data_j - \Obs_j[\pdesol(\vfield)] \right) f_j(T-t,\x).
  \end{align*}
\end{Ex}

\section{Numerical Experiments: Posterior Complexity and MCMC Convergence}\label{sec:results}
In this section, we describe applications of the above methods to two
sample problems. 
We begin with an example (\cref{sec:ex1_simple})
that yields a posterior measure with a relatively simple structure. This provides a baseline for measuring convergence of the pCN and HMC samplers (see \cref{sec:mcmc_stub}), which for our purposes represent the ``inexpensive'' and ``expensive'' ends of the computational spectrum, respectively. \Cref{sec:ex2_multihump} then presents a second example
for which the posterior measure exhibits a 
complicated, multimodal structure. This example is more
challenging for MCMC methods to sample from, and thus a good
test of the advantages offered by more ``expensive'' methods like HMC. 
The appendices present analogous results for the IS and MALA samplers (\cref{sec:results_is_mala}) and additional observables of interest to the passive scalar community (\cref{sec:results_obs}).

In each example, we generate data $\data$ by running a high-resolution
simulation of the system for a given true vector field $\vtrue$ and
applying the observation operator $\Obs$. 
The PDE solver, adjoint solver, and MCMC methods were implemented in the Julia numerical computing language \cite{bezanson2017julia}. Thousands, or in some cases millions, of samples were generated using the computational resources at Virginia Tech.\footnote{\url{http://www.arc.vt.edu}}

\subsection{Example 1: Single-welled Posterior}\label{sec:ex1_simple}
In this subsection, we construct an example that yields a posterior
distribution with a simple, single-welled structure. The problem parameters for this example are enumerated
in \cref{tab:ex1_parameters}.\footnote{Note that $\vtrue$ and
  observation locations/times were chosen randomly only to generate
  the scenario and data. This random selection ensures
  a sufficiently general problem. The data and observation operator
  then of course remain fixed during MCMC sampling.} The true 
flow $\vtrue$ is shown in \cref{fig:ex1_vtrue}.

\renewcommand{\arraystretch}{1.4}
\begin{table}[htbp]
  {\footnotesize
  \caption{Problem parameters for Example 1.} \label{tab:ex1_parameters}
  \centering
  \begin{tabular}{|>{\raggedright}p{0.125\textwidth}|p{0.3\textwidth}|>{\raggedright}p{0.125\textwidth}|p{0.3\textwidth}|} \hline
    Parameter & Value &
    Parameter & Value \\
    \hline\hline
    Observation operator, $\Obs$ & Point observations at 1,024 uniformly random $(t,x,y)$ &
    Data, $\data$ & $\G(\vtrue)$ \\\hline
    Prior, $\mu_{0}$ & Kraichnan \eqref{eq:kraichnan_prior} &
    Noise, $\noisemeasure$ & $N(0,\sigma_{\noise}^2 I)$, $\sigma_\noise=2^{-6}$ \\\hline
    True flow, $\vtrue$ (\cref{fig:ex1_vtrue}) & Randomly drawn from Kraichnan prior, $\norm{\kbf}_2\le 32$ &
    Sampling space, $\vfspace_N$ & $\norm{\kbf}_2\le 8$ (197 components) \\\hline
    Diffusion, $\conductivity$ & $0.282$, for water in air \cite{cussler2009diffusion} &
    $\pdesol_0$ & $\half-\frac{1}{4}\cos(2 \pi x)-\frac{1}{4}\cos(2 \pi y)$ \\\hline
  \end{tabular}
  }
\end{table}

\begin{figure}[!htbp]
\centering
\begin{minipage}[b]{0.49\textwidth}
	\includegraphics[width=\textwidth]{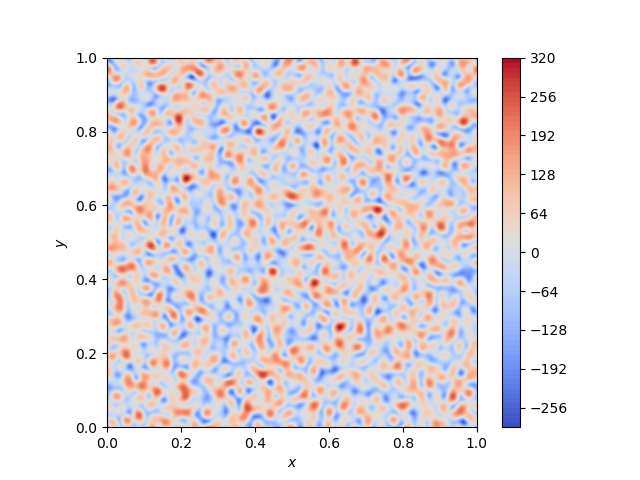}
\end{minipage}
\hfill
\begin{minipage}[b]{0.49\textwidth}
	\includegraphics[width=\textwidth]{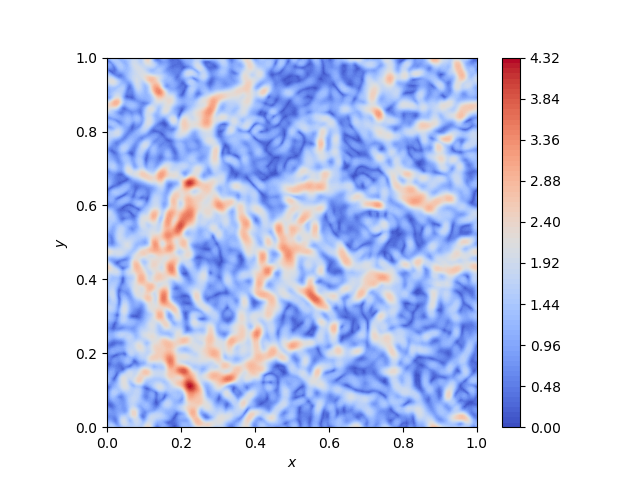}
\end{minipage}
\caption{\label{fig:ex1_vtrue}$\vtrue$ for Example 1. Left: Vorticity $\nabla \times \vtrue$, Right: $\norm{\vtrue}$.}
\end{figure}

For the prior measure, we leverage the Kraichnan model \cite{kraichnan1967inertial,kraichnan1968small} of turbulent
advection via a Gaussian random velocity field with energy spectrum 
(see \cite[Equation 28]{chen1998simulations})
\begin{equation}
  E(k) = E_0 \sum_{i=0}^N \left( \frac{k}{k_i} \right)^4 
             \exp\left[ -\frac{3}{2} \left( \frac{k}{k_i} \right)^2 \right] k_i^{-\xi}
  \label{eq:kraichnanEnergy}
\end{equation}
where $k_i = \sqrt{2}^i$ is the characteristic wave number of the
$i$th subfield, $N$ is the number of subfields, and $E_0$ controls the
overall energy. The resulting spectrum exhibits
$E(k) \propto k^{-\xi}$ for $1 < k < k_N = 2^{N/2}$ and exponential
decay for $k > k_N$. 
We then choose prior
\begin{equation}
  \mu_0 = N(0,\tilde{E}), \quad \tilde{E}_{lm}=\frac{1}{2\pi \norm{\kbf_l}_2} E(\norm{\kbf_l}_2)\delta_{lm}
  \label{eq:kraichnan_prior}
\end{equation}
where $E$ is as defined in
\eqref{eq:kraichnanEnergy}, with $\xi=\frac{3}{2}$ motivated by
\cite{chen1998simulations,kraichnan1991stochastic}.
Then for $\vfield = \sum_\kbf \vfieldnobf_k \sim \mu_0$, the expected
energy associated with wave numbers of norm $k$ (integrating across the shell $S_k=\left\{ \mathbf{k} : \norm{\kbf}_2 = k \right\}$) is $E(k)$.

Note: The Kraichnan model of mixing typically involves a velocity
field with energy spectrum \eqref{eq:kraichnanEnergy} but that is
white ($\delta$-correlated) in time \cite{shraiman2000scalar}. Here
$\vfield$ is a background flow, i.e.~constant in time; we simply use
the Kraichnan model as motivation for the energy decay modeled in the
prior.

\subsubsection{Posterior Structure}

As described above, the output of a Bayesian inference is the
posterior $\mu_\data$, a probability measure on the space of
divergence-free vector fields $\vfspace$ or, in practice, on a
finite-dimensional approximation $\vfspace_N$ given by the truncated
expansion of the basis described in \cref{sec:modelDivFree}. To
approximate the exact posterior, we assembled a list of 10 million
samples by running a series of 40 pCN MCMC chains of 250,000 samples
each, with every chain beginning with an initial
sample randomly chosen from the prior.\footnote{pCN was chosen here because it provided a computationally-inexpensive approximation to what proved to be a posterior with simple structure.} \cref{fig:ex1_posterior} shows the structure of the computed posterior. The left-hand plot shows mean, variance, skew, and
excess kurtosis (kurtosis minus $3$) of the posterior by Fourier
component of $\vfield$, where $\vfield$,
incorporating the discretization \eqref{eq:Hs} and reality condition
\eqref{eq:reality}, is constructed from the components as:
\begin{equation}
  \begin{aligned}
    \vfield(\x) &= \left[ \vfieldnobf_0, \vfieldnobf_1 \right] 
    + \vfieldnobf_2 \left[ 0, \cos(2 \pi y) \right]
    + \vfieldnobf_3 \left[ 0, - \sin(2 \pi y) \right] 
    + \vfieldnobf_4 \left[ \cos(2 \pi x), 0 \right]
    + \vfieldnobf_5 \left[ -\sin(2 \pi x), 0 \right] \\
    &\quad + \vfieldnobf_6 \left[ 0, \cos(4 \pi y) \right]
    + \vfieldnobf_7 \left[ 0, - \sin(4 \pi y) \right]
    + \vfieldnobf_8 \left[ \cos(2 \pi x), \cos(2 \pi y) \right] \\
    &\quad + \vfieldnobf_9 \left[ -\sin(2 \pi x), - \sin(2 \pi y) \right] + \dots
  \end{aligned}
  \label{eq:vcomp}
\end{equation}
Because of the influence of the
prior measure, the mean and covariance of the posterior for
higher-order components tend to zero. Skew and excess kurtosis are
near zero (up to computational resolution) for all components,
indicating that the marginal distribution for each component is
approximately Gaussian. The right-hand plot in \cref{fig:ex1_posterior} presents one- and two-dimensional
histograms of the first eight components of $\vfield$. 
Note that the histograms (and other plots omitted for brevity) all show a contiguous mass of probability, indicating that one ``class'' of vector field matches both the prior and the data.

\begin{figure}[!htbp]
\centering
\begin{minipage}[b]{0.49\textwidth}
  \includegraphics[width=\textwidth]{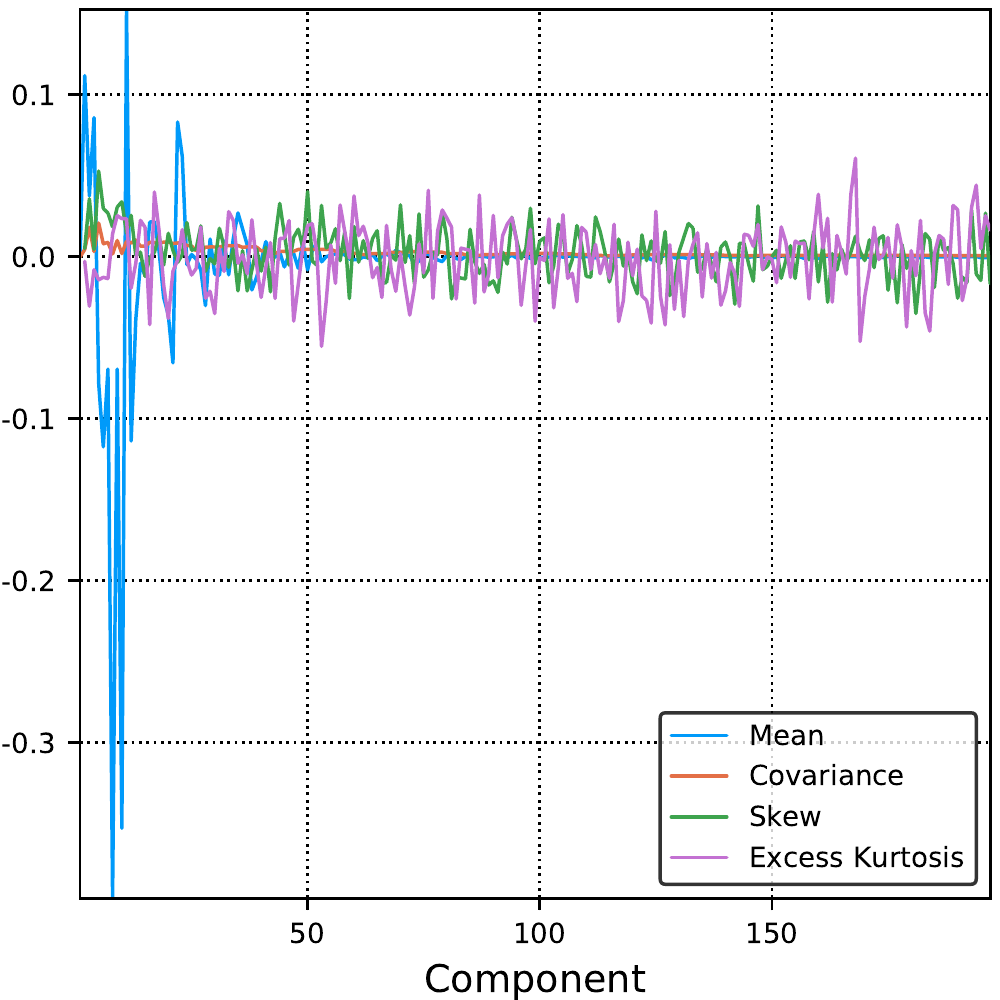}
\end{minipage}
\hfill
\begin{minipage}[b]{0.49\textwidth}
  \includegraphics[width=\textwidth]{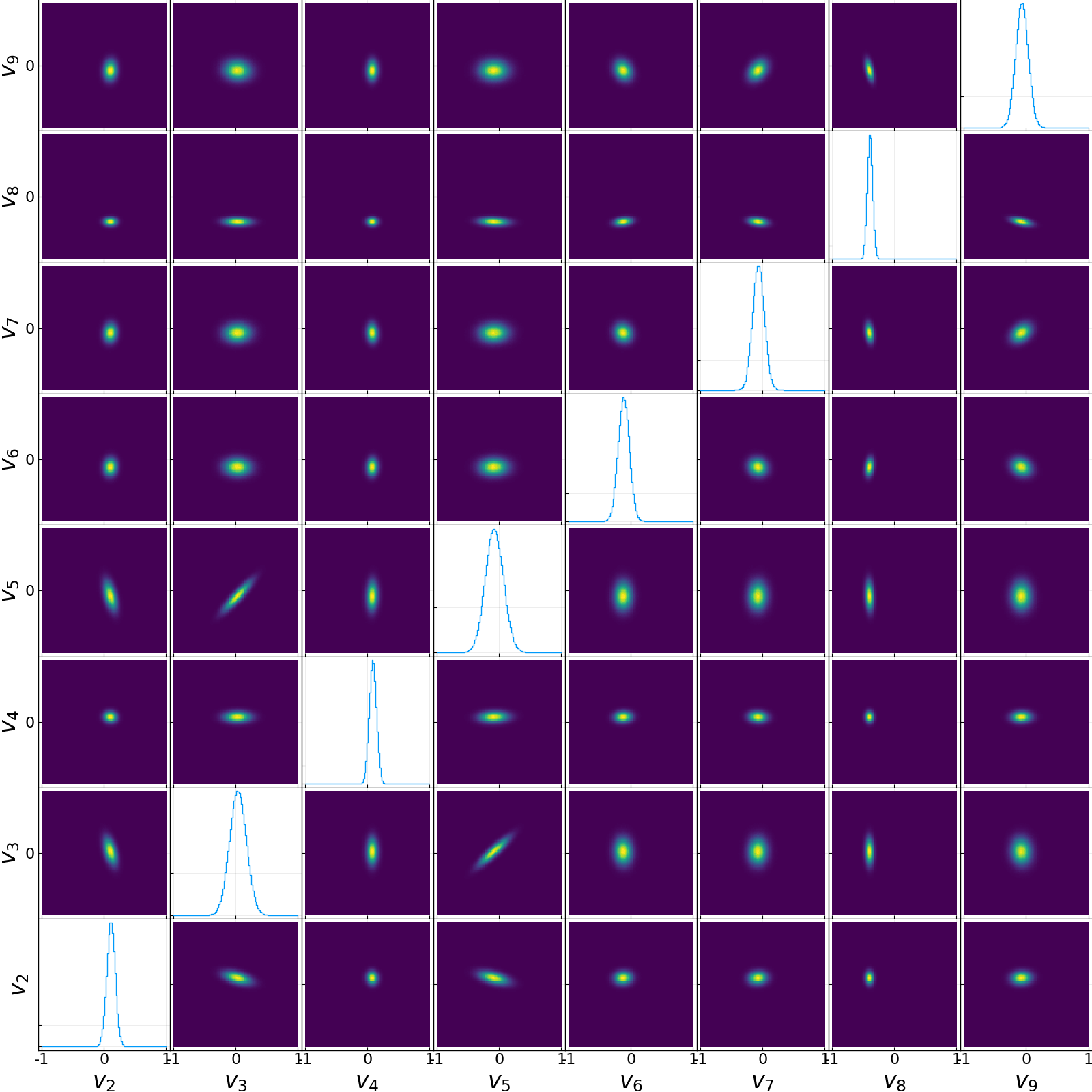}
\end{minipage}
  \caption{\label{fig:ex1_posterior} Structure of Posterior. Left: Mean, covariance, skew, and excess kurtosis of posterior measure, by component of $\vfield$. Right: Posterior ($\mu_\data$) 1D (diagonal) and 2D (off-diagonal) marginal distributions for the first eight components of $\vfield$ (out of 197).}
\end{figure}

\subsubsection{MCMC Sampling}\label{sec:ex1_sampling}

To test the behavior of ``inexpensive'' and ``expensive'' MCMC methods, both pCN and HMC (see \cref{sec:mcmc_stub}) were applied to Example 1. The pCN parameter $\beta=0.15$ was chosen to match the optimal acceptance rate of $23\%$ from \cite{roberts2001optimal}. For HMC, $\epsilon=0.125$ and $\tau=1$ were chosen because these values showed a good balance between the desire for high acceptance rate, large jumps between samples, and low computational cost in numerical experiments. The resulting acceptance rates were $23.9\%$ for pCN and $81.0\%$ for HMC. \cref{fig:ex1_misfit} shows the trace and autocorrelation of the potential $\Phi$ (see \cref{def:phi}). When pCN is applied to this example, we see ``random walk'' behavior -- the samples move about the posterior, but are correlated with each other. For HMC, the random walk effect is reduced and samples exhibit independence from each other after orders of magniture fewer iterations than for pCN. The HMC chain explores the posterior more quickly as a result. This is explored in the next section.
\begin{figure}[!htbp]
\centering
\begin{minipage}[b]{0.49\textwidth}
	\includegraphics[width=\textwidth]{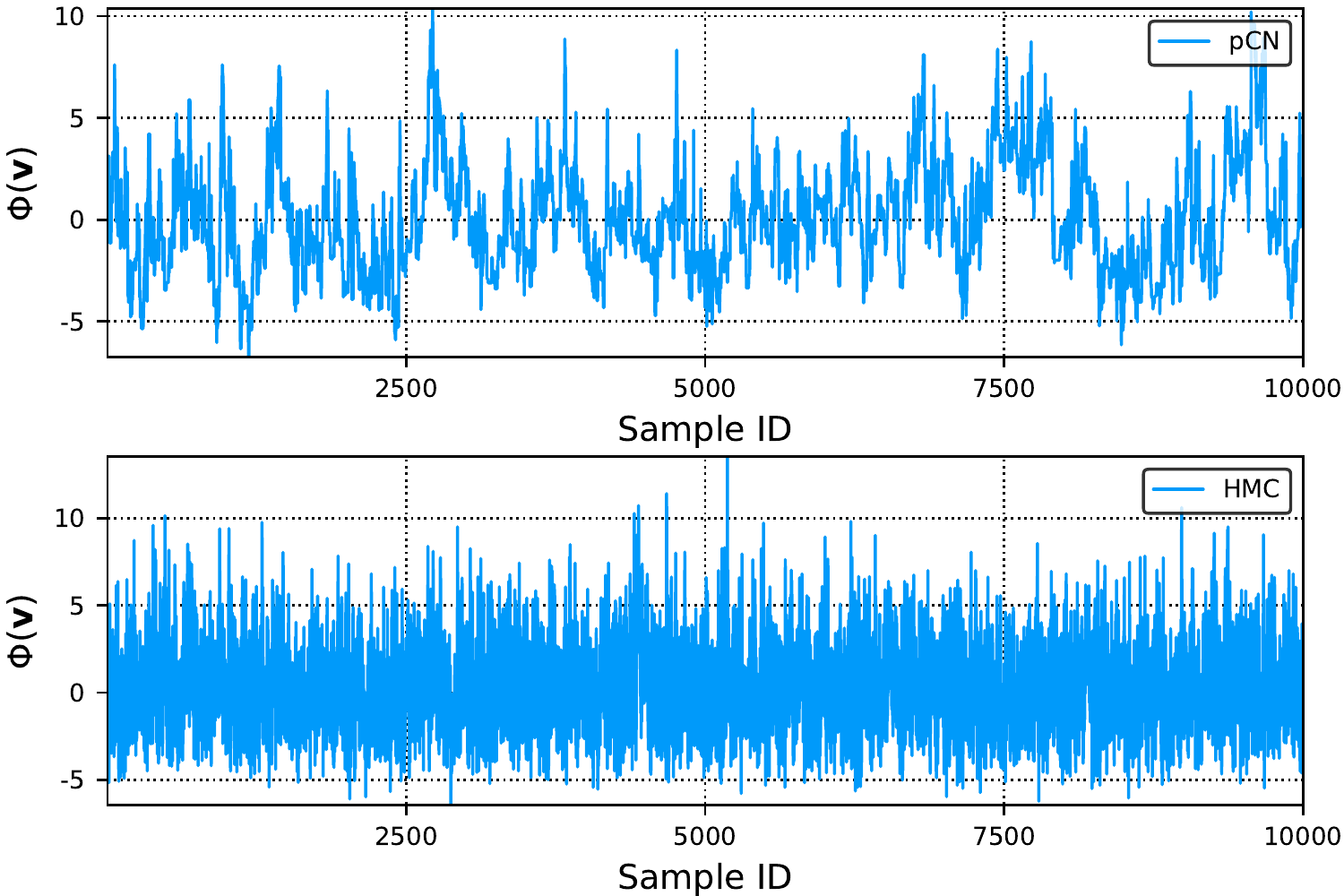}
\end{minipage}
\hfill
\begin{minipage}[b]{0.49\textwidth}
	\includegraphics[width=\textwidth]{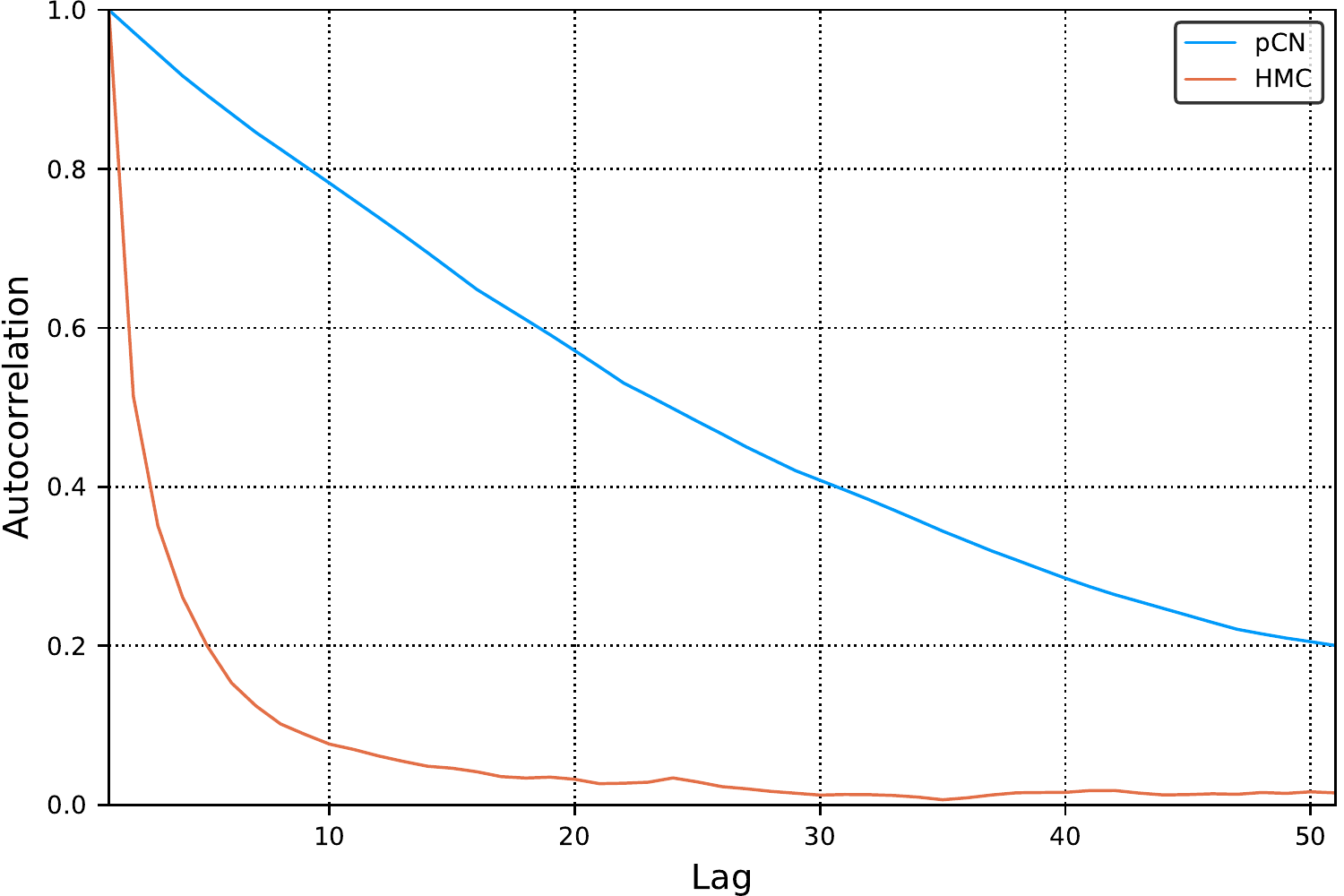}
\end{minipage}
\caption{Trace (left) and autocorrelation (right) of the potential $\Phi$.}
\label{fig:ex1_misfit}
\end{figure}

\subsubsection{Convergence of Measures}

The difference between the ``true'' (\cref{fig:ex1_posterior}) and computed marginal distributions can be evaluated via the total variation distance.\footnote{Note that the real desire would be to measure convergence in the full 197-dimensional sample space. In this paper, we will use the total variation norm to measure convergence between 1D and 2D distributions, a necessary and easy to picture -- but not sufficient -- requirement for convergence in the full-dimensional space.} Convergence of the MCMC chains to the true marginal distributions are shown in \cref{fig:ex1_totVarEvolve}. The figure shows that HMC achieves a close approximation to the posterior marginal distributions within a few hundred iterations, while similar convergence takes about an order of magnitude longer for pCN.

\begin{figure}[!htbp]
  \centering
\begin{minipage}[b]{0.49\textwidth}
	\includegraphics[width=\textwidth]{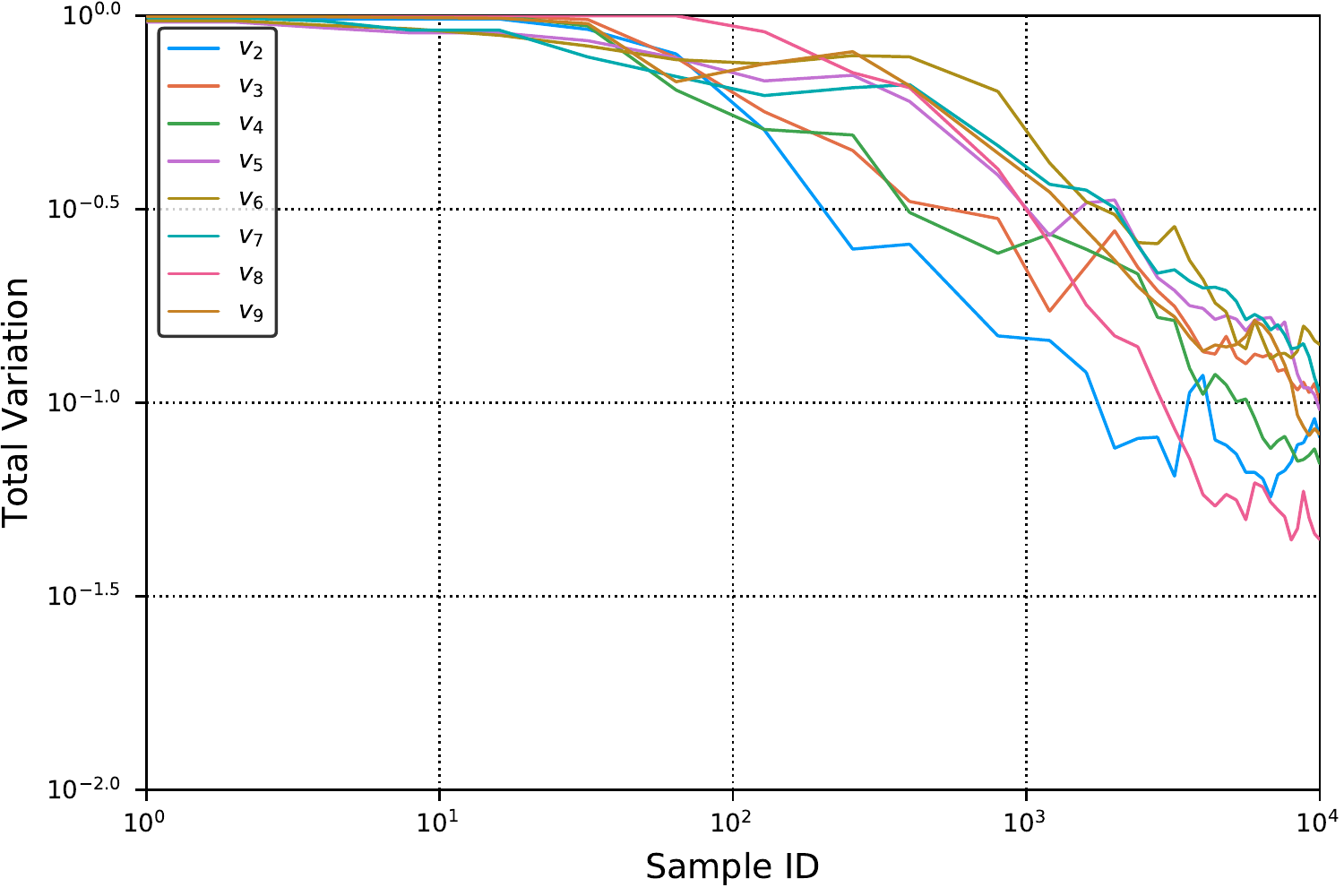}
\end{minipage}
\hfill
\begin{minipage}[b]{0.49\textwidth}
	\includegraphics[width=\textwidth]{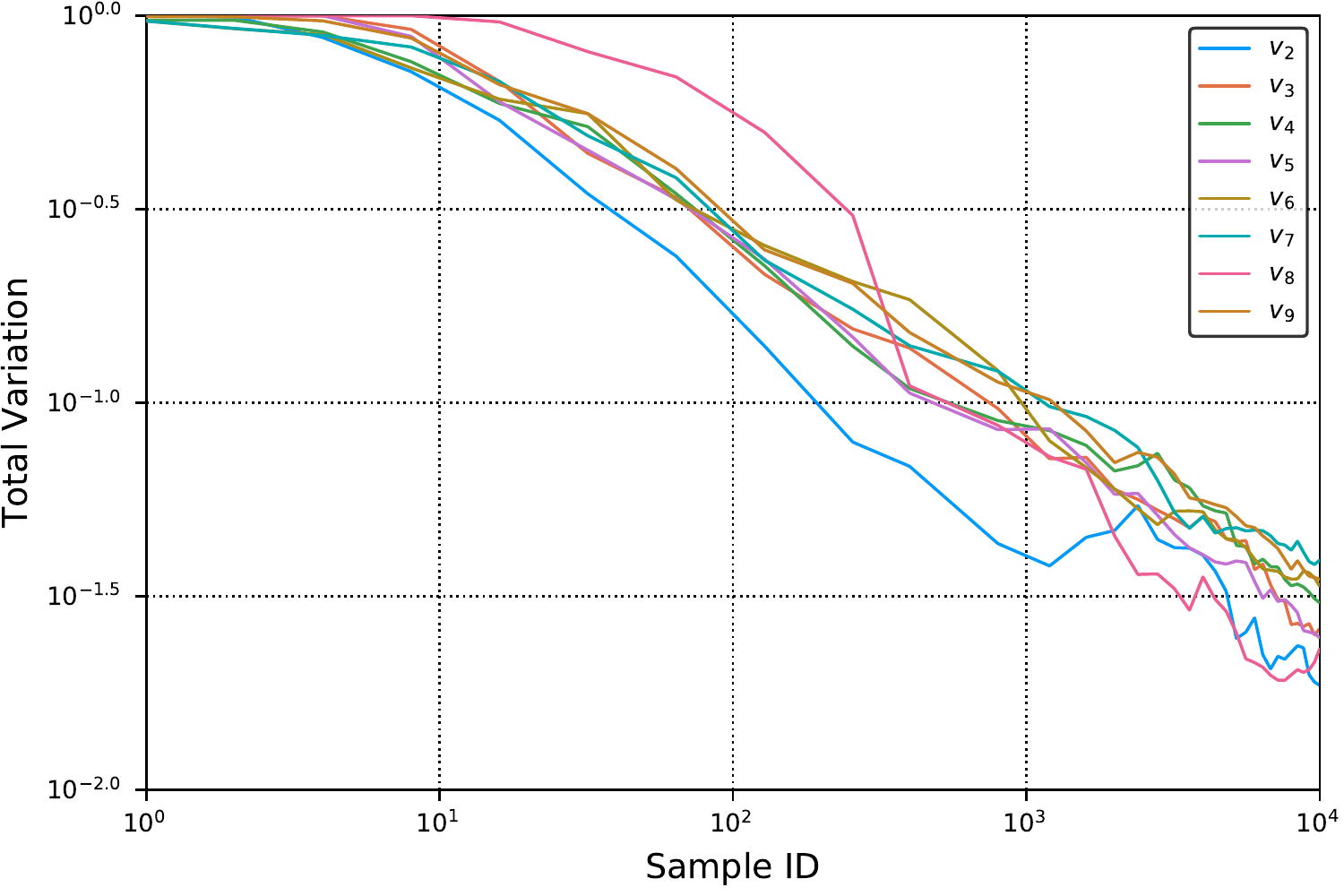}
\end{minipage}
\caption{\label{fig:ex1_totVarEvolve}Total variation norm between computed and ``true'' marginal probability density function for $\vfieldnobf_2,\dots,\vfieldnobf_9$ for 10,000 samples. Left: pCN, Right: HMC.}
\end{figure}

\subsubsection{Equal Runtime Comparison} \label{sec:ex1_runtime}

Recall from \cref{sec:ex1_sampling} that the parameters used for HMC
were $\epsilon=0.125$ and $\tau=1.0$, meaning that
$\frac{\tau}{\epsilon}=8$ PDE and adjoint solves (see
\cref{alg:mcmchmc}) were required per HMC sample. Because of these
solves and the additional costs required for the gradient computation
(see \cref{alg:adjoint}) and time integration, each HMC sample took the time of approximately
$39$ pCN samples to compute. Thus, we can
reweight pCN samples by $39$ to get a
comparison of the sampling accuracy per computational unit. 
\cref{fig:ex1_totVarEvolve_runTime} shows the convergence of
total variation norm for chains of runtime equal to
10,000 samples of HMC; the results can be compared with
\cref{fig:ex1_totVarEvolve}. We see that chains of equal runtime are largely
equivalent between the two methods when applied to Example 1; the faster convergence of HMC is essentially balanced
by the larger amount of computation required to generate
the samples. 

\begin{figure}[!htbp]
  \centering
\begin{minipage}[b]{0.49\textwidth}
	\includegraphics[width=\textwidth]{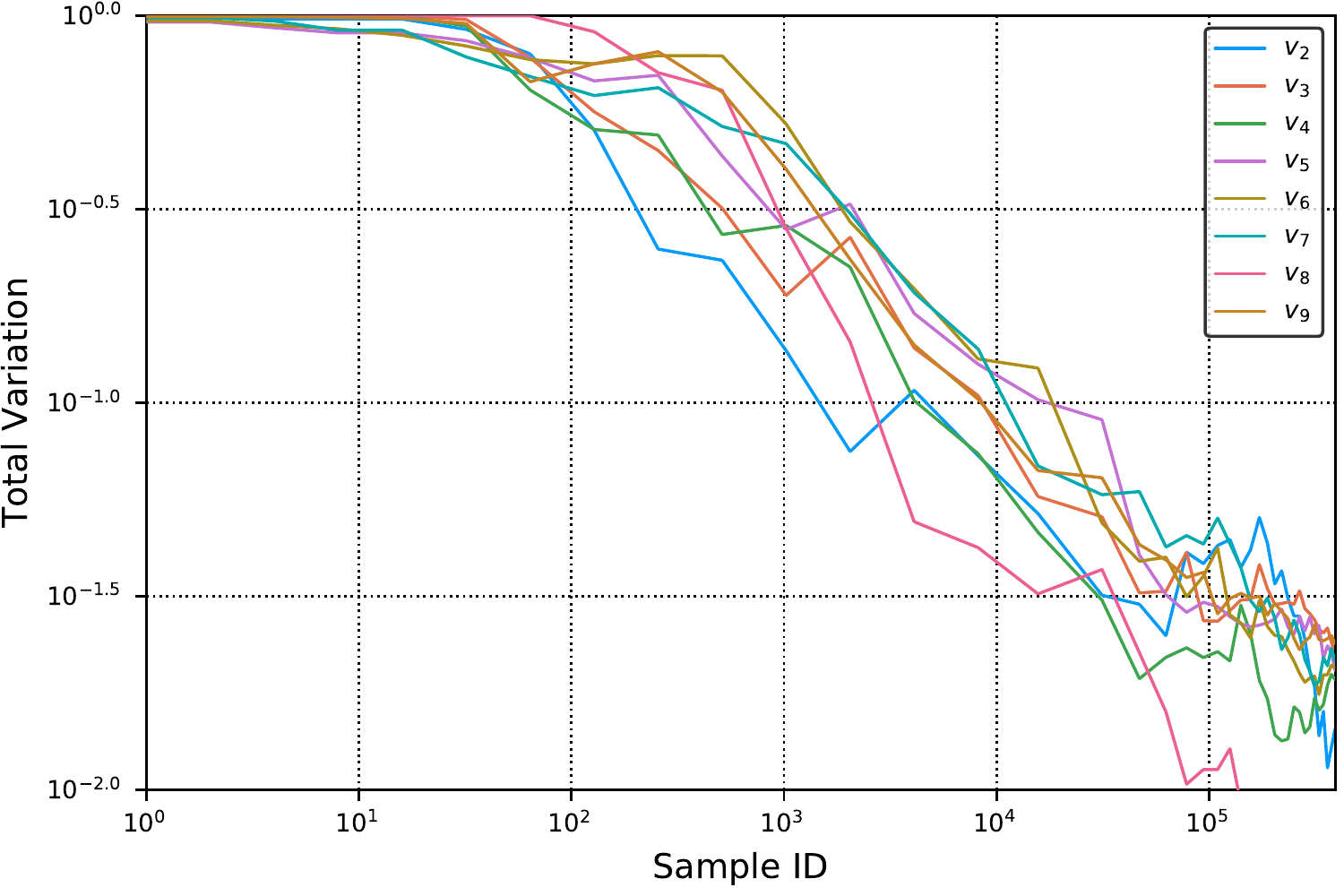}
\end{minipage}
\hfill
\begin{minipage}[b]{0.49\textwidth}
	\includegraphics[width=\textwidth]{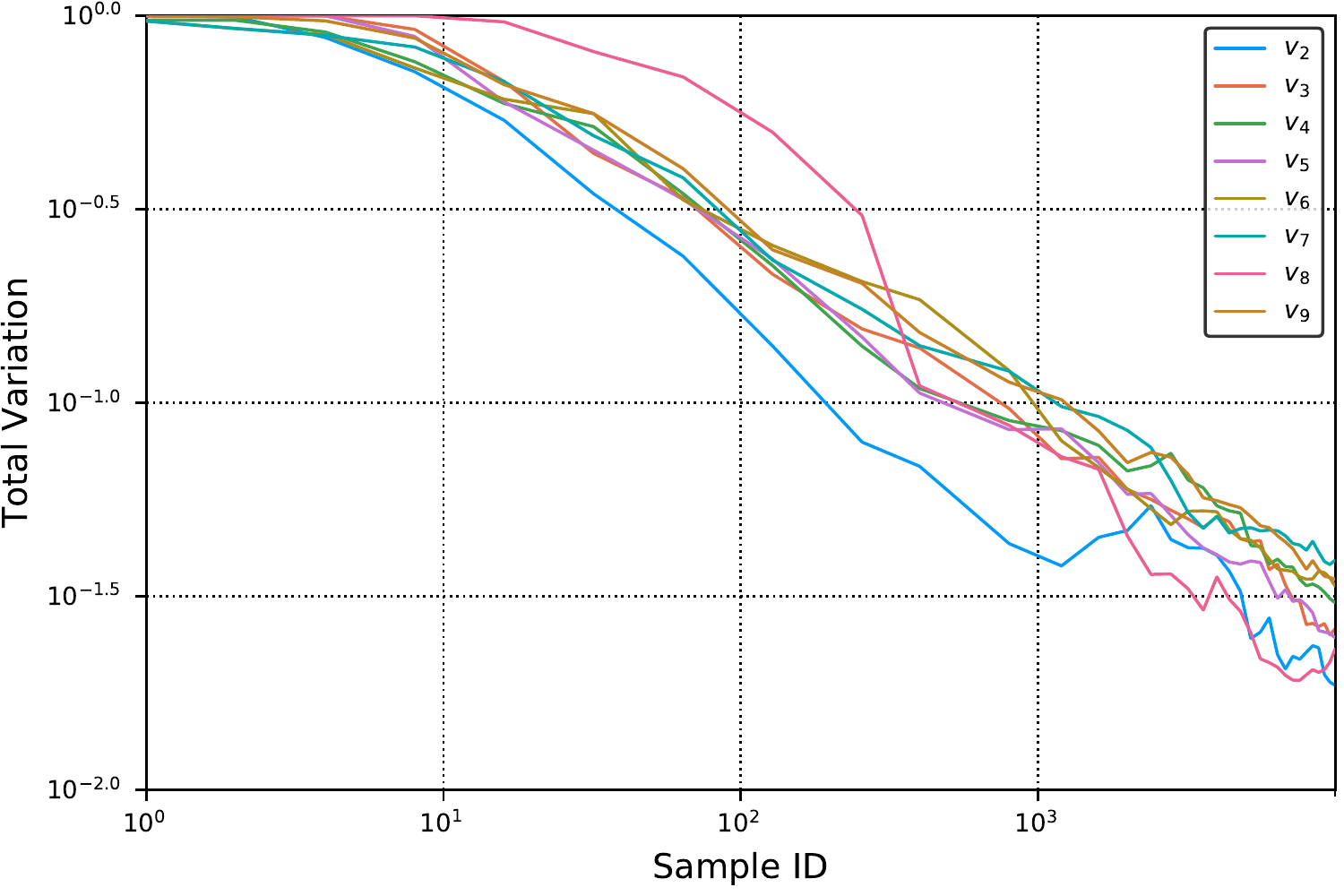}
\end{minipage}

\caption{\label{fig:ex1_totVarEvolve_runTime}Total variation norm between computed and ``true'' marginal probability density function for $\vfieldnobf_2,\dots,\vfieldnobf_9$ for runtime equivalent to 10,000 HMC samples. Left: pCN, Right: HMC.}
\end{figure}

\subsection{Example 2: Multimodal Posterior}\label{sec:ex2_multihump}
In this section we present an example where the prior and data
interact to produce a posterior with multiple regions of mass;
posteriors of this kind are difficult for MCMC methods to resolve
because chains have trouble jumping between the wells.  We take the
initial condition $\pdesol_0(\x) = \frac{1}{2} - \frac{1}{4}\cos 2\pi x - \frac{1}{4}\cos 2\pi y$ and true background flow $\vtrue=\left[ 8\cos 2\pi y , 8\cos 2\pi x \right]$. 
Symmetry guarantees that for $\x_1=[0,0]$ and $\x_2=[\half,\half]$ we have $\pdesol(\vtrue,t,\x_i)=\pdesol(-\vtrue,t,\x_i), i=1,2$. 
(In fact, there are more points for which this is true; however, two points suffice for the purposes of this example.) We therefore let the data $\data$ be point measurements $\pdesol(t,\x)$ from
$t=0.001$ to $0.050$ in intervals of $0.001$ at each of $\x_1$ and
$\x_2$. Then we have
$\Phi(\vtrue)=\Phi(-\vtrue)$, i.e.~both $\vtrue$ and $-\vtrue$ match
the data equally well. Finally, we use the mean-zero Kraichnan prior 
\eqref{eq:kraichnan_prior}, which assigns the same probability to both $\vtrue$ and $-\vtrue$.
The problem parameters for this example are listed in \cref{tab:ex2_parameters}.

Since both $\vtrue$ and $-\vtrue$ are given the same probability by both the prior and the
data, they will be equally likely according to the posterior. We show in the next section that the symmetry in the
problem setup results in multiple distinct probability masses in the
posterior.

\renewcommand{\arraystretch}{1.4}
\begin{table}[htbp]
  {\footnotesize
  \caption{Problem parameters for Example 2.} \label{tab:ex2_parameters}
  \centering
  \begin{tabular}{|>{\raggedright}p{0.125\textwidth}|p{0.3\textwidth}|>{\raggedright}p{0.125\textwidth}|p{0.3\textwidth}|} \hline
    Parameter & Value &
    Parameter & Value \\
    \hline\hline
    Observation operator, $\Obs$ & Point observations at $\x_1=[0,0]$ and $\x_2=[\half,\half]$ &
    Data, $\data$ & $\G(\vtrue)$ \\\hline
    Prior, $\mu_{0}$ & Kraichnan \eqref{eq:kraichnan_prior} &
    Noise, $\noisemeasure$ & $N(0,\sigma_{\noise}^2 I)$, $\sigma_{\noise}=2^{-3}$ \\\hline
    True flow, $\vtrue$ & $\left[8\cos 2\pi y, 8\cos 2\pi x \right]$ &
    Sampling space, $\vfspace_N$ & $\norm{\kbf}_2\le 8$ (197 components) \\\hline
    Diffusion, $\conductivity$ & $3 \times 10^{-5}$ \cite[Table I]{chen1998simulations} &
    $\pdesol_0$ & $\half-\frac{1}{4}\cos(2 \pi x)-\frac{1}{4}\cos(2 \pi y)$ \\\hline
  \end{tabular}
  }
\end{table}

\subsubsection{Posterior Structure} \label{sec:ex2_posterior} As in
Example 1, we approximate the exact posterior via a large number of
samples; in this case we use 500,000 samples generated from 100 HMC
chains of 5,000 samples apiece, each beginning with an initial sample randomly chosen from the prior measure. We chose HMC chains because they provided better
convergence to the posterior than the other methods, as we describe
below. \cref{fig:ex2_posterior} shows the resulting posterior structure. The left plot shows the computed mean, variance, skew,
and excess kurtosis of the posterior, by Fourier component of
$\vfield$. We note that, due to the influence of the prior
measure, the mean and covariance of the posterior tend to zero for
higher-order components. Also, the deviations of excess kurtosis from
zero indicate the presence of highly non-Gaussian marginal
distributions for some components. The plot on the right presents one- and two-dimensional histograms for the first few
components of $\vfield$ (see the expansion in \eqref{eq:vcomp} for interpretation of
the components). Note that the symmetry of the problem results in
multiple large modes both $\vfieldnobf_2$ and $\vfieldnobf_4$, as well
as in several smaller bumps in the distributions of the other
components.

\begin{figure}[!htbp]
\centering
\begin{minipage}[b]{0.49\textwidth}
  \includegraphics[width=\textwidth]{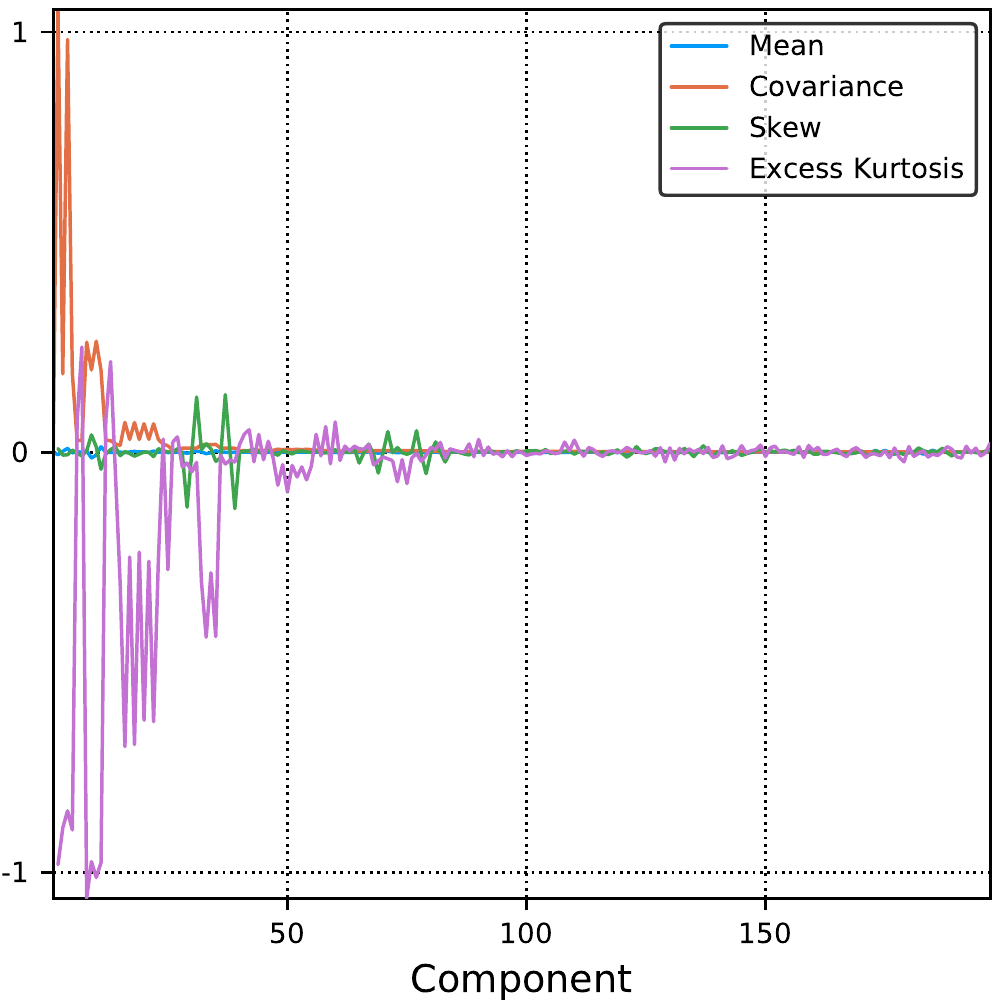}
\end{minipage}
\hfill
\begin{minipage}[b]{0.49\textwidth}
  \includegraphics[width=\textwidth]{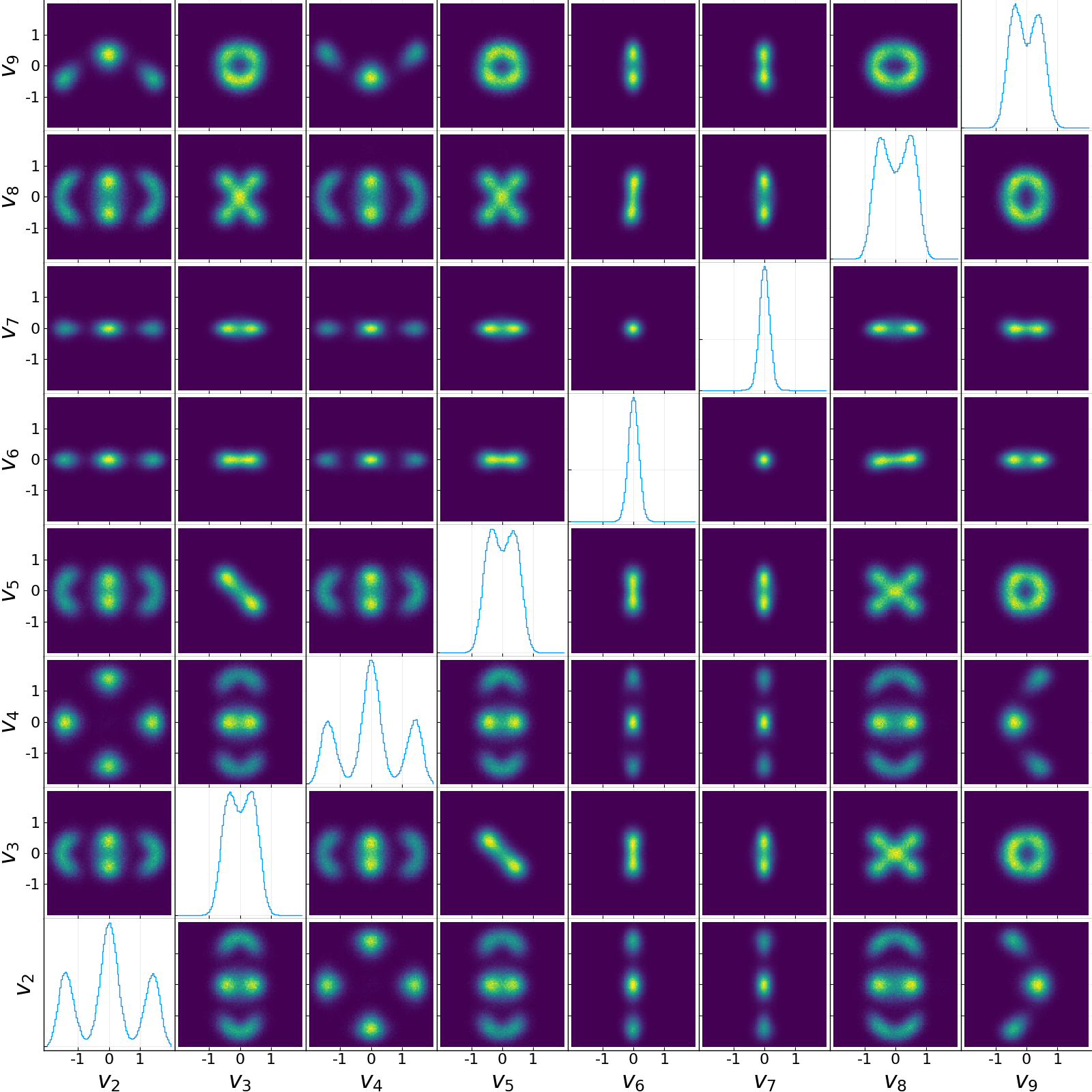}
\end{minipage}
  \caption{\label{fig:ex2_posterior} Structure of Posterior. Left: Mean, covariance, skew, and excess kurtosis of posterior measure, by component of $\vfield$. Right: Posterior ($\mu_\data$) 1D (diagonal) and 2D (off-diagonal) marginal distributions for the first eight components of $\vfield$ (out of 197).}
\end{figure}

Moreover, in constrast to Example 1, the two-dimensional histograms -- the approximate posterior joint probability density of pairs of vector field components -- show that the vector field components are highly correlated with each other (see, e.g., the ``X'' shape  between $\vfieldnobf_3$ and $\vfieldnobf_8$). 
It is worth noting that the posterior contains these correlation structures even though the prior assumes independence of the components.

Finally, it is worth noting that not all observables of the posterior
exhibit complicated structures. \cref{fig:ex2_vort_hist2d} shows
the computed posterior one- and two-dimensional histograms of
background flow vorticity at nine observation locations. The
one-dimensional histograms are simple -- i.e., nearly Gaussian -- at
each point. However, the two-dimensional histograms (except at the
center point, $\x_5$), exhibit multiple modes of different shapes.
\begin{figure}[!htbp]
\centering
\begin{minipage}[b]{0.49\textwidth}
  \includegraphics[width=\textwidth]{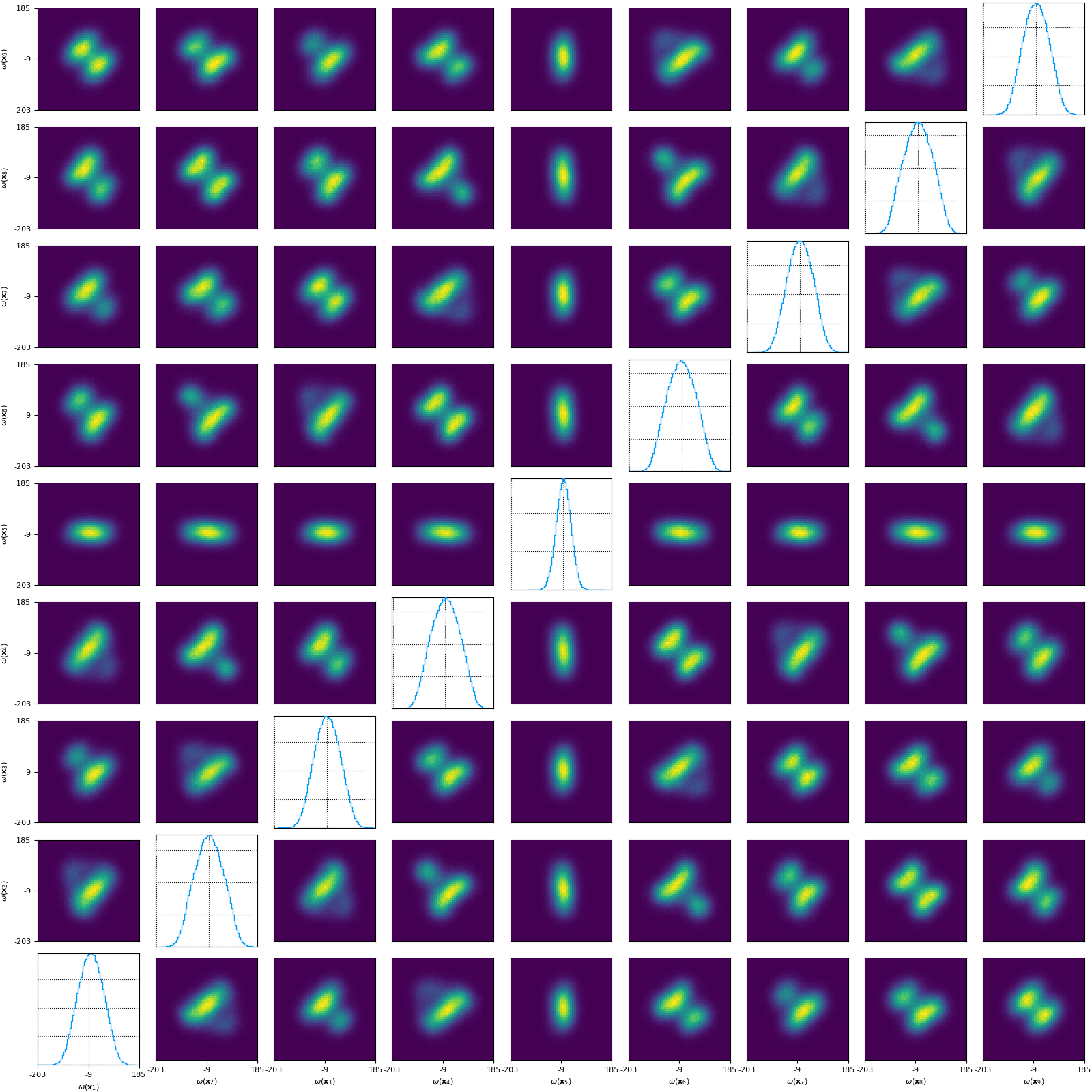}
\end{minipage}
\hfill
\begin{minipage}[b]{0.49\textwidth}
  \includegraphics[width=\textwidth]{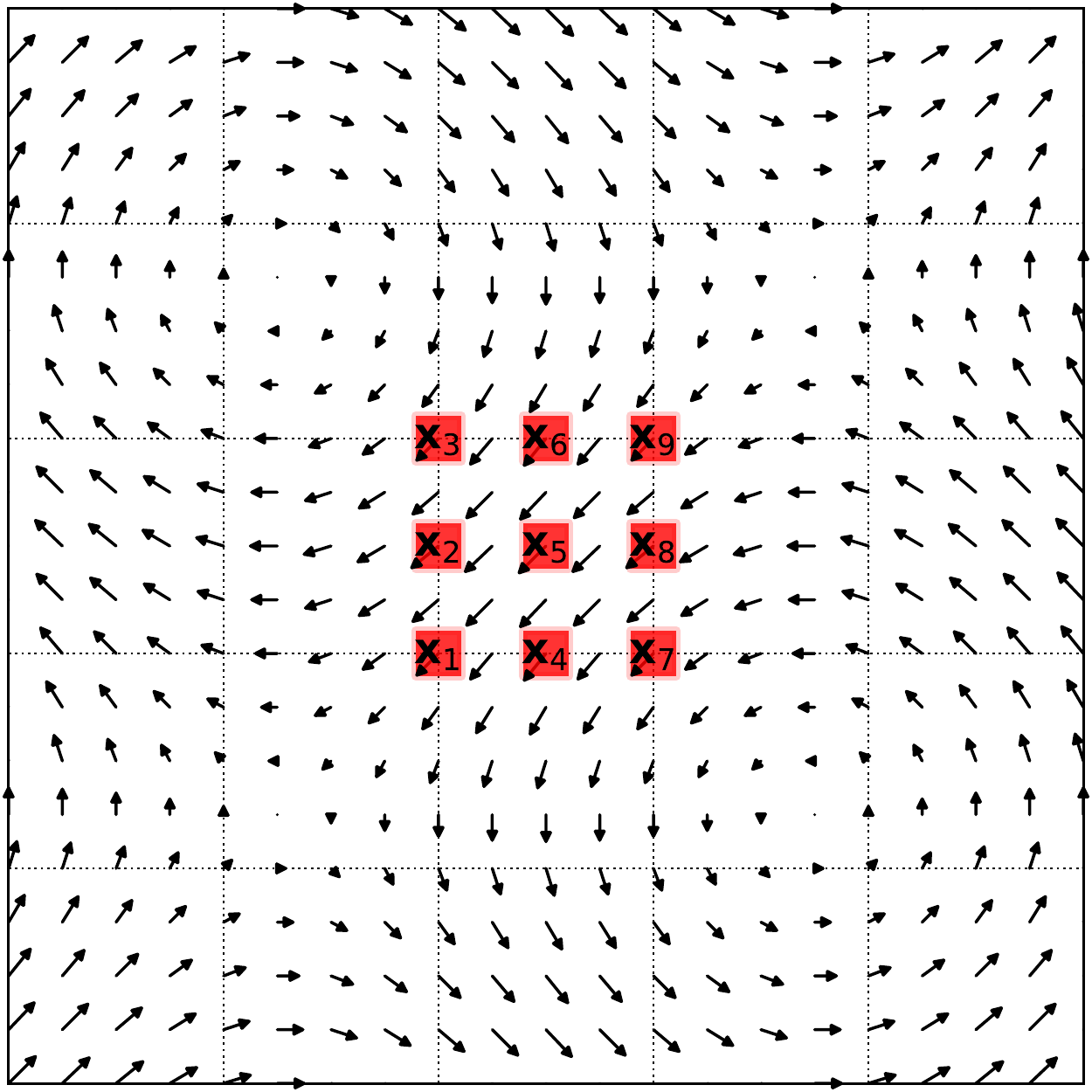}
\end{minipage}
\caption{\label{fig:ex2_vort_hist2d}Posterior one- and two-dimensional histograms of vorticity (left) at nine observation points (shown against $\vtrue$, right).}
\end{figure}

\subsubsection{MCMC Sampling} \label{sec:ex2_sampling} 
The multimodal structure of the posterior is typical of distributions that are difficult for MCMC methods to resolve efficiently, as the chains have difficulty moving across the regions of low probability between the regions of mass. We now use this structure to test the viability of pCN and HMC (see \cref{sec:mcmc_stub}) in resolving complicated posteriors. The tests use parameter values of $\beta = 0.2$ for pCN (again corresponding to the optimal acceptance rate of $23\%$ from \cite{roberts2001optimal}) and $\epsilon=0.125$ and $\tau=4$ for HMC, which in numerical experiments showed good convergence behavior.

\cref{fig:ex2_misfit} shows the trace and autocorrelation of the potential $\Phi$ (see \cref{def:phi}) for MCMC sampling of Example 2. As in Example 1, we see ``random walk'' behavior for pCN, whereas for HMC many fewer iterations are required to achieve statistical independence between samples. Unlike Example 1, however, the HMC chains for Example 2 exhibit negative autocorrelation between consecutive samples. The author plans to investigate this phenomenon, which may be related to the multimodal structure of the problem, in later work.
\begin{figure}[!htbp]
\centering
\begin{minipage}[b]{0.49\textwidth}
	\includegraphics[width=\textwidth]{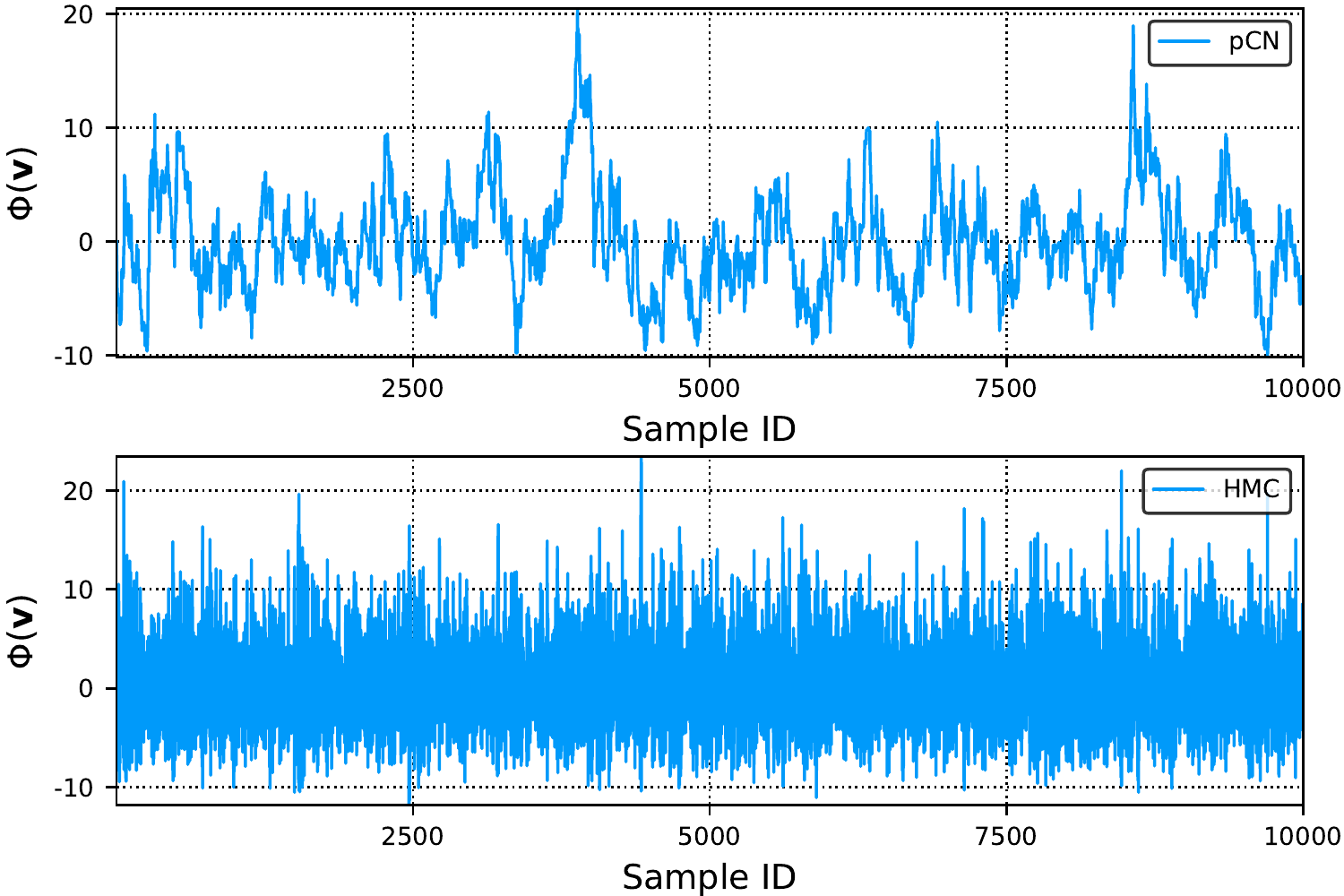}
\end{minipage}
\hfill
\begin{minipage}[b]{0.49\textwidth}
	\includegraphics[width=\textwidth]{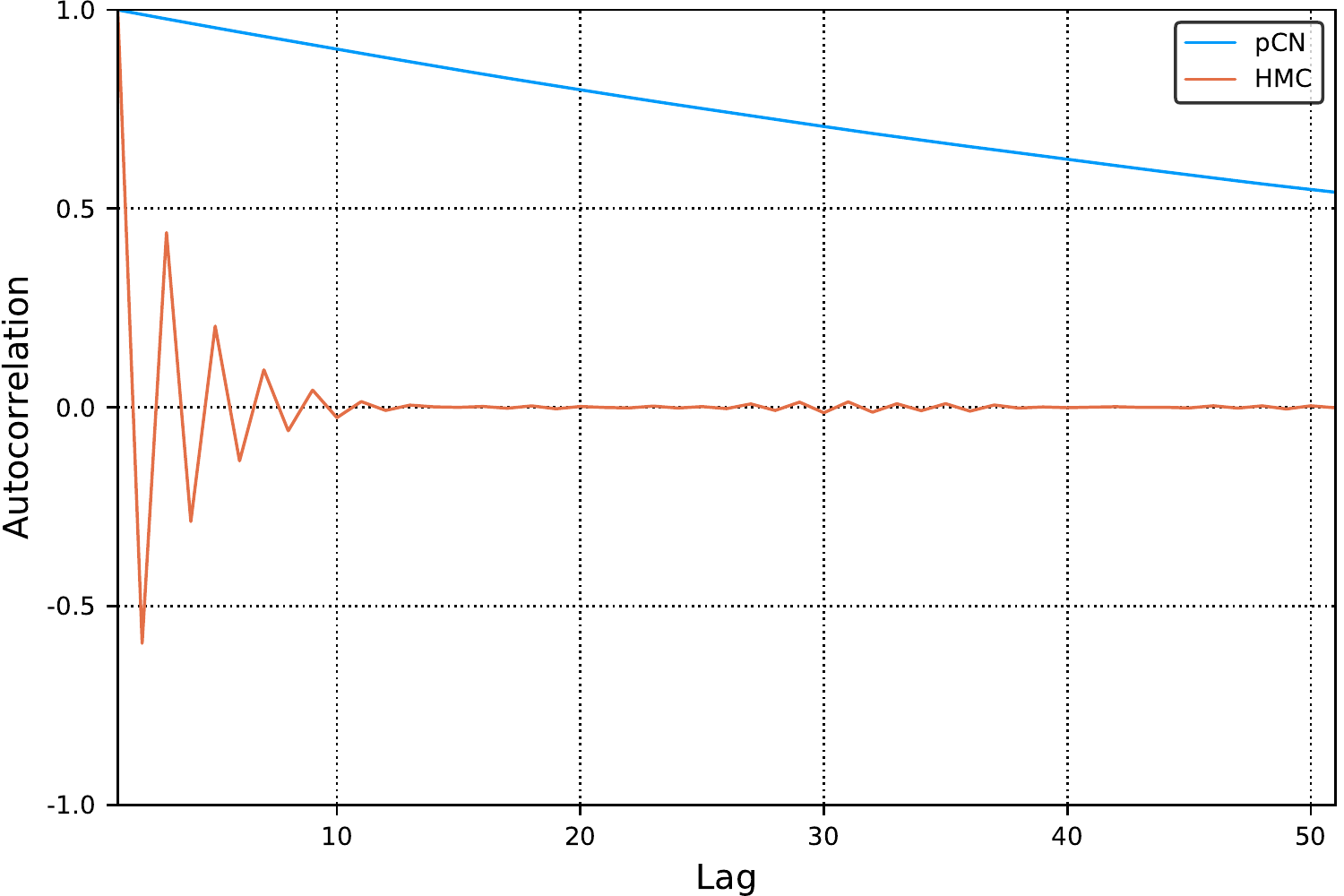}
\end{minipage}
\caption{Trace (left) and autocorrelation (right) of the potential $\Phi$.}
\label{fig:ex2_misfit}
\end{figure}

We can also see the contrast between pCN and HMC in the traces of vector field components shown in \cref{fig:ex2_trace}. The pCN samples move within a relatively limited range (a single probability mass), while the HMC samples occasionally jump between the different probability regions. In parameter testing, we observed that the frequency of these jumps increased roughly linearly with $\tau$ ($\tau=4$ produced twice as many jumps as $\tau=2$, for example) because longer integration times allowed the Hamiltonian system to evolve further, overcoming the areas of low probability that separate the regions of mass.

\begin{figure}[!htbp]
  \centering
\begin{minipage}[b]{0.49\textwidth}
	\includegraphics[width=\textwidth]{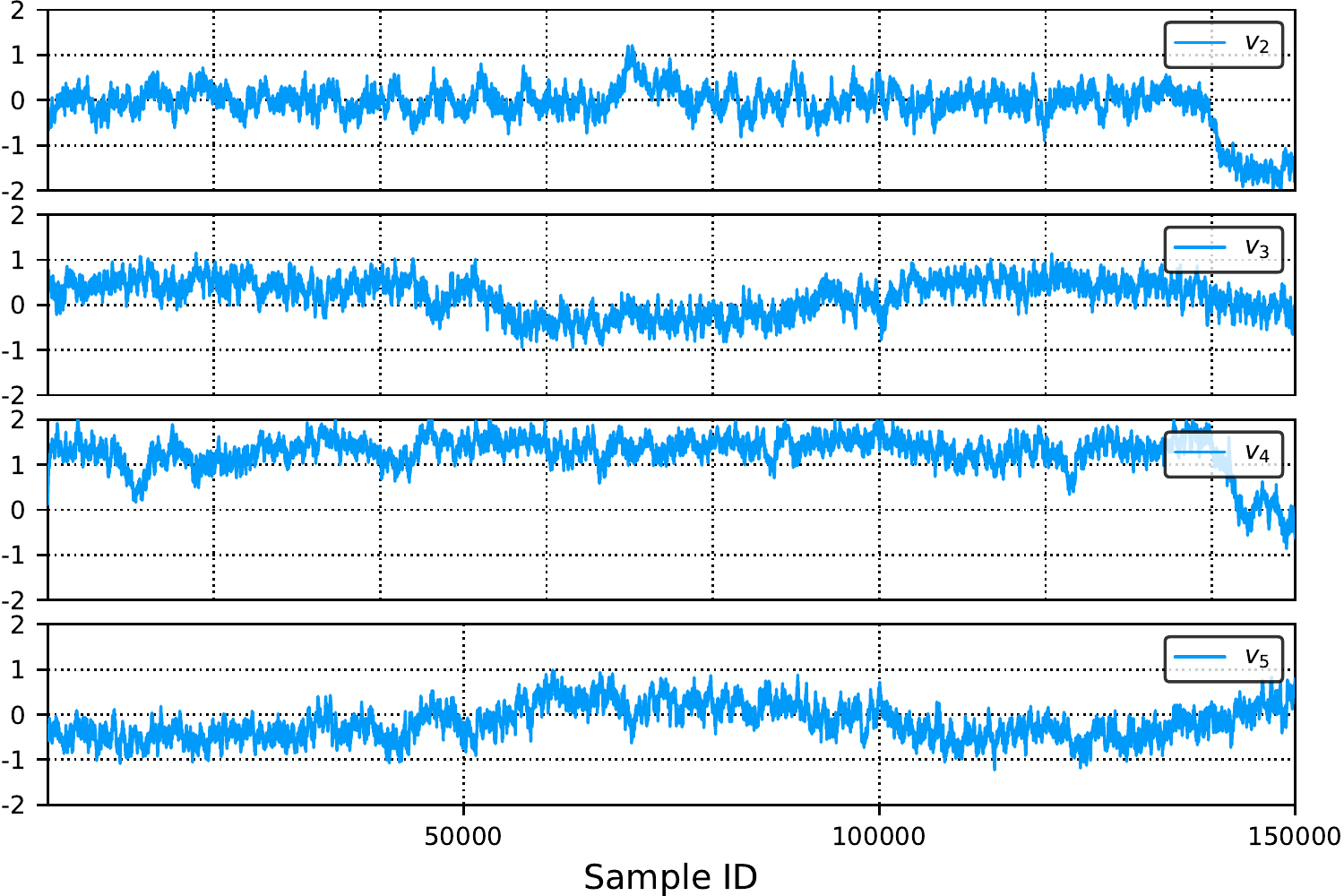}
\end{minipage}
\hfill
\begin{minipage}[b]{0.49\textwidth}
	\includegraphics[width=\textwidth]{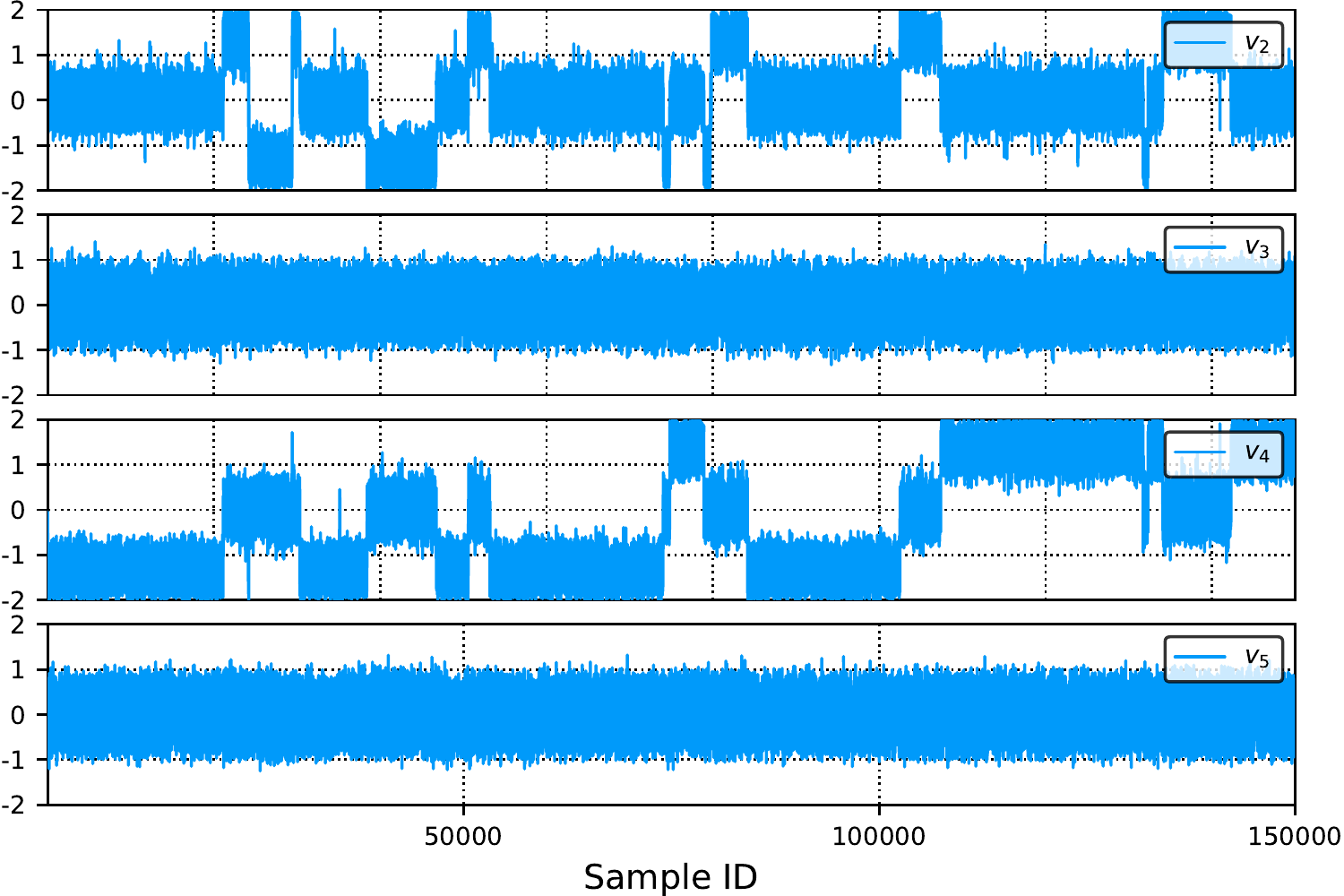}
\end{minipage}
\caption{\label{fig:ex2_trace}Trace of $\vfieldnobf_2,\dots,\vfieldnobf_5$ by sample number, pCN (left) and HMC (right).}
\end{figure}

\subsubsection{Convergence of Measures}
We now consider how the computed probability density functions
(normalized histograms) compare for each of the MCMC
methods. \cref{fig:ex2_posterior_method} shows the computed one- and two-dimensional distributions for the first few vector field components after 150,000
samples. This figure can be compared with the ``true'' 
distributions shown in \cref{fig:ex2_posterior}. The histograms for pCN show only some of the many distinct probability regions in the posterior, as the chain failed to jump across the regions of low probability. The
distributions for $\vfieldnobf_2$ and $\vfieldnobf_4$, in particular,
only show one or two of the three modes shown in
\cref{fig:ex2_posterior}. The histograms for HMC, by contrast, resolve all major features in the posterior, though the some of the features still exhibited imbalance when the chain terminated. The asymmetry in $\vfieldnobf_4$ and
$\vfieldnobf_9$ indicates that the chain has perhaps not fully
converged yet.

As in Example 1, we can get a feel for convergence of the methods by computing the total variation distance
between the ``true'' $\mu$ and $N$-sample
$\mu^{(N)}$ marginal distributions. The results are shown in
\cref{fig:ex2_totVarEvolve}. We note the ``sawtooth'' behavior for HMC, as the number of samples in each mode
of, for example, $\vfieldnobf_4$ slowly balances.

\begin{figure}[!htbp]
  \centering
\begin{minipage}[b]{0.49\textwidth}
	\includegraphics[width=\textwidth]{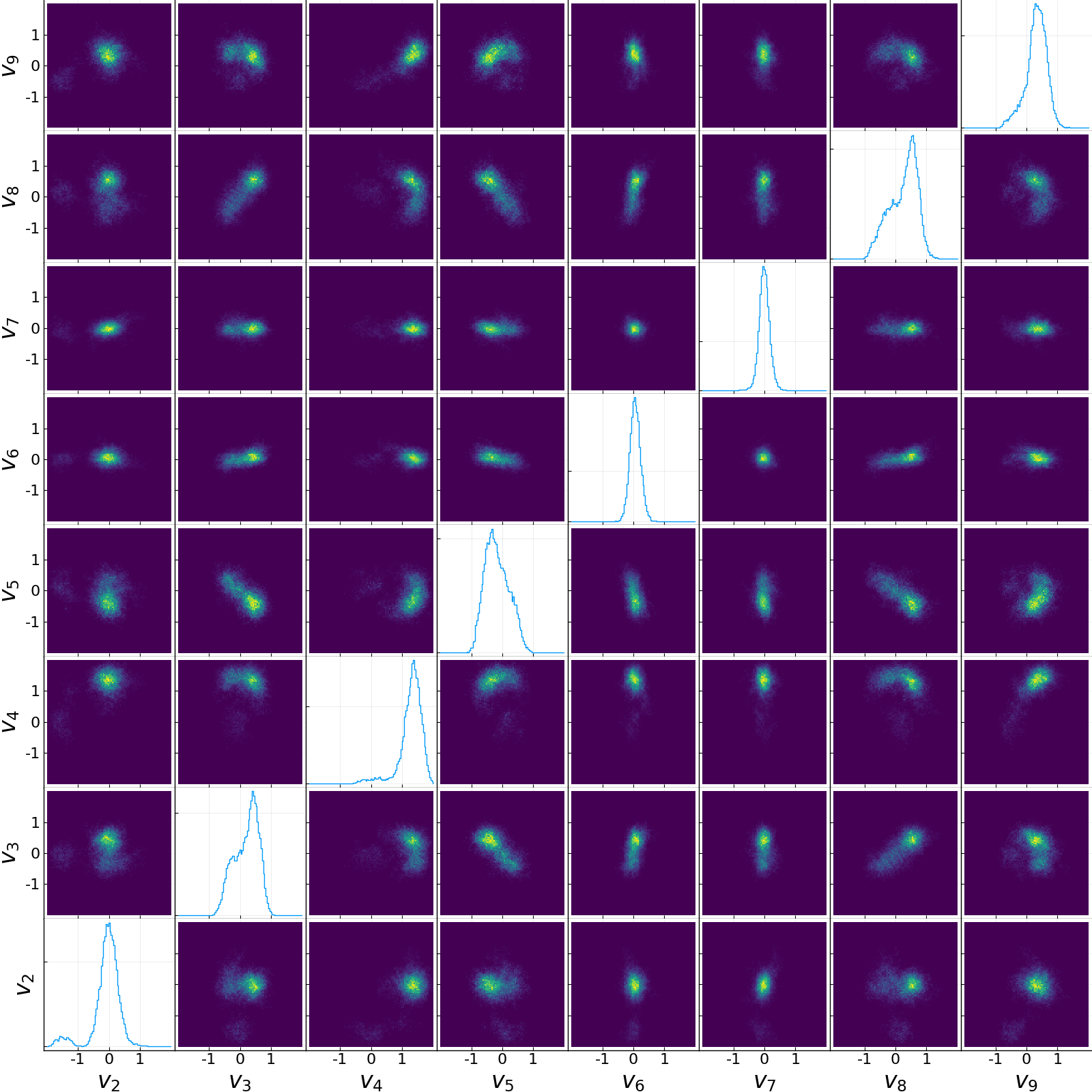}
\end{minipage}
\hfill
\begin{minipage}[b]{0.49\textwidth}
	\includegraphics[width=\textwidth]{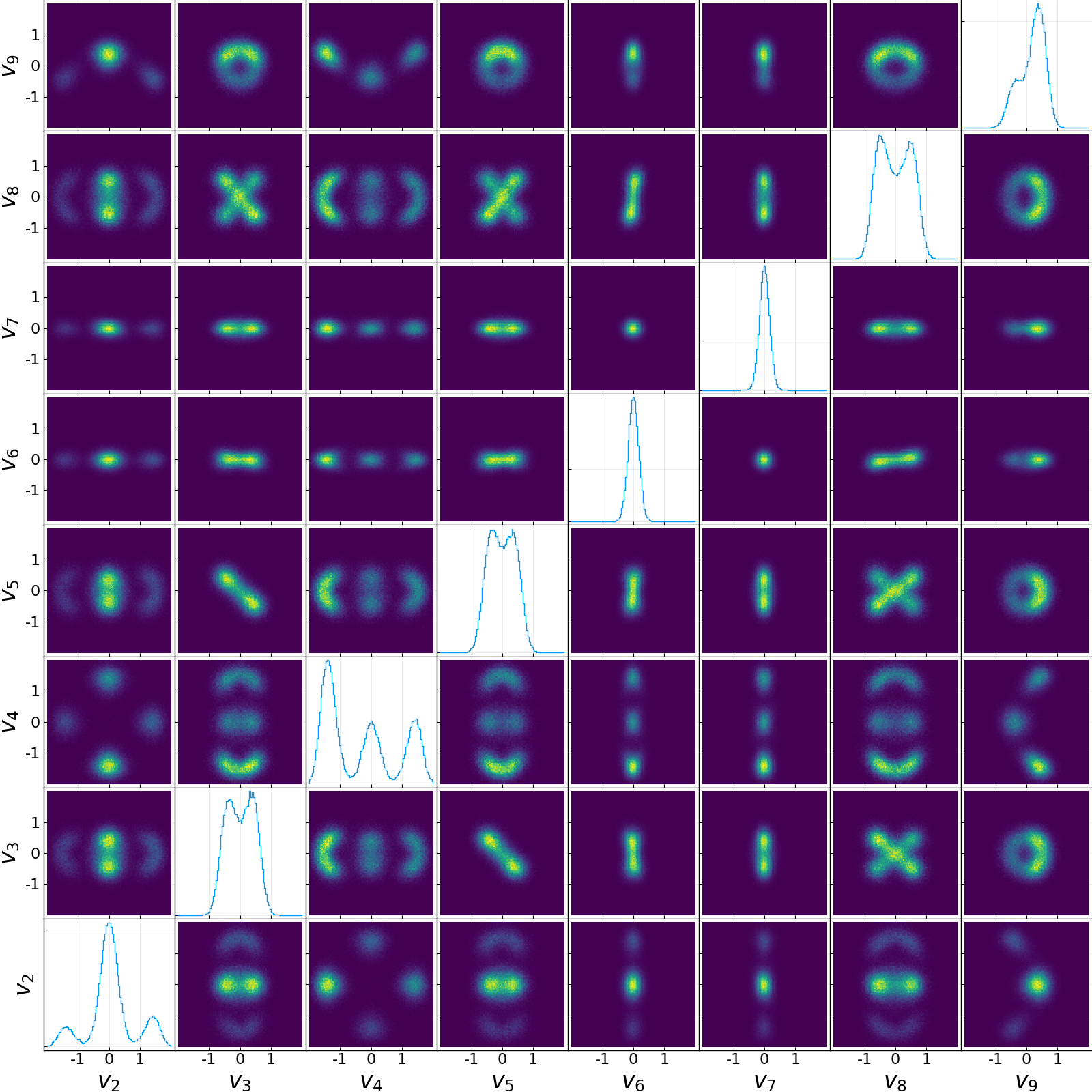}
\end{minipage}
\caption{\label{fig:ex2_posterior_method}Computed 1D and 2D marginal distributions for each of the first eight vector field components (out of 197) for 150,000 samples, pCN (left) and HMC (right).}
\end{figure}

\begin{figure}[!htbp]
  \centering
\begin{minipage}[b]{0.49\textwidth}
	\includegraphics[width=\textwidth]{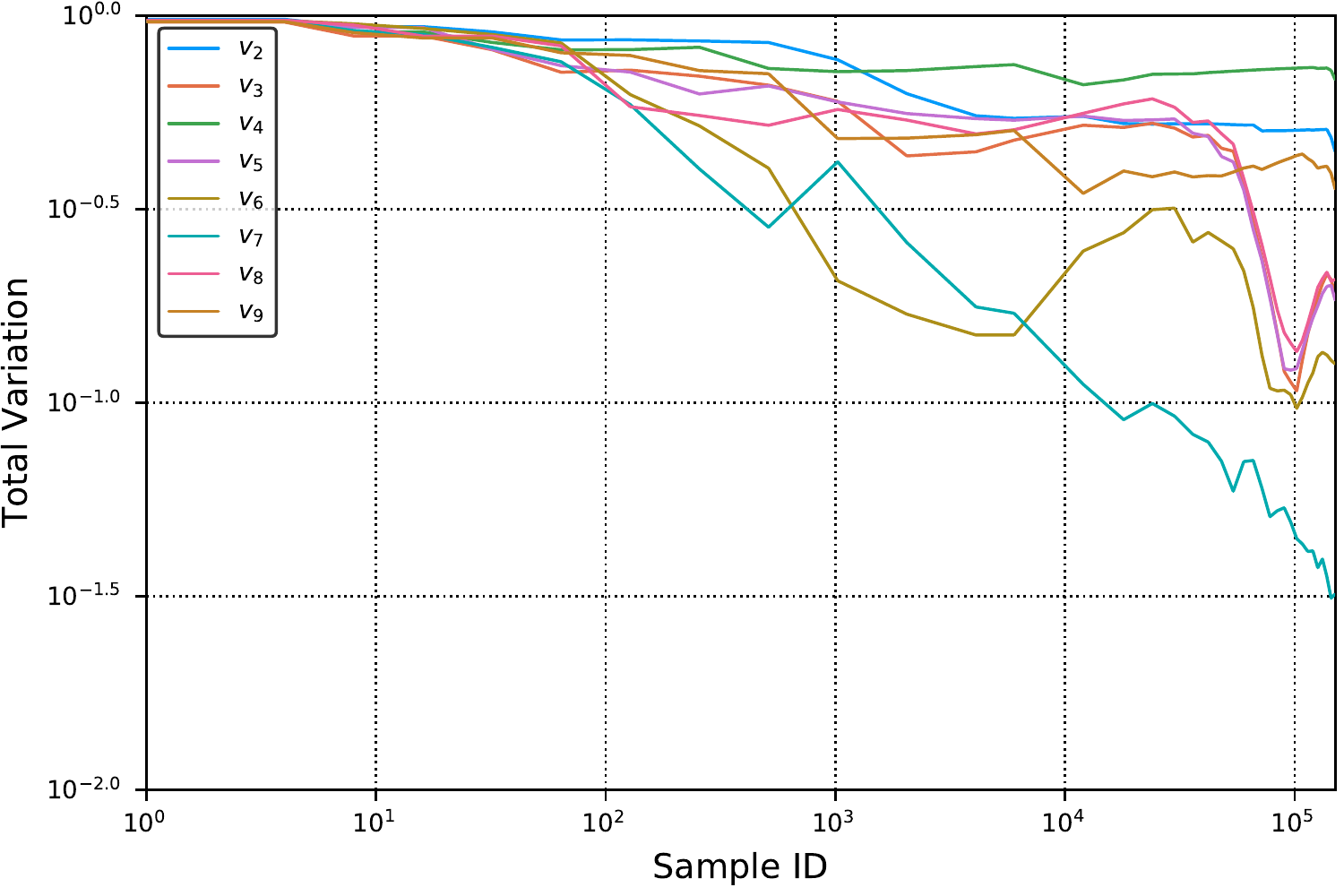}
\end{minipage}
\hfill
\begin{minipage}[b]{0.49\textwidth}
	\includegraphics[width=\textwidth]{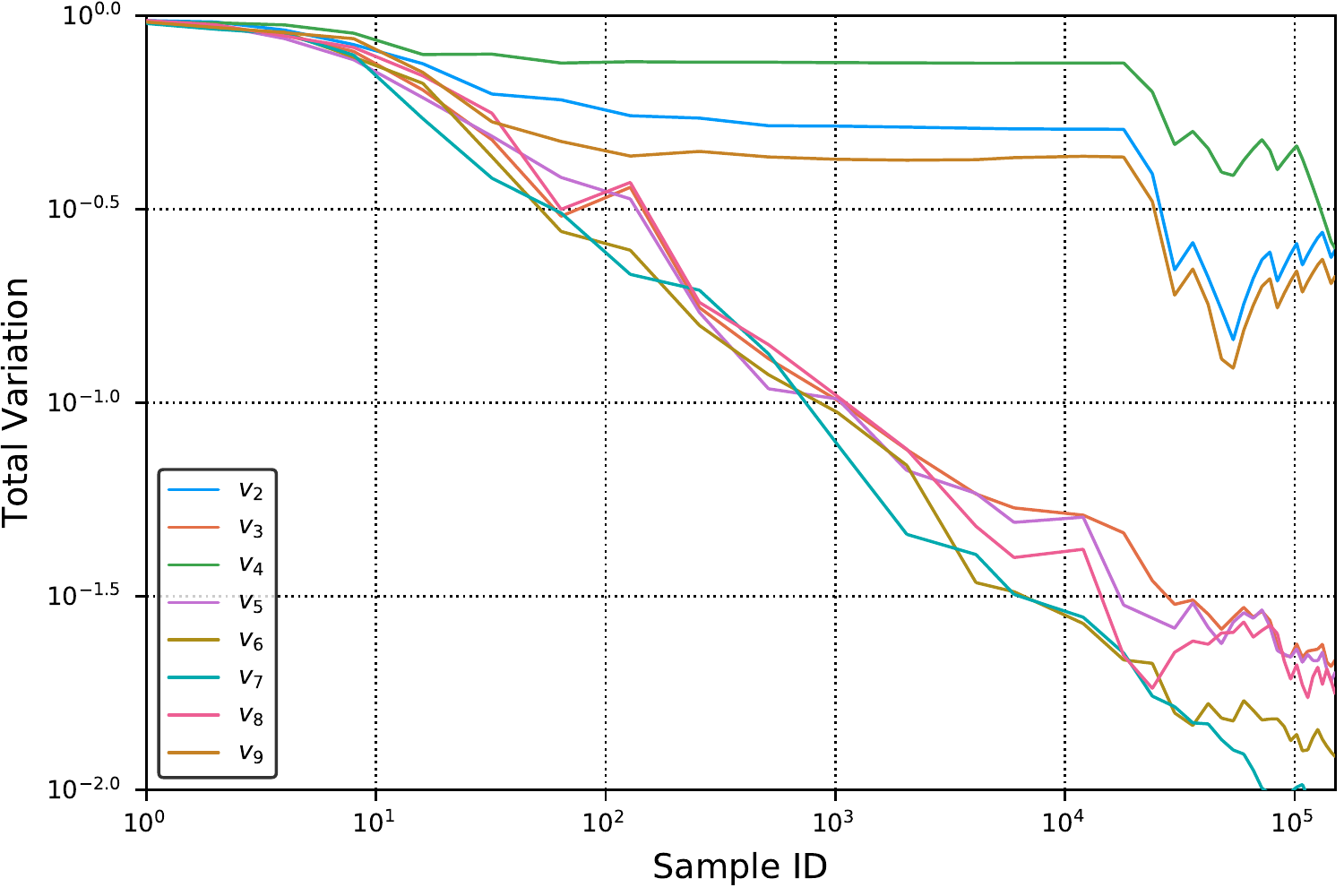}
\end{minipage}
\caption{\label{fig:ex2_totVarEvolve}Total variation norm between computed (150,000 samples) and ``true'' marginal probability density function for $\vfieldnobf_2,\dots,\vfieldnobf_9$, pCN (left) and HMC (right).}
\end{figure}

\subsubsection{Equal Runtime Comparison} \label{sec:ex2_runtime}
It is worth noting that, due to the selected values of the parameters $\epsilon$ and $\tau$, the HMC method used above required 32 PDE and 32 adjoint solves per sample, making it quite expensive relative to a sample of pCN. In our implementation, we were able to compute 125 pCN samples for each HMC sample. As a result, almost 19 million pCN samples could be computed in the time required to generate 150,000 HMC samples.  
For comparison in terms of equal computational cost, we now present the results of a single 19 million-sample pCN chain with the 150,000 sample HMC run shown earlier. \cref{fig:ex2_trace_runTime} compares the trace of the first few vector field components by sample number. We see that pCN does eventually achieve the jumps between states that HMC shows; however, the jumps are much less frequent for pCN, even when weighted by runtime, than for HMC (approximately 5 jumps vs. 20 jumps, respectively). 

\begin{figure}[!htbp]
  \centering
\begin{minipage}[b]{0.49\textwidth}
	\includegraphics[width=\textwidth]{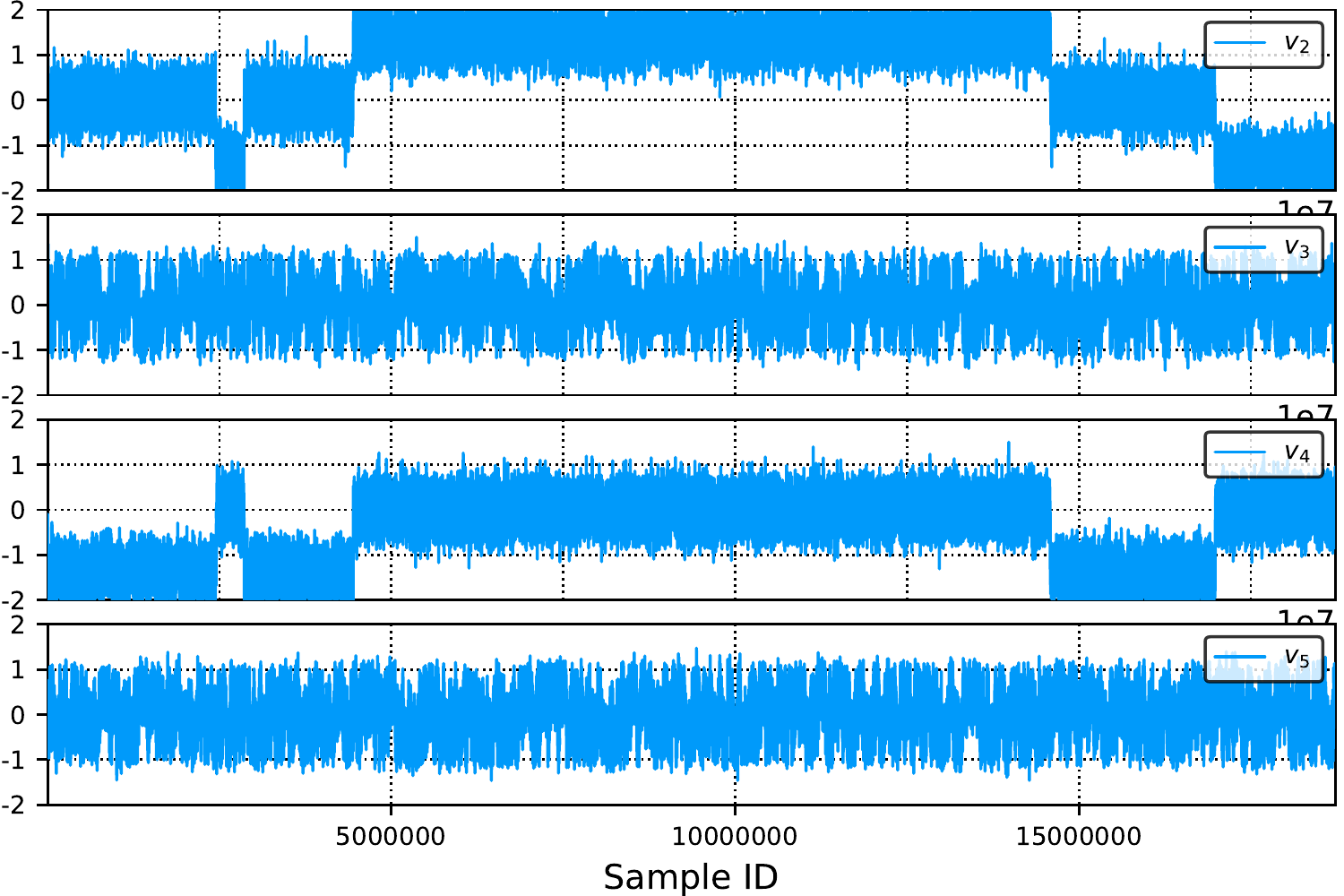}
\end{minipage}
\hfill
\begin{minipage}[b]{0.49\textwidth}
	\includegraphics[width=\textwidth]{figures/ex2/job_hmc/v_trace_150000.pdf}
\end{minipage}
\caption{\label{fig:ex2_trace_runTime}Trace of $\vfieldnobf_2,\dots,\vfieldnobf_5$ by sample number for runtime equivalent to 150,000 HMC samples, pCN (left) and HMC (right).}
\end{figure}

\cref{fig:ex2_posterior_method_runTime} shows the computed one- and two-dimensional histograms for the 19 million-sample pCN chain, which can be compared with the HMC figure in \cref{fig:ex2_posterior_method}. \cref{fig:ex2_totVar2dEvolve_runTime} compares the evolution of the total variation norm between the computed and ``true'' two-dimensional distributions for pCN (blue) and HMC (orange) chains of equal runtime. 
(Convergence of the $(\vfieldnobf_2,\vfieldnobf_9)$ correlation structure, for example, is shown in the bottom right subplot.) 
These two plots are more equivocal between the two methods. pCN produces better convergence for histograms with one probability mode (e.g., the pair $(\vfieldnobf_6,\vfieldnobf_7)$); for these components, the extra computations involved in HMC (which was tuned for larger jumps) do not appear to provide a benefit. However, the two methods exhibited very similar convergence for multimodal distributions (e.g., those associated with $\vfieldnobf_2$ or $\vfieldnobf_4$). HMC also reached all of the modes in the distribution, while pCN generated no samples in the mode near $\vfieldnobf_4 \approx 1.5$. Overall, it appears that pCN did a better job (on an equal-runtime basis) of sampling within modes while HMC did a better job of finding modes.

\begin{figure}[!htbp]
  \centering
\begin{minipage}[b]{0.49\textwidth}
	\includegraphics[width=\textwidth]{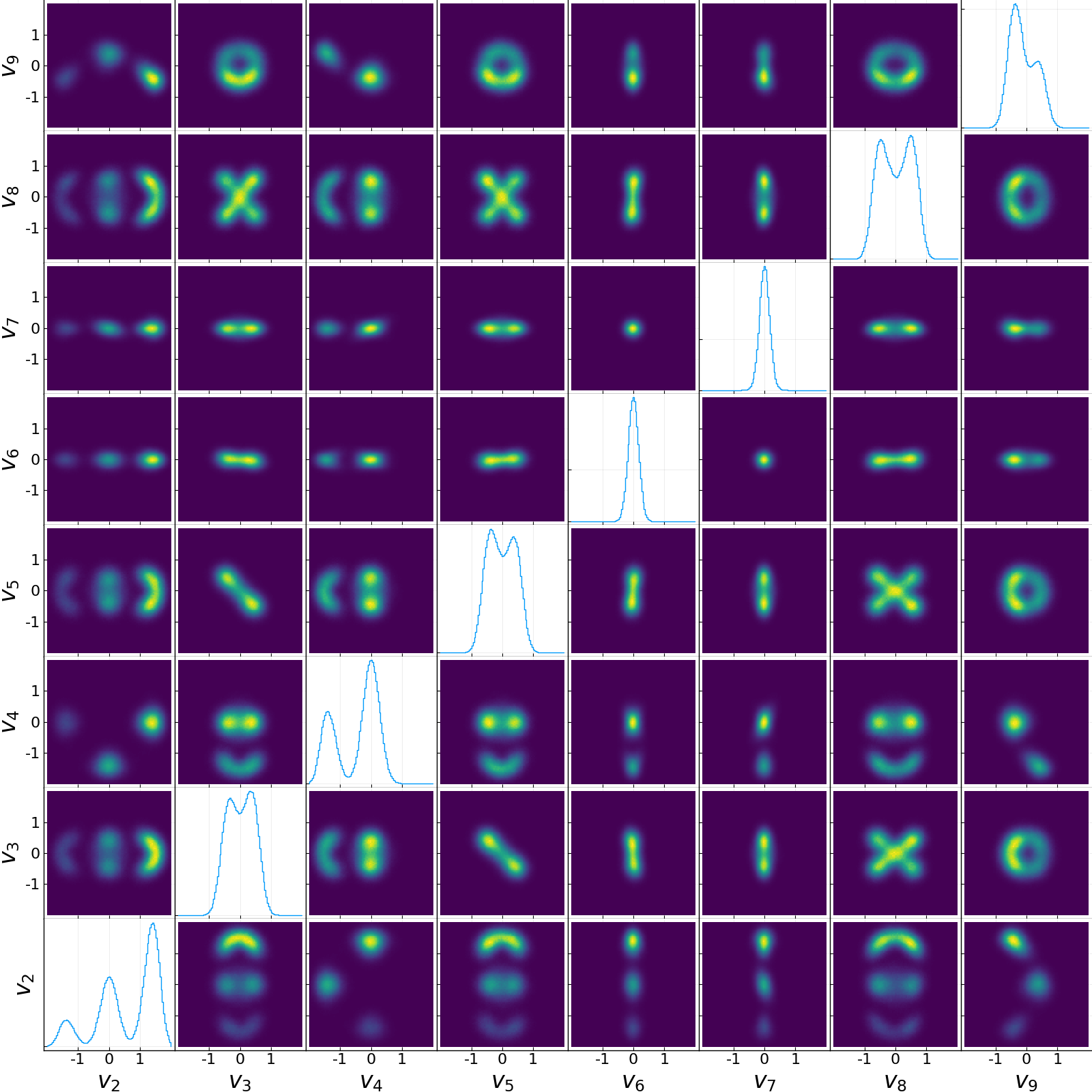}
  \caption{\label{fig:ex2_posterior_method_runTime}Computed 1D and 2D marginal distributions for each of the first eight vector field components (out of 197) for 19 million samples of pCN (same runtime as 150,000 samples of HMC).}
\end{minipage}
\hfill
\begin{minipage}[b]{0.49\textwidth}
	\includegraphics[width=\textwidth]{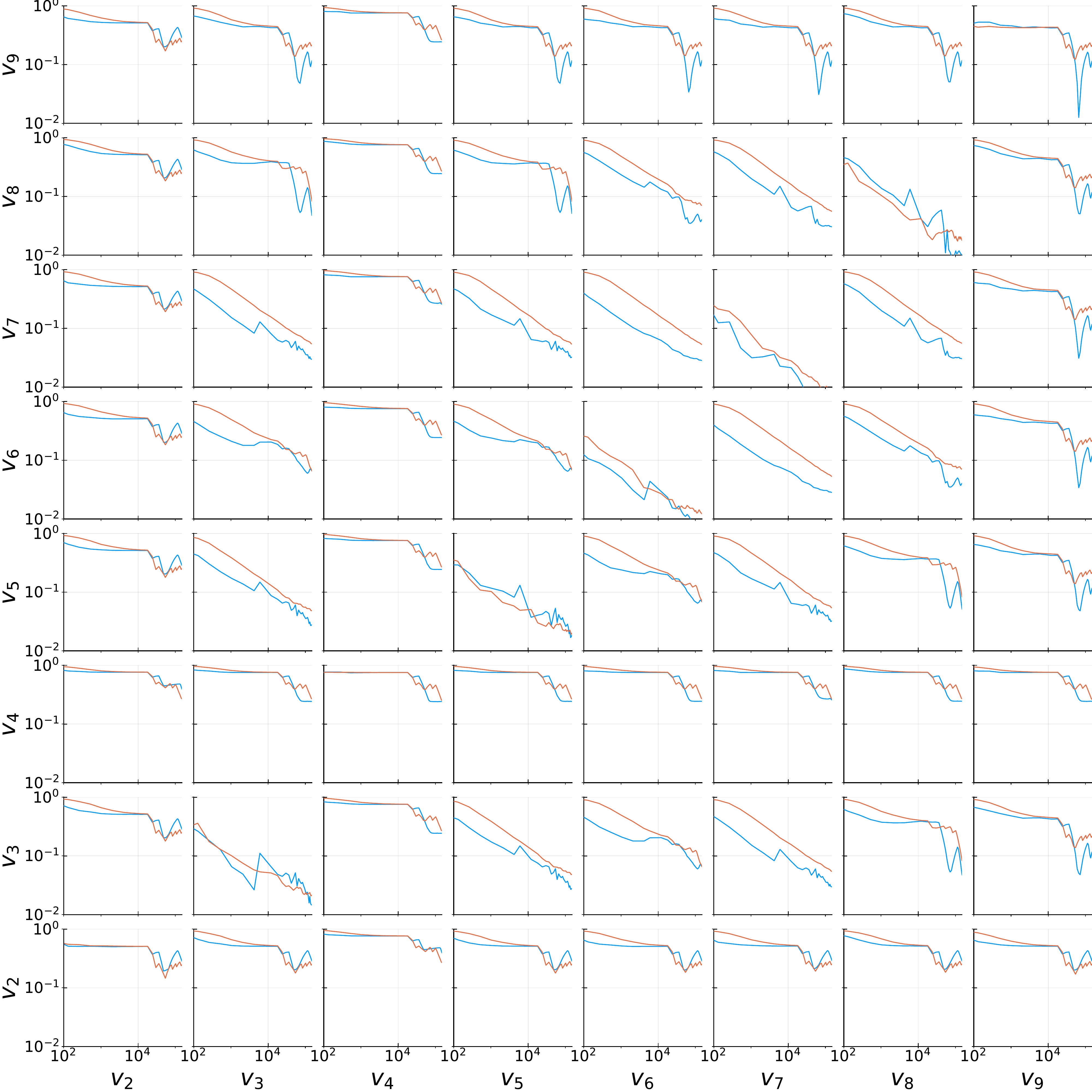}
  \caption{\label{fig:ex2_totVar2dEvolve_runTime}Total variation norm between computed and ``true'' 2D probability density for pairs of vector field components for pCN (blue) and HMC (orange) of equal computational time.}
\end{minipage}
\end{figure}

\subsection{Summary of Numerical Experiments}
To summarize the results of Examples 1 and 2, we see trade-offs between the MCMC methods:
\begin{itemize}
  \item pCN provides an inexpensive method to generate samples and explore local regions of a probability measure, with a free parameter $\beta$ that can be tuned to the problem. 
  \item HMC samples are more computationally expensive to generate; for the HMC test cases reported here, each HMC sample took 39-125 times as much time as one pCN sample, though in general this ratio will be dictated by the cost of the gradient computation and the choice of number of integration steps $\frac{\tau}{\epsilon}$. 
  \item In our numerical experiments, for posterior distributions with simple structure (e.g., Example 1 or some parts of Example 2), pCN exhibits similar (as measured by equal number of samples) or better (as measured by equal runtime) convergence to HMC.
  \item In our numerical experiments, for posterior distributions with more complicated structure (some components in Example 2), pCN still appeared to do a better job (for equal runtime) of sampling within probability modes, while HMC appeared to do a better job of jumping between states to find new modes. The overall impact on performance is difficult to discern and will likely depend heavily on the desired observables.
  \item Finally, we note that implementation of HMC is much more involved than pCN. With a working PDE solver, pCN can be implemented in a matter of minutes or hours. Developing a gradient solver and implementing the HMC leap frog integration (and debugging both) -- if even possible -- can require on the order of days or months of development time, depending on the complexity of the PDE solver. 
\end{itemize}

\clearpage
\appendix

\section{Selected Numerical Results for the IS and MALA Algorithms}\label{sec:results_is_mala}
In the main body of the paper, we present MCMC results for the preconditioned Crank-Nicolson (pCN, \cref{alg:mcmcpcn}) and Hamiltonian (HMC, \cref{alg:mcmchmc}) MCMC algorithms. Here we also present results for the independence sampler (IS) and Metropolis-adjusted Langevin (MALA) methods. IS, a special case of pCN when $\mu_0$ is Gaussian, draws proposals from the prior and requires one PDE solve per iteration. MALA uses one PDE and adjoint solve per iteration and is therefore more computationally expensive than iterations of IS or pCN but in general less computationally expensive than those of HMC.
\begin{algorithm}
\caption{Independence Sampler MCMC.}\label{alg:mcmcind}
\begin{algorithmic}[1]
\item Given sample $\mcmcsamp^{(k)}$
\item Propose $\mcmccand \sim \mu_0$
\item Set $\mcmcsamp^{(k+1)} = \mcmccand$ with probability
  $\min\left\{1,\exp\left(\Phi\left(\mcmcsamp^{(k)}\right)
      - \Phi(\mcmccand) \right)\right\}$, otherwise $\mcmcsamp^{(k+1)} = \mcmcsamp^{(k)}$
\end{algorithmic}
\end{algorithm}

\begin{algorithm}
\caption{Metropolis-Adjusted Langevin (MALA) MCMC.}\label{alg:mcmcmala}
\begin{algorithmic}[1]
\item Given free parameter $h$ and sample $\mcmcsamp^{(k)}$
\item Propose $\mcmccand 
  = \frac{2-h}{2+h}\mcmcsamp^{(k)} -\frac{2h}{2+h}\covar D\Phi(\mcmcsamp^{(k)}) 
  + \frac{\sqrt{8h}}{2+h} \xi^{(k)}$, $\xi^{(k)} \sim N(0,\covar)$
\item Set $\mcmcsamp^{(k+1)} = \mcmccand$ with probability
  $\alpha(\mcmcsamp^{(k)},\mcmccand)=1 \wedge
  \exp\left(\rho\left(\mcmcsamp^{(k)},\mcmccand\right) -
    \rho\left(\mcmccand,\mcmcsamp^{(k)}\right) \right)$,
  where $\rho(\mcmcsamp,\mcmccand)$ is given by 
  \begin{equation}
    \rho(\mcmcsamp,\mcmccand) 
    = \Phi(\mcmcsamp) + \half \ip{\mcmccand-\mcmcsamp}{D\Phi(\mcmcsamp)} 
      + \frac{h}{4}\ip{\mcmcsamp+\mcmccand}{D\Phi(\mcmcsamp)} + \frac{h}{4}\norm{\covar^{\half}D\Phi(\mcmcsamp)}^2
    \label{eq:mh_mala_rho}
  \end{equation}
  Otherwise $\mcmcsamp^{(k+1)} = \mcmcsamp^{(k)}$ (unchanged)
\end{algorithmic}
\end{algorithm}

\subsection{Example 1}
Here we present IS and MALA results for numerical Example 1 (see \cref{sec:ex1_simple}). For MALA, we chose $h=0.005$ to approximate the optimal acceptance rate of $57\%$ from \cite{roberts2001optimal}. The actual acceptance rates for IS and MALA were $0.012\%$ and $53.7\%$, respectively. The chain concluded with thousands of consecutive rejections. \cref{fig:ex1_totVarEvolve_sm}, an extension of \cref{fig:ex1_totVarEvolve_sm} to four methods, shows convergence, as measured by total variation norm, of the 1D marginal distributions to the ``true'' marginal distributions shown in the diagonal of \cref{fig:ex1_posterior}. We observe IS failing to converge due to the high number of rejections. MALA converges at roughly the same rate as pCN.

\begin{figure}[!htbp]
  \centering
\begin{minipage}[b]{0.4\textwidth}
	\includegraphics[width=\textwidth]{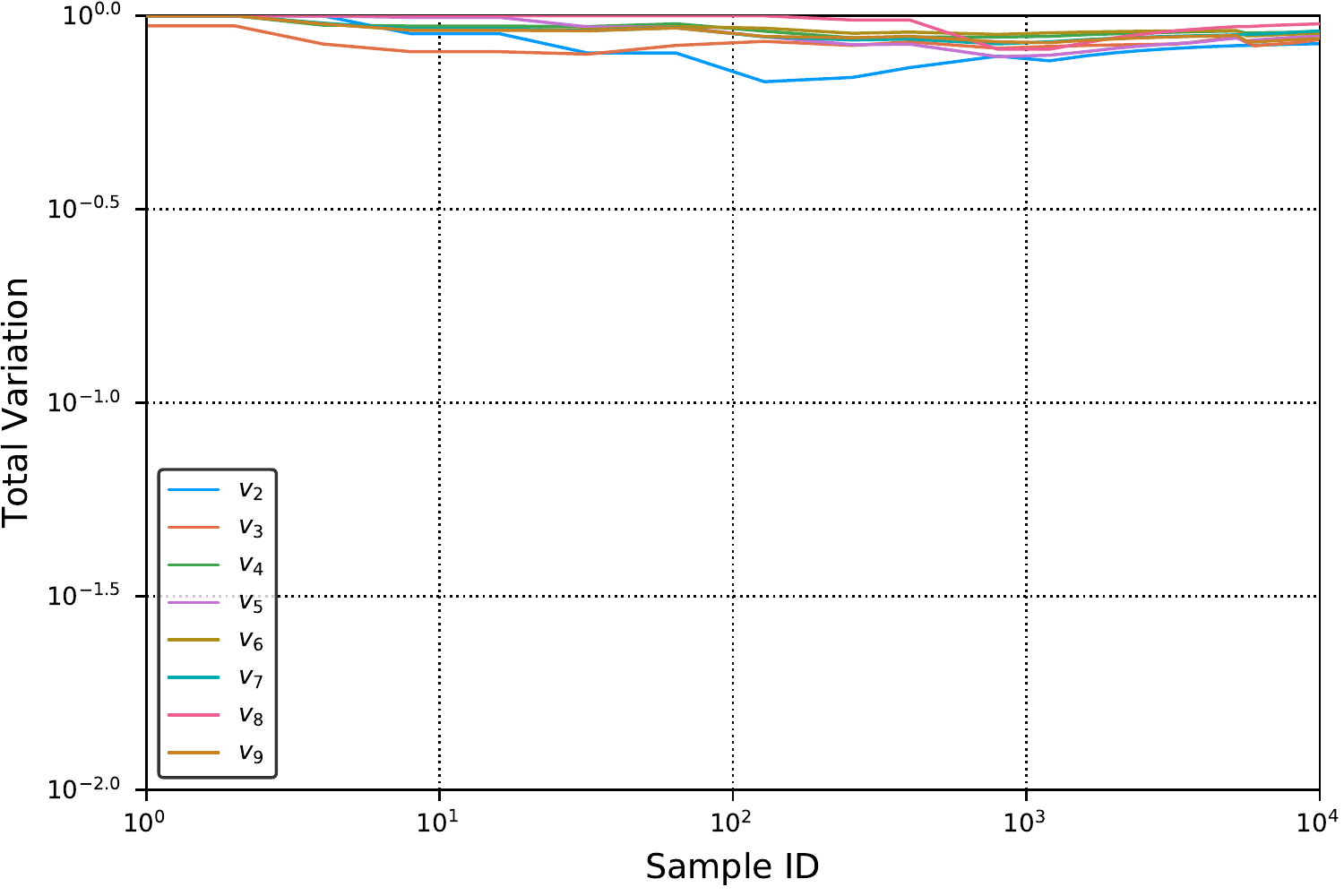}
\end{minipage}
\qquad
\begin{minipage}[b]{0.4\textwidth}
	\includegraphics[width=\textwidth]{figures/ex1/job_pcn1/obs/totVar_evolve_loglog_10000.pdf}
\end{minipage}
\begin{minipage}[b]{0.4\textwidth}
	\includegraphics[width=\textwidth]{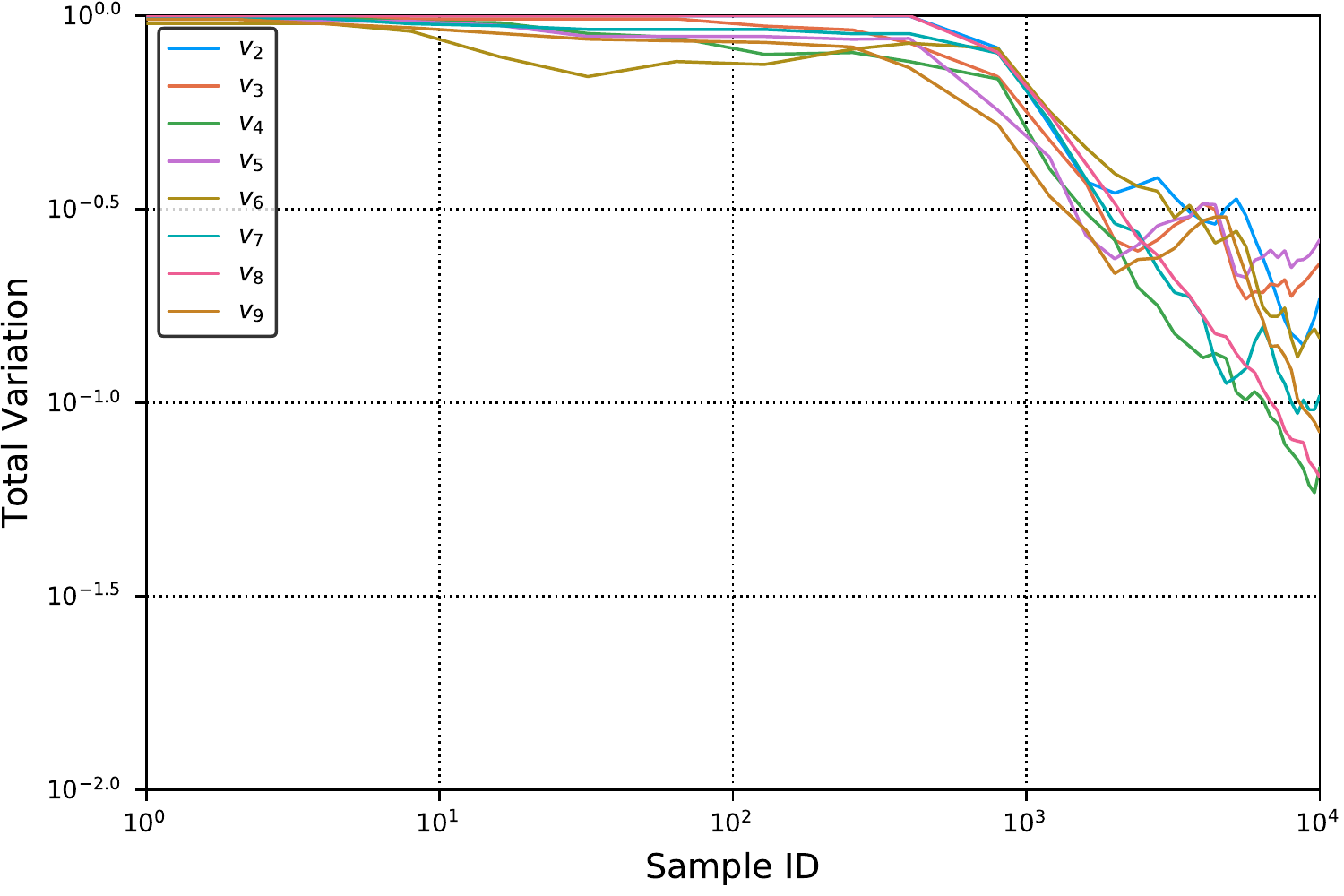}
\end{minipage}
\qquad
\begin{minipage}[b]{0.4\textwidth}
	\includegraphics[width=\textwidth]{figures/ex1/job_hmc/obs/totVar_evolve_loglog_10000.pdf}
\end{minipage}

\caption{\label{fig:ex1_totVarEvolve_sm}Total variation norm between computed and ``true'' marginal probability density function for $\vfieldnobf_2,\dots,\vfieldnobf_9$ for 10,000 samples, Example 1. Top Left: IS, Top Right: pCN, Bottom Left: MALA, Bottom Right: HMC.}
\end{figure}

\cref{fig:ex1_totVarEvolve_runTime_sm}, an extension of \cref{fig:ex1_totVarEvolve_runTime} to four methods, shows the same total variation convergence normalized by runtime. IS samples are roughly the same cost to generate as pCN samples, so 39 IS samples were generated per HMC sample. Similarly, an HMC sample took roughly $8$ times as long to generate as a MALA sample because the version of HMC used in Example 1 required eight PDE and adjoint solves per sample, while MALA only requires one of each. Thus, we can reweight IS and pCN samples by $39$ and MALA samples by $8$ to get a comparison of the sampling accuracy per unit time. IS again fails to exhibit any meaningful convergence. MALA converges somewhat slower than either pCN or HMC, the former because MALA samples took more than twice as long to generate due to the need for an adjoint solve at each iteration.

\begin{figure}[!htbp]
  \centering
\begin{minipage}[b]{0.4\textwidth}
	\includegraphics[width=\textwidth]{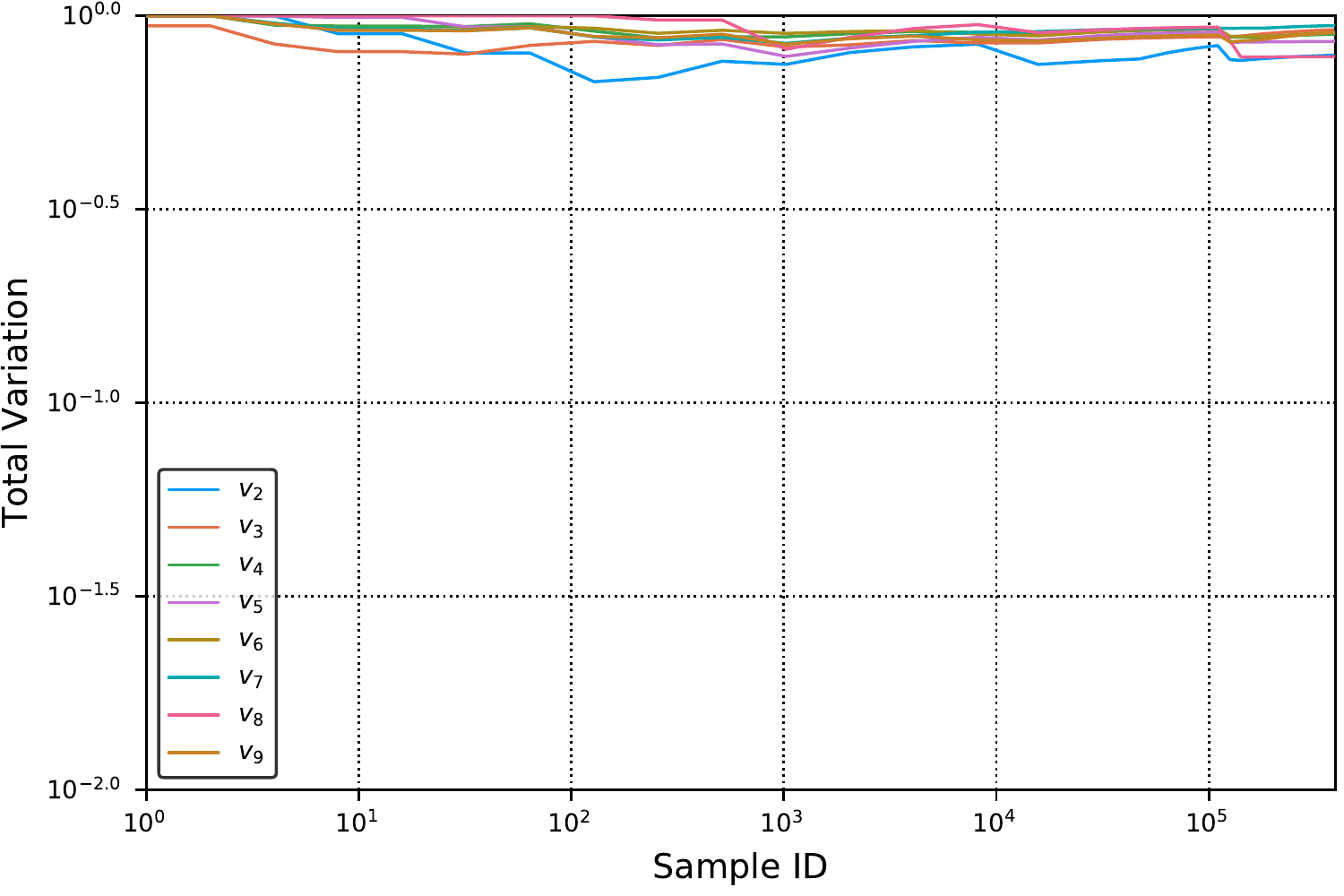}
\end{minipage}
\qquad
\begin{minipage}[b]{0.4\textwidth}
	\includegraphics[width=\textwidth]{figures/ex1/job_pcn1/obs/totVar_evolve_loglog_391672.pdf}
\end{minipage}
\begin{minipage}[b]{0.4\textwidth}
	\includegraphics[width=\textwidth]{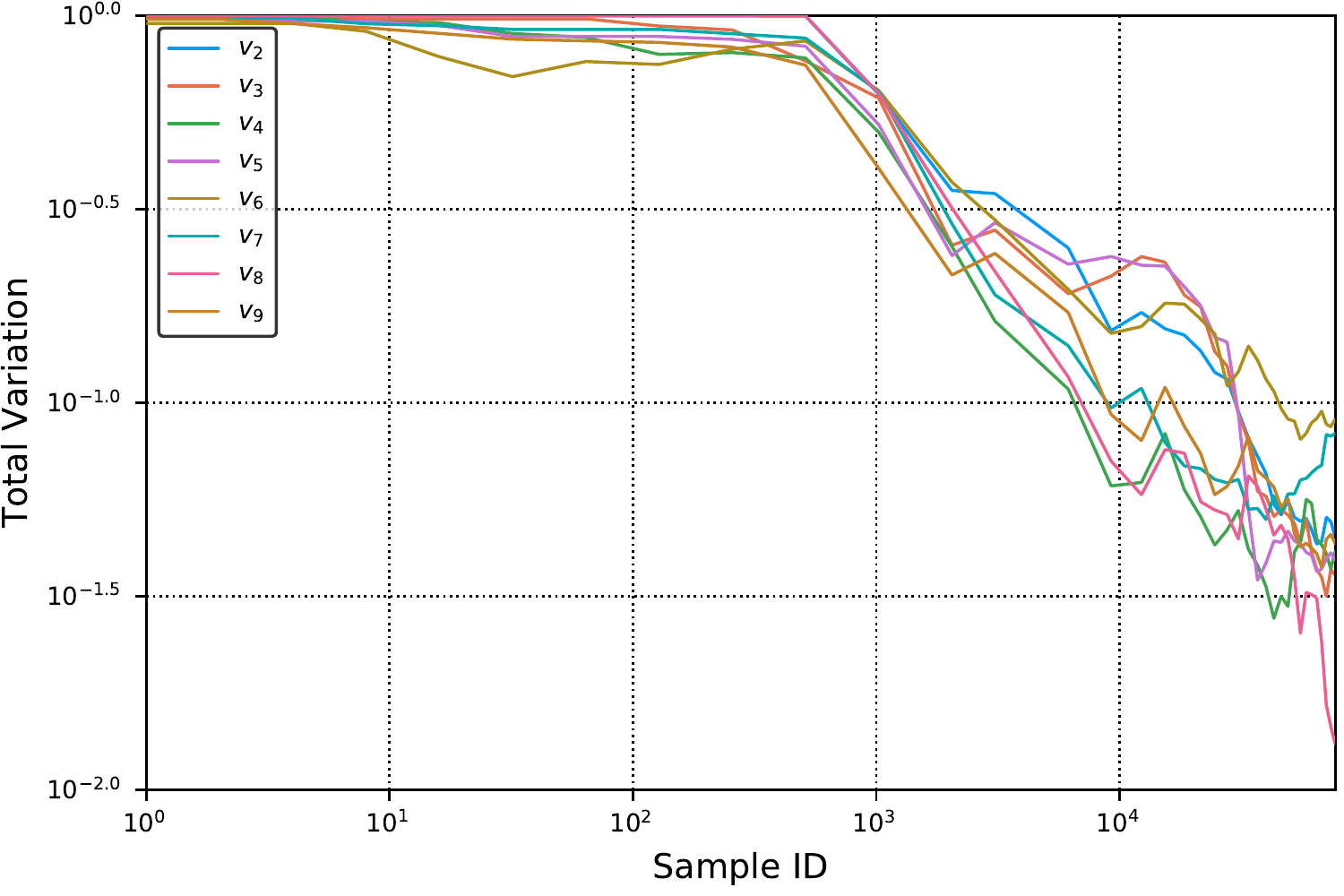}
\end{minipage}
\qquad
\begin{minipage}[b]{0.4\textwidth}
	\includegraphics[width=\textwidth]{figures/ex1/job_hmc/obs/totVar_evolve_loglog_10000_noendpt.pdf}
\end{minipage}

\caption{\label{fig:ex1_totVarEvolve_runTime_sm}Total variation norm between computed and ``true'' marginal probability density function for $\vfieldnobf_2,\dots,\vfieldnobf_9$ for runtime equivalent to 10,000 HMC samples, Example 1. Top Left: IS, Top Right: pCN, Bottom Left: MALA, Bottom Right: HMC.}
\end{figure}

\subsection{Example 2}
For Example 2 (see \cref{sec:ex2_multihump}), we chose $h=0.001$ for MALA to again match the optimal acceptance rate from \cite{roberts2001optimal}. \cref{fig:ex2_posterior_method_sm}, which can be compared with analogous plots for pCN and HMC in \cref{fig:ex2_posterior_method}, shows the computed 1D and 2D histograms for each of the first eight vector field components (out of 197) for 100,000 MALA samples. We see that the MALA chain failed to resolve the multiple modes of the posterior measure seen in \cref{fig:ex2_posterior}.
\begin{figure}[!htbp]
  \centering
\begin{minipage}[b]{0.49\textwidth}
	\includegraphics[width=\textwidth]{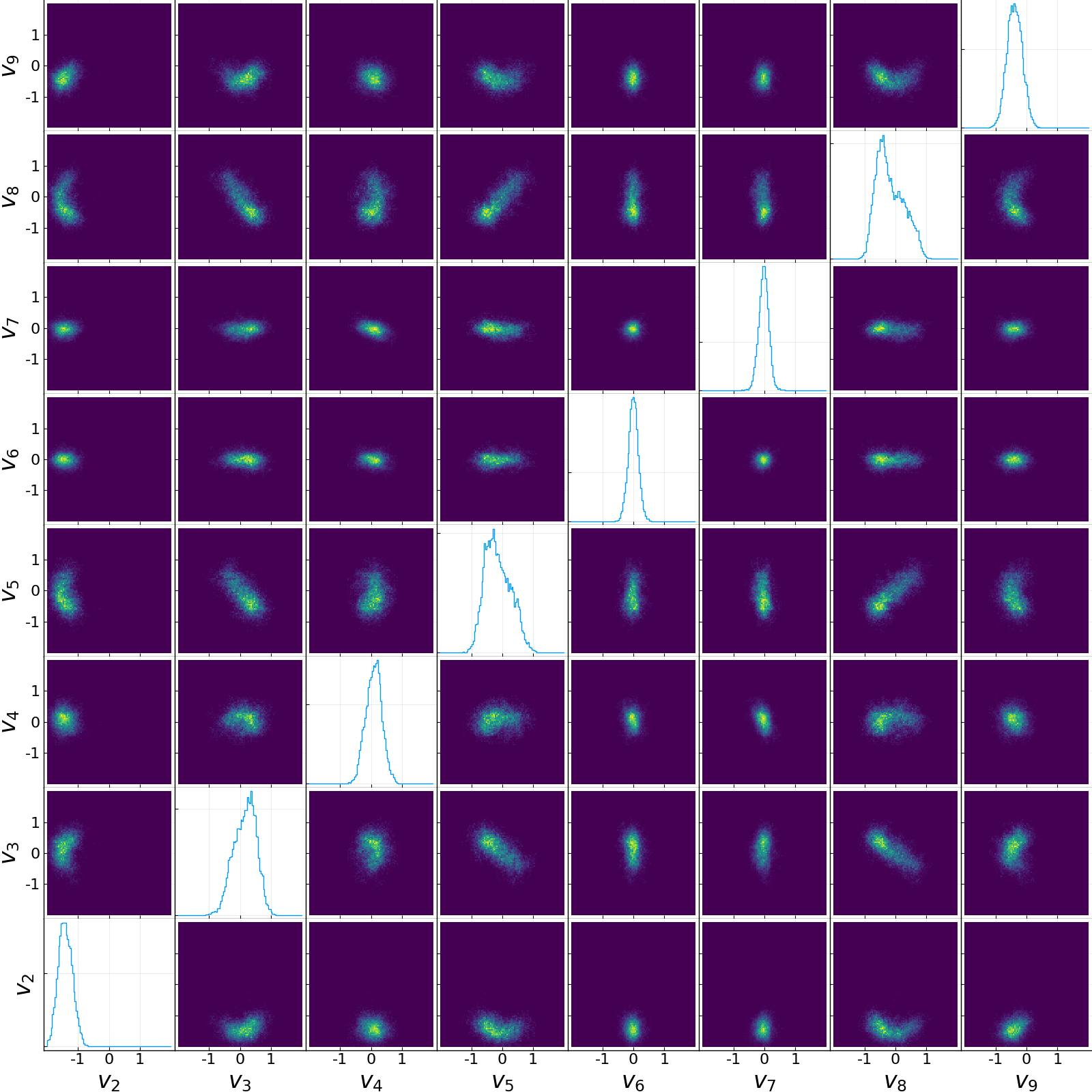}
\end{minipage}
\caption{\label{fig:ex2_posterior_method_sm}Computed 1D and 2D marginal distributions for each of the first eight vector field components (out of 197) for 100,000 samples of MALA, Example 2. (Compare with \cref{fig:ex2_posterior_method}.)}
\end{figure}

\section{Convergence of Observables}\label{sec:results_obs}
In this section, we compare pCN and HMC convergence for a series of observables that are of interest to the passive scalar community. These observables, which involve both $\vfield$ and $\pdesol$, are summarized in
\cref{tab:observables}. Convergence is measured as relative error of the mean vs.~the mean given by the computed ``true'' posterior measures shown in \cref{fig:ex1_posterior} for Example 1 and \cref{fig:ex2_posterior} for Example 2. 
\begin{table}[h]
{\footnotesize
  \caption{Observables.}  \label{tab:observables}
\begin{center}
  \begin{tabular}{ | c | c | }
  \hline
  Observable & Formula \\
  \hline    \hline
  Mean Scalar Variance \cite{warhaft2000passive} & $\|\pdesol-\overline{\pdesol}\|_{L^2}^2$ \\\hline
  Mean Scalar Dissipation Rate \cite{shraiman2000scalar,warhaft2000passive} & $\epsilon_\pdesol = 2\kappa\norm{\nabla \pdesol}_{L^2}^2$ \\\hline
  Enstrophy \cite{kupiainen2000statistical,baiesi2005enstrophy} & $\half \| \nabla \times \vfield \|_{L^2}^2$ \\\hline
  Enstrophy Dissipation Rate & $\norm{\nabla \left( \nabla \times \vfield \right)}_{L^2}^2$ \\\hline
  Scalar Differences  \cite{shraiman2000scalar,warhaft2000passive} & $\Delta_r \pdesol = \pdesol(\x+\mathbf{r},t)-\pdesol(\x,t)$ \\\hline
  \end{tabular}
\end{center}
}
\end{table}

\subsection{Example 1}
In this section, we compare convergence of observables for the single-mode posterior in Example 1 (see \cref{sec:ex1_simple}).  
\cref{fig:ex1_obsEvolve} shows the relative error in the mean value of
the first four observables in \cref{tab:observables} at $t=1$, through 
10,000 samples.\footnote{The sharp downward dips in these relative 
error plots occur when a cumulative moving average drifts past the 
``true'' value. At these points, the relative error is zero, so a 
log plot exhibits a downward dip toward $\log(0)=-\infty$.} 
HMC converges an order of magnitude more quickly than pCN.
\begin{figure}[!htbp]
  \centering
  \begin{minipage}[b]{0.49\textwidth}
  	\includegraphics[width=\textwidth]{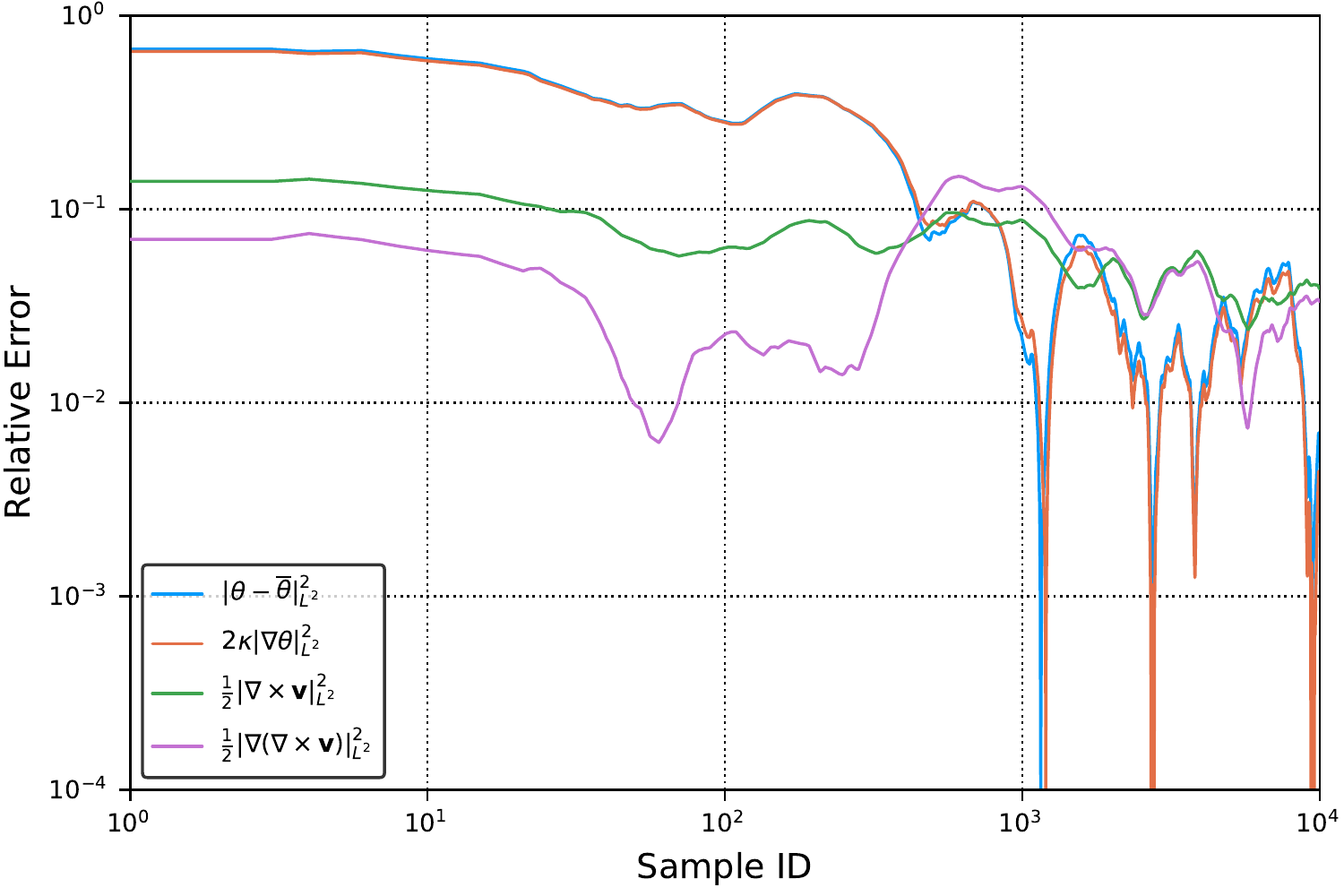}
  \end{minipage}
  \hfill
  \begin{minipage}[b]{0.49\textwidth}
  	\includegraphics[width=\textwidth]{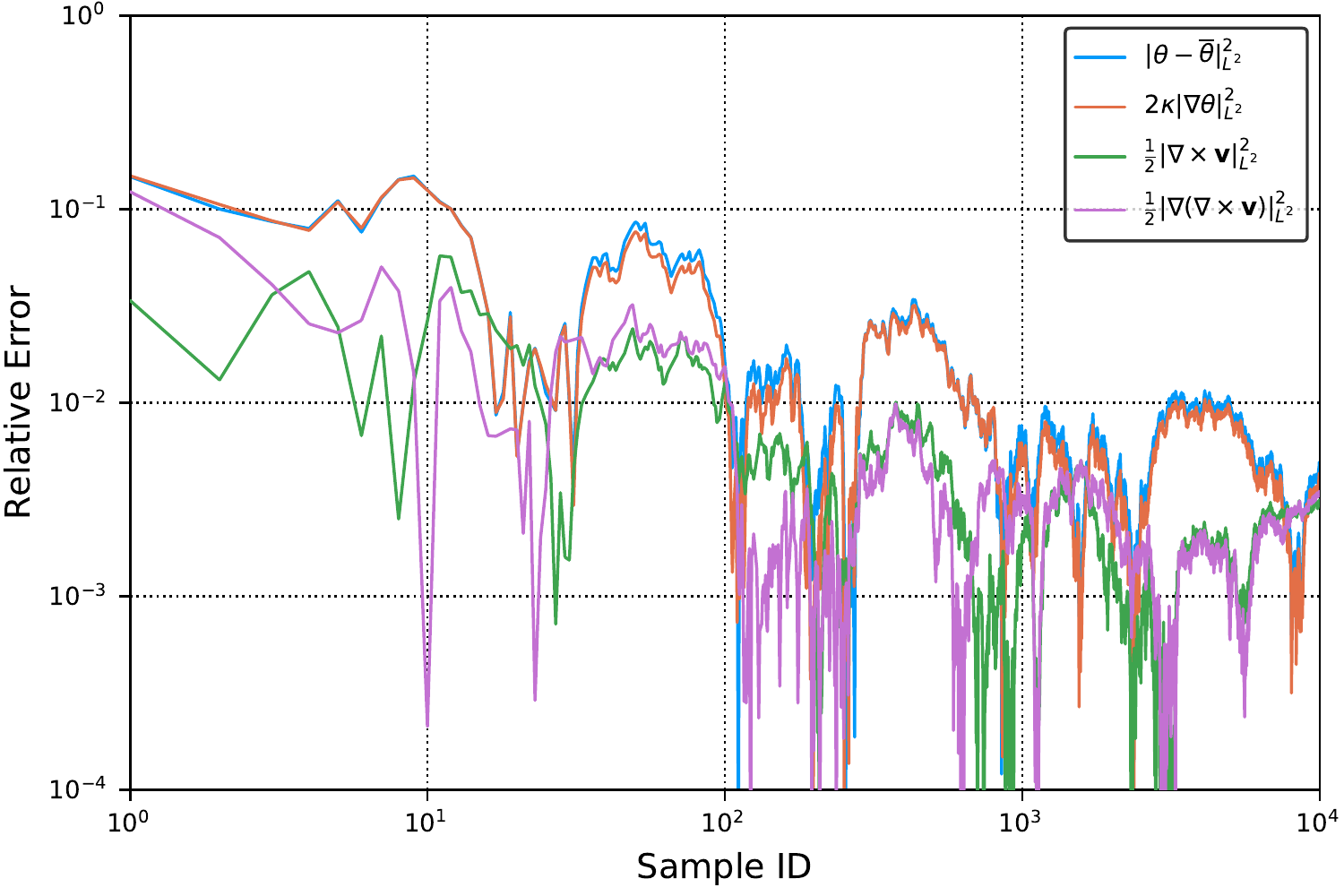}
  \end{minipage}

  \caption{\label{fig:ex1_obsEvolve}Relative error for the mean (cumulative moving average) of observables scalar variance, scalar dissipation rate, enstrophy, and enstrophy dissipation for 10,000 samples, Example 1. Left: pCN, Right: HMC.}
\end{figure}

More challenging observables to resolve are scalar differences (see \cref{tab:observables}), which require resolving $\pdesol$ at two different locations. 
\cref{fig:ex1_scalarDiffEvolve} shows convergence results for the mean values (cumulative moving averages) of scalar differences with $\x=[0,0]$ and $\mathbf{r_i}=2^{-i}[1,1]$ for $i=1,2,3,4$, up to 100,000 samples (10 times longer than the results shown in \cref{fig:ex1_obsEvolve}).
Convergence for scalar differences is much slower than for the observables in \cref{fig:ex1_obsEvolve}, though again we see that the relative error decays more quickly for HMC than for pCN. For both methods, the convergence is fairly uniform across scales -- the various scalar differences converge at the same rates.
\begin{figure}[!htbp]
  \centering
  \begin{minipage}[b]{0.49\textwidth}
  	\includegraphics[width=\textwidth]{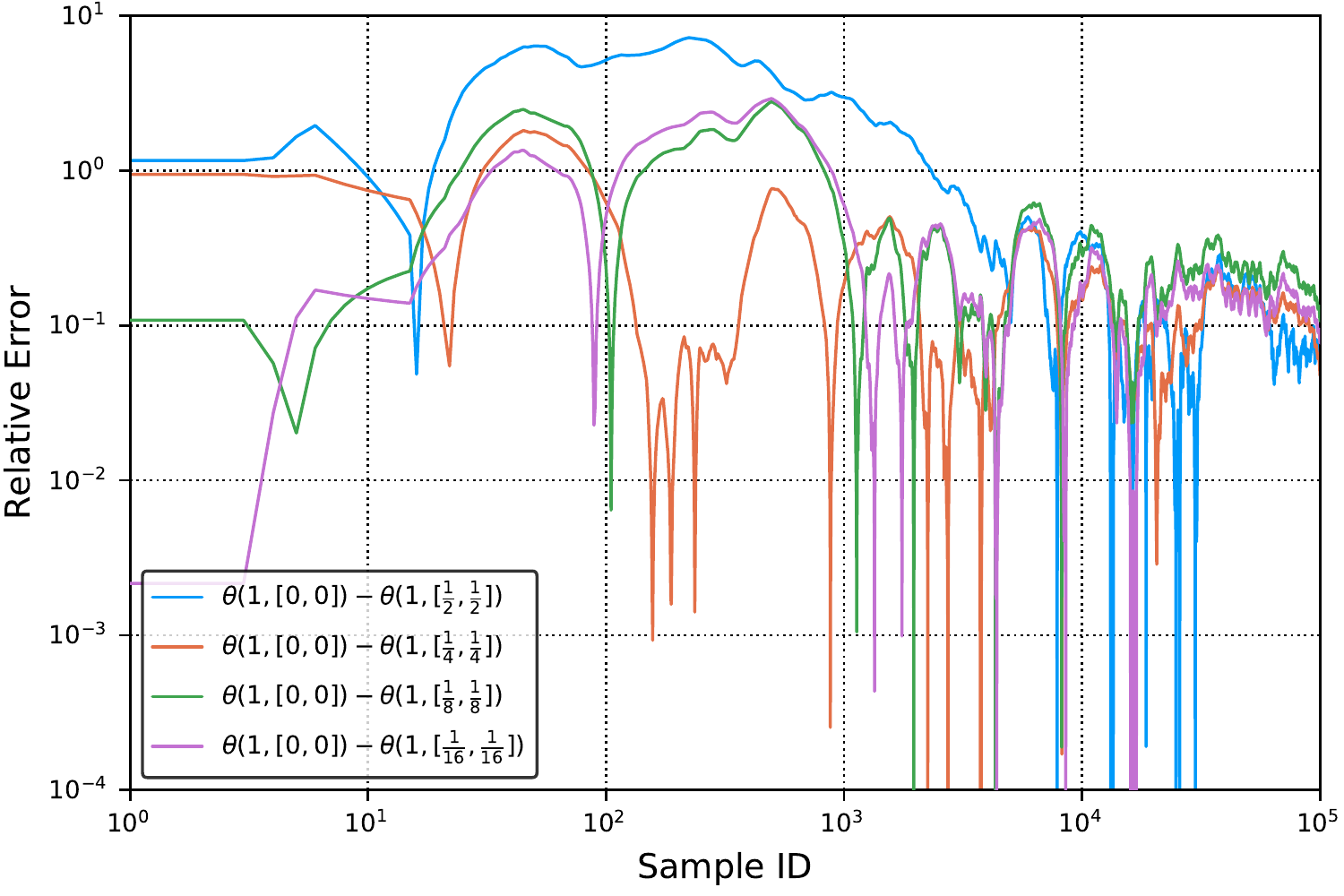}
  \end{minipage}
  \hfill
  \begin{minipage}[b]{0.49\textwidth}
  	\includegraphics[width=\textwidth]{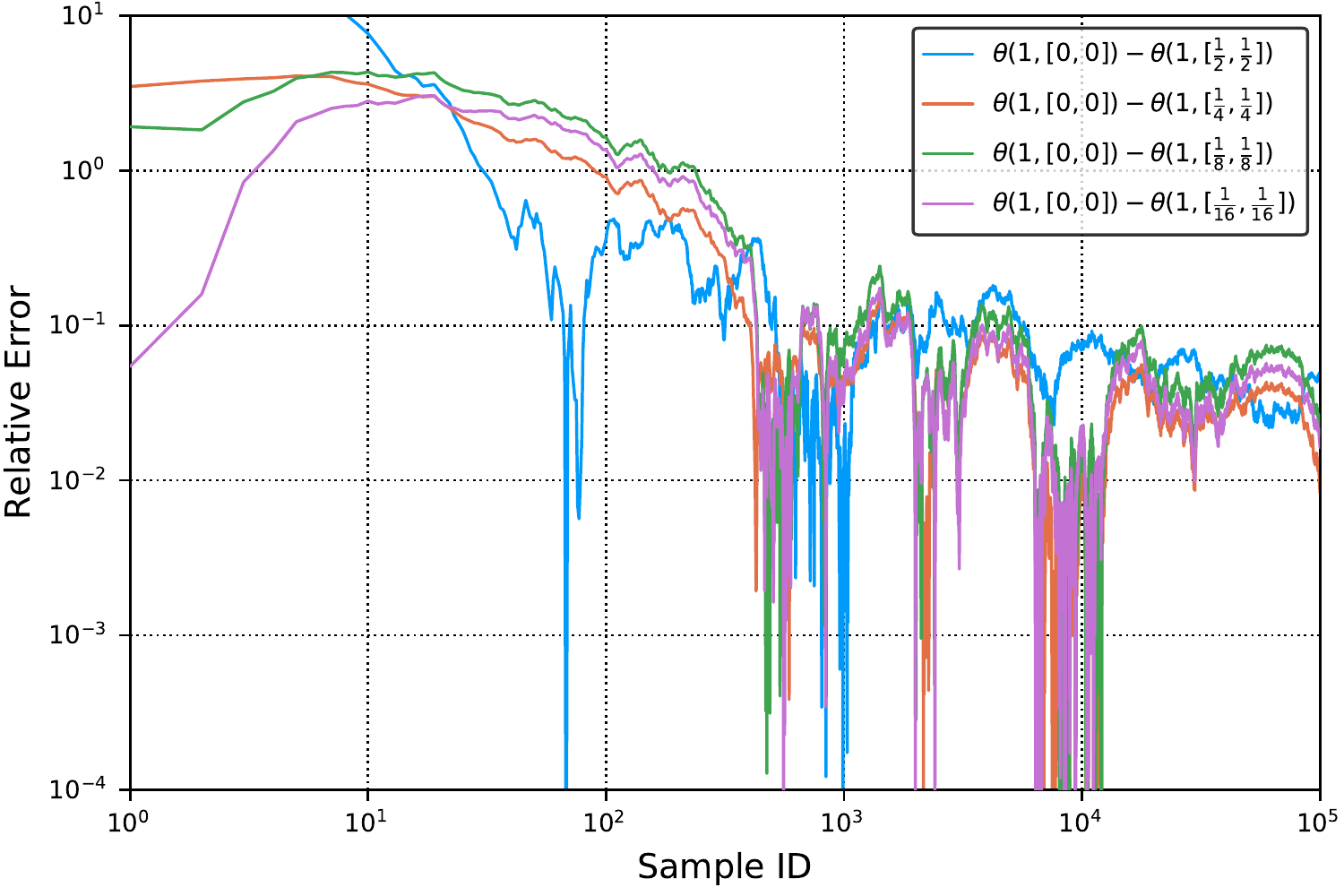}
  \end{minipage}
  
  \caption{\label{fig:ex1_scalarDiffEvolve}Relative error for the mean (cumulative moving average) of scalar differences at $t=1$ between the origin and $\left[\frac{1}{2},\frac{1}{2}\right]$, $\left[\frac{1}{4},\frac{1}{4}\right]$, $\left[\frac{1}{8},\frac{1}{8}\right]$, $\left[\frac{1}{16},\frac{1}{16}\right]$ for 100,000 samples, Example 1. Left: pCN, Right: HMC.}
\end{figure}

\subsection{Example 2} 
\label{sec:ex2_obsConverge}\label{sec:ex2_mcmcCompare}
In this section, we compare convergence of observables for the multimodal problem in Example 2 (see \cref{sec:ex2_multihump}).  
\cref{fig:ex2_obsEvolve} shows the cumulative moving
average of scalar variance, scalar dissipation rate, enstrophy, and
enstrophy dissipation (see \cref{tab:observables}) for pCN and HMC. The means for both methods
converge, though HMC converges in an order of magnitude fewer samples
than pCN.

\begin{figure}[!htbp]
  \centering
  \begin{minipage}[b]{0.49\textwidth}
  	\includegraphics[width=\textwidth]{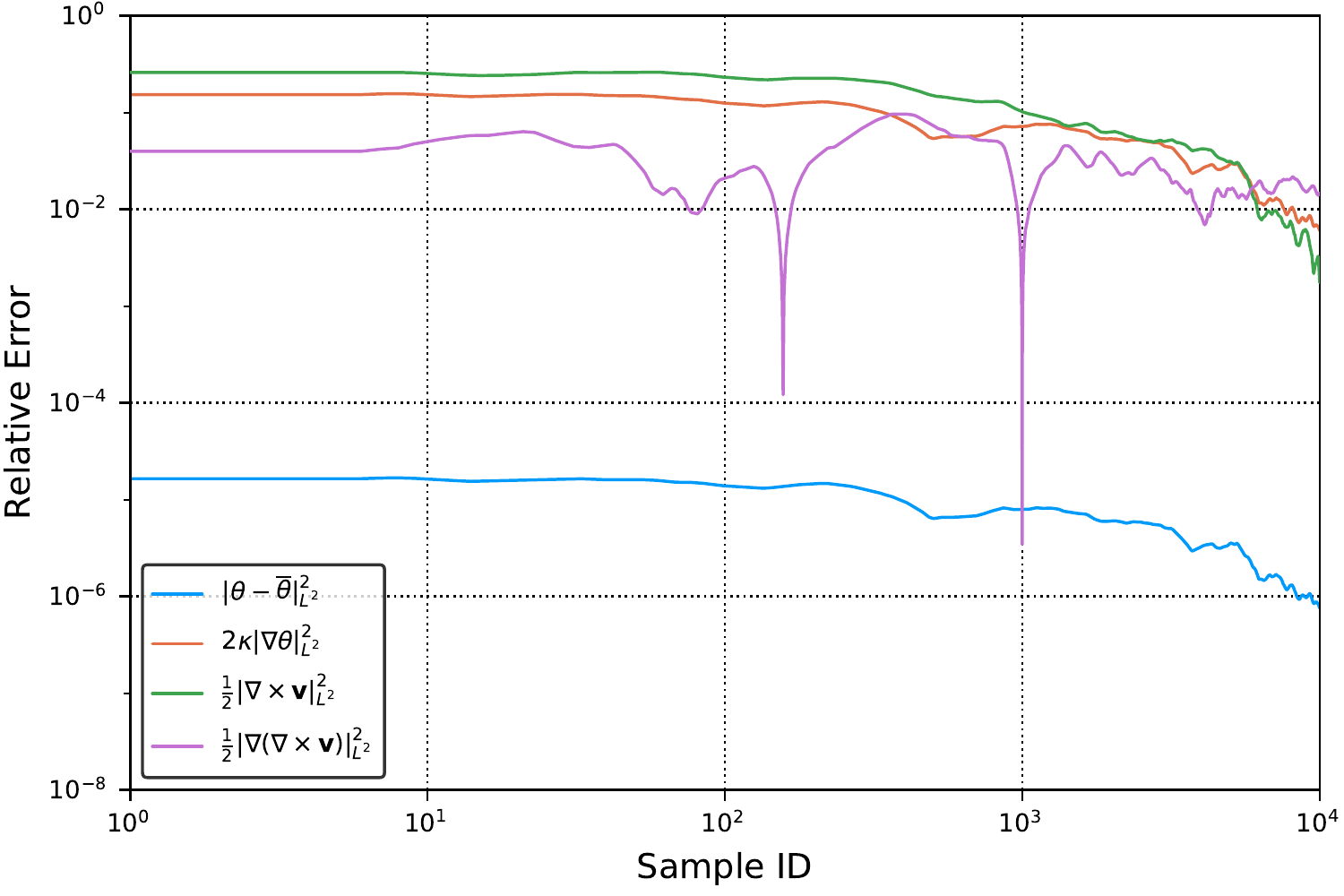}
  \end{minipage}
  \hfill
  \begin{minipage}[b]{0.49\textwidth}
  	\includegraphics[width=\textwidth]{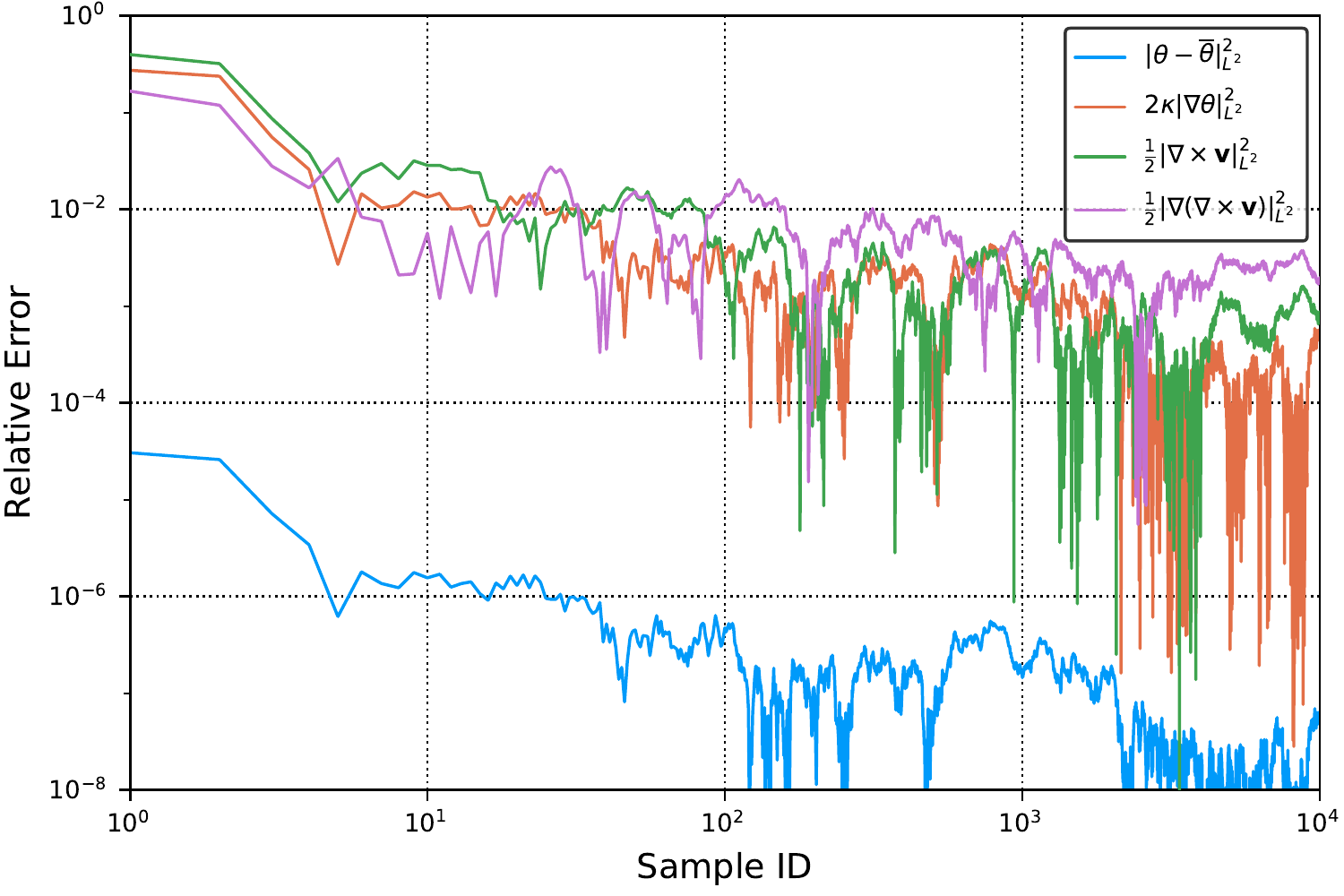}
  \end{minipage}
  \caption{\label{fig:ex2_obsEvolve}Relative error for the mean (cumulative moving average) of observables scalar variance, scalar dissipation rate, enstrophy, and enstrophy dissipation for 10,000 samples, Example 2. Left: pCN, Right: HMC.}
\end{figure}

\cref{fig:ex2_scalarDiffEvolve} shows similar convergence plots for
scalar differences (see \cref{tab:observables}), which proved much harder for the MCMC methods to
resolve; the figure shows results through 150,000 samples (15 times
more samples than in \cref{fig:ex2_obsEvolve}). 
Scalar differences across small distances proved much more difficult for 
the methods to resolve than the longer distances; pCN, in particular, shows 
almost no convergence for the two shorter-range differences. The analogous plots for 
HMC begin to exhibit a sawtooth shape as the number of samples grows; this is the 
result of balancing between the number of samples that the chain produces in each of the various
probability modes due to the jumps seen in \cref{fig:ex2_trace}.

\begin{figure}[!htbp]
  \centering
  \begin{minipage}[b]{0.49\textwidth}
  	\includegraphics[width=\textwidth]{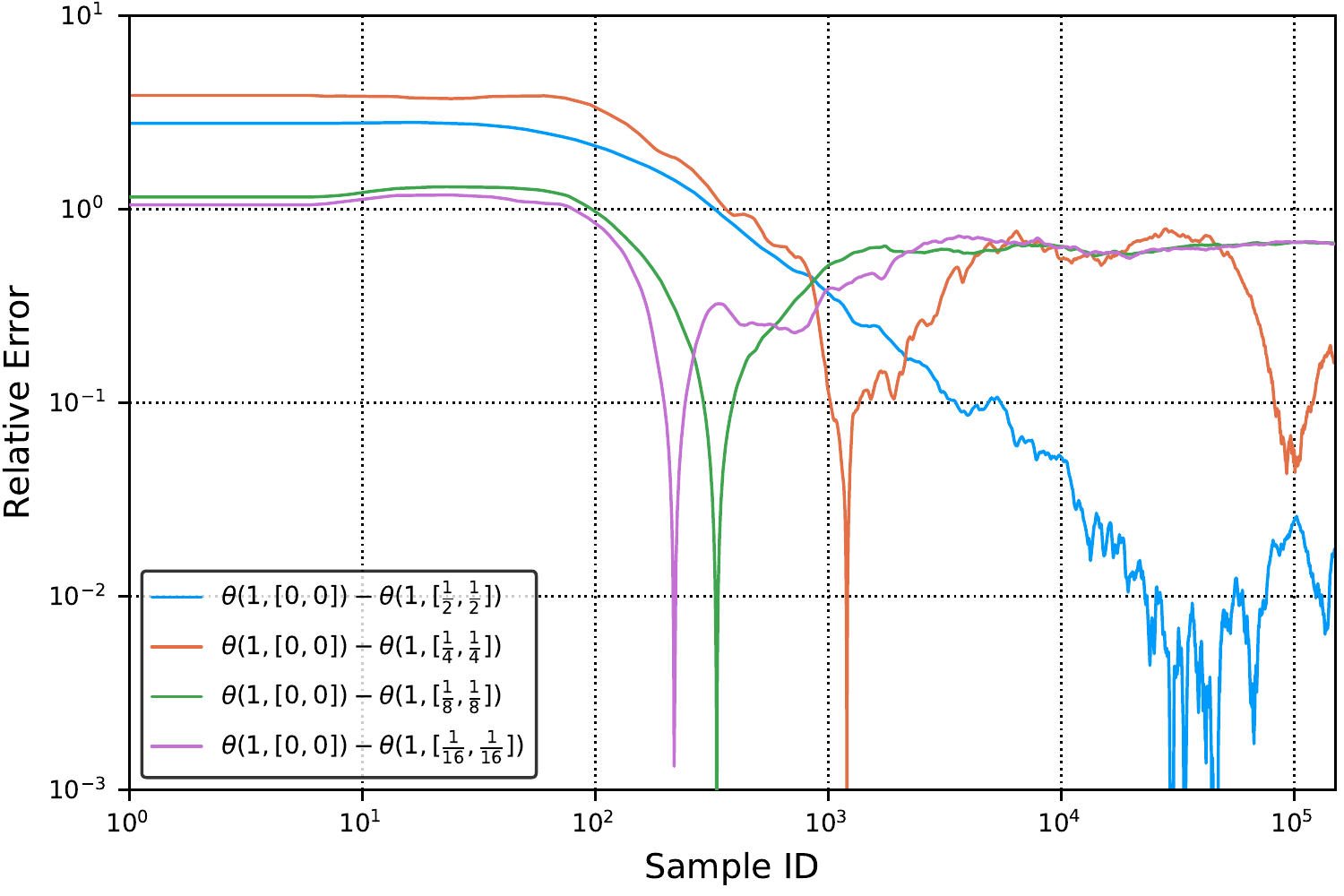}
  \end{minipage}
  \hfill
  \begin{minipage}[b]{0.49\textwidth}
  	\includegraphics[width=\textwidth]{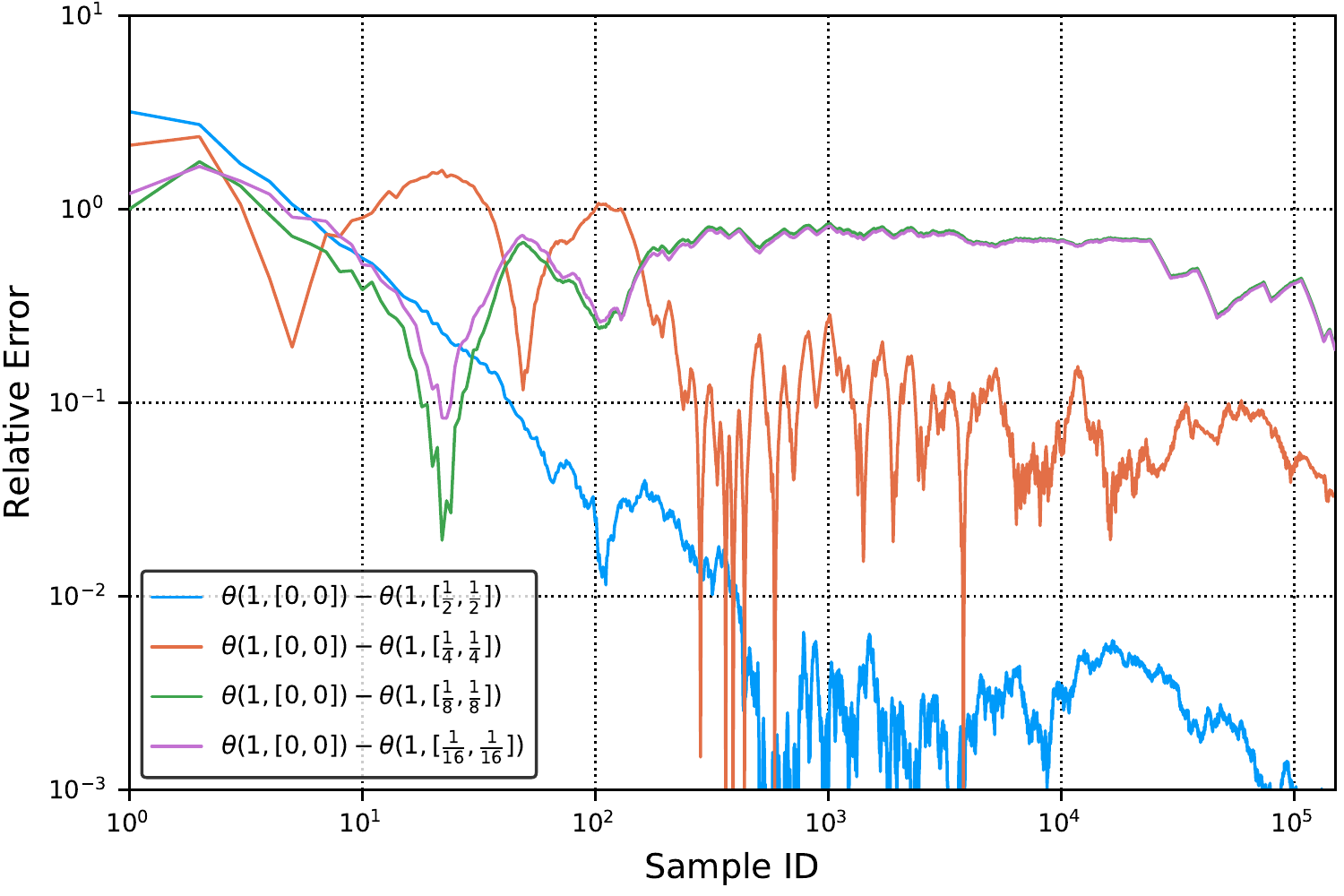}
  \end{minipage}
  \caption{\label{fig:ex2_scalarDiffEvolve}Relative error for the mean (cumulative moving average) of scalar differences at $t=1$ between the origin and $\left[\frac{1}{2},\frac{1}{2}\right]$, $\left[\frac{1}{4},\frac{1}{4}\right]$, $\left[\frac{1}{8},\frac{1}{8}\right]$, $\left[\frac{1}{16},\frac{1}{16}\right]$ for 100,000 samples, Example 2. Left: pCN, Right: HMC.}
\end{figure}

\section{A General Setting for Bayes' Theorem}\label{sec:bayes_general}

In this Appendix we consider an infinite dimensional setting for a
Bayesian Theorem applicable to a broad class of statistical inverse
problems that includes the problem considered in this paper.  Our presentation is
slightly more general than most treatments, e.g.~\cite{dashti2017bayesian};
namely we do not assume an additive noise structure in the
observational error or suppose that the prior distribution and
observation noise are independent.  

The problem at hand is to estimate an unknown parameter $\urv$ sitting in
a separable Hilbert space $\usp$ and subject to an observational noise
$\nrv$.  The forward model is given as
\begin{align}
	\yrv = \mathcal{F}(\urv, \nrv).
	\label{eq:forward}
\end{align}
Here $\Fd : \usp \times \nsp \to \ysp$ for possibly different
$N, M > 0$ and we assume that $\Fd$ is a Borel measurable map between
the given spaces.\footnote{We could consider the more general case
  when $\nsp$ and $\ysp$ are replaced by seperable Hilbert spaces in what follows.
  However, we avoid this additional complication for simplicity of
  presentation and since we are primarily interested situations involving a
  finite number of observations.} We treat $\urv$ and $\nrv$ as random
variables on an underlying probability space
$(\Omega, \mathcal{A}, \Prb)$.  The elements $\urv$ and $\nrv$ are
distributed as $\mu_0 \in Pr(\usp)$ and $\gamma_0 \in Pr(\nsp)$
respectively.  Note that we will not assume that $\urv$ and $\nrv$ are
statistically independent in general.  In the language of Bayesian statistical
inversion, $\mu_0$ is the prior distribution on our unknown parameter
$\urv$ and $\gamma_0$ is the distribution of the measurement noise
$\nrv$. Let $\lambda_0 \in Pr(\RR^N)$ denote the distribution of
$\yrv$.

We wish to rigorously define the conditional probabilities
$\yrv| \urv = \usm$ and $\urv | \yrv =\ysm$.  The former represents
the `likelihood of an observed data set $\yrv$ given $\urv = \usm$' while the
later is `the Bayesian posterior distribution for $\urv$ given
$\yrv = \ysm$'.  

For this purpose we recall some classical definitions around conditional 
expectations and probabilities from abstract
probability theory.  For further generalities germane to our 
discussions here, see e.g.~\cite{durrett2010probability, dudley2018real}.
\begin{Def}
\label{def:Cond:Exp:Precise:1}
Given a $\sigma$-algebra $\mathcal{H} \subseteq \mathcal{A}$ and a
random variable $Z$ the \emph{conditional expectation of $Z$ given
$\mathcal{H}$ denoted $\E(Z| \mathcal{H})$} is the unique (up
to a set of measure zero) random variable such that
\begin{align}
  \label{eq:cond:exp:0}
 \E(Z| \mathcal{H}) \text{ is measurable with respect to } \mathcal{H} 
\end{align}
and such that
\begin{align}
  \label{eq:cond:exp}
  \E(\E(Z| \mathcal{H})\indFn{A}) = \E( Z \indFn{A}) \quad
  \text{ for any } A \in \mathcal{H}.
\end{align}
Given another random variable $W$ we typically abuse notation and
write $\E(Z|W)$ for $\E(Z|\mathcal{H}_W)$ where $\mathcal{H}_W$
is the $\sigma$-algebra generated by $W$.  Futhermore,
we denote $\Prb(Z \in A| W) := \E( \indFn{Z \in A} | W)$.
\end{Def}
We next remind the reader of the definition of a regular conditional distribution as
\begin{Def}[Regular Conditional Distribution]
  Consider random variables $Z_1, Z_2$ taking values in the complete
  metric spaces $(X_1,d_1)$ and $(X_2,d_2)$, respectively.  We denote
  the Borel $\sigma$-algebra associated with $(X_2,d_2)$ by
  $\mathcal{B}_2$.  A \emph{regular conditional distribution} $\rcd$
  associated with $Z_1, Z_2$ is any function
  $\rcd: X_1 \times \mathcal{B}_2 \to [0,1]$ such that:
  \begin{itemize}
  \item[(i)] For every $z \in X_1$, $\rcd(z, \cdot)$ is a probability
    measure on $\mathcal{B}_2$ and, for every $A \in \mathcal{B}_2$,
    $\rcd(\cdot, A)$ is a Borel measurable function on $X_1$.
   \item[(ii)] For each $A \in \mathcal{B}_2$, 
      \begin{align}
        \rcd(Z_1(\omega), A) = \Prb(Z_2 \in A | Z_1)(\omega)
        \quad \text{ for almost every $\omega \in \Omega$}.
        \label{eq:Reg:C:D:1}
      \end{align}
  \end{itemize}
\end{Def}
We now define 
\begin{Def}
  \label{def:Cond:Exp:Precise:2}
  Relative to a given regular conditional distribution $\rcd$, we define the
  distribution $Z_1 | Z_2 = z_2$ rigorously as $\rcd(z_2, \cdot)$.
\end{Def}
To clarify these definitions, several remarks are in order.  
\begin{Rmk}
\mbox{}
\begin{itemize}
\item[(i)] For any two random variables $Z_1, Z_2$ an associated
  regular conditional distribution always exists; see
  \cite[Theorem 5.1.9]{durrett2010probability} for a construction.
\item[(ii)] In the case when $Z_1,Z_2$ take values in $\RR^n$ and are
  jointly, continuously distributed according to the probability density
  function $p$, then for any probability density function $g$,
  \begin{align}
    \label{eq:rcd:cont}
    \rcd(z, A) :=  
    \begin{cases}
      \frac{\int_A p(z,w) dw}{\int p(z,w) dw} & \text{ if } \int p(z,w) dw > 0,\\
      \int_A g(w) dw & \text{ otherwise},
    \end{cases}
  \end{align}
  defines a regular conditional distribution for $Z_1$, $Z_2$.
\item[(iii)] As illustrated by the previous example, regular conditional
  distributions are not unique in general as the choice of $g$ in 
  \eqref{eq:rcd:cont} was arbitrary.
\end{itemize}
\end{Rmk}

Let us now fix regular conditional distributions 
\begin{align}
	\QQ: \usp \times \mathcal{B}(\ysp) \to [0,1]
  \label{eq:y:u:rcd}
\end{align}
to define $\yrv| \urv = \usm$ and 
\begin{align}
	\mu: \ysp \times \mathcal{B}(\usp) \to [0,1]
  \label{eq:u:y:rcd}
\end{align}
to make sense of $\urv | \yrv =\ysm$. For more compact notation
below we will sometimes write $\QQ_\usm(\cdot) := \QQ(\usm, \cdot)$
and $\mu_{\ysm}(\cdot) := \mu(\ysm, \cdot)$.

While we now have a rigorous definition of $\mu_{\ysm}$ we would like to  
make sense of the usual Bayesian formulation
\begin{align*}
  \text{``posterior distribution $\propto$ likelihood(\ysm) $\times$ prior distribution"}
\end{align*}
in this general setting.  It turns out that all that is needed to
derive such a formula is the existence of a distribution $\gamma$ such that conditional probabilities $\QQ_\usm$, $\usm \in H$
are absolutely continuous with respect to $\gamma$.\footnote{Recall the a probability
  distribution $\rho$ is absolutely continuous with respect to another
  distribution $\tilde{\rho}$ if $\rho(A) = 0$ whenever
  $\tilde{\rho}(A) = 0$.  We typically denote this relationship by
  $\rho << \tilde{\rho}$.}
\begin{Prop}
	Assume there exists a distribution $\gamma \in Pr(\ysp)$ such that
	\begin{align}
          \QQ_{\usm} < < \gamma \quad \text{ for every } \usm \in \usp
          \label{eq:abs:cnt:cond}
	\end{align}
	and suppose that the resulting Radon-Nikodym derivative
        $\frac{d\QQ_{\usm}}{d \gamma}(\ysm)$ is measurable in $\usm$
        and $\ysm$.  Define
	\begin{align}
          Z(\ysm, \gamma) 
          := \int_{\usp} \frac{d\QQ_{\usm}}{d \gamma}(\ysm) \mu_0(d\usm).
          \label{eq:gen:norm}
	\end{align}
  Then, for any $\tilde{\mu} \in Pr(\usp)$,
	\begin{align}
          \label{eq:bayes:L:Pr}
          \mu_{\ysm} (d\usm) = \mu(\ysm, d\usm)
          :=
          \begin{cases}
            \frac{1}{ Z(\ysm, \gamma)}  
            \frac{d\QQ_{\usm}}{d \gamma}(\ysm) \mu_0(d\usm),
            &\text{ if }
            Z(\ysm, \gamma) > 0\\
            \tilde{\mu}(d\usm), &\text{ otherwise},
          \end{cases}
	\end{align}
        defines a regular condition distribution for $\urv | \yrv =\ysm$
        in the sense of \cref{def:Cond:Exp:Precise:2}.
  \label{thm:bayes_general}
\end{Prop}

Before turning to the proof we make the following simple observation
\begin{Lem}
For any bounded and measurable $\psi: \ysp \to \RR$, 
\begin{align}
	\E \psi(\yrv) 
  = \int_{\ysp} \psi(\ysm) \lambda_0(d\ysm) 
  = \int_{\usp} \int_{\ysp} \psi(\ysm) \QQ(\usm, d\ysm) \mu_0(d\usm),
  \label{eq:double:cond}
\end{align}
where $\lambda_0$ is the distribution of $\yrv$ defined by \eqref{eq:forward}
and $\QQ$ is the regular conditional distribution (\ref{eq:y:u:rcd}).

\end{Lem}
\begin{proof}
  Notice that 
  \begin{align*}
    \Prb( \yrv \in B) 
    =& \E \indFn{\yrv \in B}
    = \E(\E (\indFn{\yrv \in B}| \urv))
    = \E \QQ(\urv, B)
    = \int_{\usp} \QQ(\usm, B) \mu_0(d\usm)\\
    =& \int_{\usp} \int_{\ysp} \chi_B(\ysm)  \QQ(\usm, d\ysm) \mu_0(d\usm).
  \end{align*}
  Thus, by linearity we have shown \eqref{eq:double:cond} 
  for simple functions $\phi$.  We can now extend to the general
  case by a standard density argument.
\end{proof}
We turn now to 
\begin{proof}[Proof of \cref{thm:bayes_general}]
  We need to verify that $\mu_y$ given by \eqref{eq:bayes:L:Pr}
  satisfies the conditions for \cref{def:Cond:Exp:Precise:1}.
  The regularity properties in (i) are immediate from the given
  assumptions on $\frac{d\QQ_\usm}{d\gamma}$.  We
  verify (ii) by showing that, cf.~\eqref{eq:cond:exp},
  \begin{align}
    \E (\int_H \phi(\usm) \mu(\yrv, d\usm) \indFn{\yrv \in B})
    	= \E (\phi(\urv) \indFn{\yrv \in B})
    \label{eq:Bayes:ID:sf:cond}
  \end{align}
  for any $B \in \mathcal{B}(\ysp)$, and any bounded and measurable
  $\phi: H \to \RR$.  Using elementary properties of conditional
  expectations (see \cite[Chapter 5]{durrett2010probability}),
  recalling the definition of $\QQ$ in \eqref{eq:y:u:rcd}, and finally
  using \eqref{eq:abs:cnt:cond} we have
\begin{align*}
  \E (\phi(\urv) \indFn{\yrv \in B}) 
  =& \E (\E (\phi(\urv) \indFn{\yrv \in B}| \urv))
  = \E (\phi(\urv) \E ( \indFn{\yrv \in B}| \urv))
  = \E (\phi(\urv) \QQ(\urv, B))\\
  =& \int_{\usp} \phi(\usm) \QQ(\usm, B) \mu_0(d\usm)
     =\int_{\usp} \phi(\usm) 
     \int_{\ysp} \chi_B(\ysm) \QQ(\usm, d\ysm) \mu_0(d\usm)\\
  =& \int_{\usp} \phi(\usm) 
     \int_{\ysp} \chi_B(\ysm) \frac{d\QQ_{\usm}}{d\gamma}(\ysm) 
                             \gamma(d\ysm) \mu_0(d\usm)\\
  =& \int_{\ysp} \chi_B(\ysm)
     \int_{\usp} \phi(\usm)  \frac{d\QQ_{\usm}}{d\gamma}(\ysm) \mu_0(d\usm)
                             \gamma(d\ysm).
\end{align*}
On the other hand, we can show
\begin{align}
  Z(\ysm, \gamma) \int_{\usp} \phi(\usm) \mu(\ysm, d\usm)  
  =  \int_{\usp} \phi(\usm) \frac{d\QQ_{\usm}}{d \gamma}(\ysm) \mu_0(d\usm),
  \label{eq:bayes:norm:identity}
\end{align}
for any $y \in \ysp$.  When $Z(\ysm,\gamma)>0$, \eqref{eq:bayes:norm:identity} is true by definition of $\mu$ (see \eqref{eq:bayes:L:Pr}), and when $Z(\ysm,\gamma)=0$, we have $\frac{d\QQ_{\usm}}{d \gamma}(\ysm)=0$ $\mu_0$-almost surely (see \eqref{eq:gen:norm}), so both sides of \eqref{eq:bayes:norm:identity} are zero. Therefore, by combining the previous two identities
and recalling \eqref{eq:gen:norm} we find
\begin{align}
  \E (\phi(\urv) \indFn{\yrv \in B}) 
  =& \int_{\ysp} \chi_B(\ysm) Z(\ysm, \gamma)\int_{\usp} 
     \phi(\usm) \mu(y, d\usm)
                             \gamma(d\ysm)
  \notag\\
 =& \int_{\ysp} \int_{\usp}  \chi_B(\ysm) 
      \int_{\usp} \phi(\usm) \mu(y, d\usm)
    \frac{d\QQ_{\mathbf{u}}}{d \gamma}(\ysm) \mu_0(d \mathbf{u}) 
      \gamma(d\ysm)
  \notag\\
  =& \int_{\usp} \int_{\ysp} \chi_B(\ysm) \int_{\usp} \phi(\usm) \mu(y, d\usm)
     \QQ(\mathbf{u}, d\ysm) \mu_0(d\mathbf{u}).
\label{eq:final:bayes:ID}
\end{align}
Taking $\psi(y) = \chi_B(\ysm)\int_H \phi(\usm) \mu(y, d \usm)$ in
(\ref{eq:double:cond}) and combining this identity with
\eqref{eq:final:bayes:ID} now finally yields
\eqref{eq:Bayes:ID:sf:cond}, completing the proof.
\end{proof}

\section*{Acknowledgments}
This work was supported in part by the National Science Foundation
under grants DMS-1313272 (NEGH), DMS-1816551 (NEGH), DMS-1522616 (JTB), and DMS-1819110 (JTB); the National Institute for
Occupational Safety and Health under grant 200-2014-59669 (JTB); and the
Simons Foundation under grant 515990 (NEGH). We would also like to thank the
Mathematical Sciences Research Institute, the Tulane University Math
Department, and the International Centre for Mathematical Sciences
(ICMS) where significant portions of this project were developed and
carried out. The Society for Industrial and Applied Mathematics (SIAM)
and NSF grant DMS-1613965, NSF grant 1700124, and the Virginia Tech
Interdisciplinary Center for Applied Math (ICAM), respectively,
provided funds for travel to the 2016 Gene Golub SIAM Summer School,
2016 FOMICS Winter School on Uncertainty Quantification, and 2017 ICMS
workshop Probabilistic Perspectives in Nonlinear Partial Differential
Equations, where JK presented the preliminary results of this work and received
valuable feedback. We would also like to thank G. Didier, J. Foldes,
S. McKinley, C. Pop, G. Richards, G. Simpson, A. Stuart, and
J. Whitehead for many useful discussions and thoughtful feedback on
this work. 
The authors acknowledge Advanced Research Computing at Virginia Tech\footnote{\url{http://www.arc.vt.edu}} for providing computational resources and technical support that have contributed to the results reported within this paper.

\addcontentsline{toc}{section}{References}
\begin{footnotesize}
\bibliographystyle{plain}
\bibliography{references}
\end{footnotesize}

\vspace{1.5in}
\begin{multicols}{2}
\noindent
Jeff Borggaard\\
{\footnotesize Department of Mathematics\\
Virginia Tech\\
Web: \url{https://www.math.vt.edu/people/jborggaa/}\\
Email: \href{mailto:jborggaard@vt.edu}{\nolinkurl{jborggaard@vt.edu}}} \\[.5cm]
Nathan Glatt-Holtz\\ {\footnotesize
Department of Mathematics\\
Tulane University\\
Web: \url{http://www.math.tulane.edu/~negh/}\\
Email: \href{mailto:negh@tulane.edu}{\nolinkurl{negh@tulane.edu}}} \\[.2cm]

\columnbreak

 \noindent Justin Krometis\\
{\footnotesize
Advanced Research Computing\\
Virginia Tech\\
Web: \url{https://www.arc.vt.edu/justin-krometis/}\\
Email: \href{mailto:jkrometis@vt.edu}{\nolinkurl{jkrometis@vt.edu}}} \\[.2cm]
 \end{multicols}

\end{document}

%% file: hdr.tex
\usepackage{mathtools} 
\usepackage{xparse}
\usepackage{mathrsfs}

\newcommand{\R}{\mathbb{R}}           
\newcommand{\C}{\mathbb{C}}           
\newcommand{\Z}{\mathbb{Z}}           
\newcommand{\ddt}{\frac{d}{dt}}
\newcommand{\half}{\frac{1}{2}}
\newcommand{\recip}[1]{\frac{1}{{#1}}}

\newcommand{\Frechet}{Fr\'echet~}

\DeclarePairedDelimiterX\innerp[2]{\langle}{\rangle}{#1,#2}
\newcommand{\ip}[2]{ \innerp{#1}{#2} }
\newcommand{\ipZ}[3]{ \innerp{#1}{#2}_{#3} }

\newcommand{\norm}[1]{\left\lVert#1\right\rVert}

\NewDocumentCommand\Exp{g}{%
    \mathbb{E}\IfNoValueTF{#1}{}{\left[ #1 \right]}%
}

\NewDocumentCommand\Prob{g}{%
    \mathbb{P}\IfNoValueTF{#1}{}{\left[ #1 \right]}%
}